\documentclass[11pt,twoside,reqno]{amsart}

\RequirePackage{etex}

\usepackage[utf8]{inputenc}
\usepackage{graphicx}
\usepackage{tikz}

\usepackage{amscd}
\usetikzlibrary{matrix,arrows,decorations.pathmorphing}
\tikzset{my loop/.style =  {to path={
  \pgfextra{}
  [looseness=12,min distance=10mm]
  \tikz@to@curve@path},font=\sffamily\small
  }}  
\usepackage{tikz-cd}

\usepackage[all]{xy}
\usepackage{fullpage}
\usepackage{hyperref}
\usepackage{comment}

\hypersetup{
    colorlinks=true,
    linkcolor=blue,
    filecolor=magenta,      
    urlcolor=cyan,
}

\usepackage{etex}
\usepackage{amsmath, amssymb}
\usepackage{array}

\theoremstyle{plain}
\newtheorem{theorem}{Theorem}[section]

\newtheorem{lemma}[theorem]{Lemma}
\newtheorem{corollary}[theorem]{Corollary}
\newtheorem{proposition}[theorem]{Proposition}

\theoremstyle{definition}

\newtheorem{remark}[theorem]{Remark}

\newtheorem{example}[theorem]{Example}
\newtheorem{definition}[theorem]{Definition}

\newcommand{\bK}{\mathbf{K}}
\newcommand{\cB}{\mathbb{B}}

\DeclareMathOperator{\SL}{SL}

\DeclareMathOperator{\Spec}{Spec}

\DeclareMathOperator{\Pic}{Pic}
\DeclareMathOperator{\NE}{NE}
\DeclareMathOperator{\codim}{codim}
\DeclareMathOperator{\tor}{tor}

\DeclareMathOperator{\GIT}{GIT}
\DeclareMathOperator{\Ext}{Ext}
\DeclareMathOperator{\bb}{\mathrm{bb}}
\DeclareMathOperator{\In}{\mathrm{In}}

\title[]{
Geometric interpretation of toroidal compactifications of moduli of points in the line and cubic surfaces
}

\author{Patricio Gallardo}
\address{
{\small Department of Mathematics,
900 University Ave.
Riverside, CA 92521
Skye Hall}
}
\email{pgallard@ucr.edu}

\author{Matt Kerr}
\address{{\small Department of Mathematics and Statistics
Washington University in St. Louis
Campus Box 1146
One Brookings Drive}}
\email{matkerr@wustl.edu}

\author{Luca Schaffler}
\address{{\small Department of Mathematics \& Statistics, University of Massachusetts Amherst, Amherst, MA 01003, USA}}
\email{schaffler@math.umass.edu}


\begin{document}

\begin{abstract}
It is known that some GIT compactifications associated to moduli spaces of either points in the projective line or cubic surfaces are isomorphic to Baily-Borel compactifications of appropriate ball quotients. In this paper, we show that their respective toroidal compactifications are isomorphic to moduli spaces of stable pairs as defined in the context of the MMP. Moreover, we give a precise mixed-Hodge-theoretic interpretation of this isomorphism for the case of eight labeled points in the projective line.
\end{abstract}

\bibliographystyle{alpha}
\maketitle


\section{Introduction}

Understanding the interplay between geometric and Hodge-theoretic compactifications of a given moduli space is one of the main challenges in algebraic geometry. This type of problem was investigated for moduli spaces of different types of algebraic varieties: for instance, abelian varieties \cite{Ale02}, cubic threefolds \cite{CMGHL15}, and very recently degree $2$ K3 surfaces \cite{AET19}. In the current paper we describe in detail such interplay for certain moduli spaces of points in $\mathbb{P}^1$ studied by Deligne and Mostow, and the moduli space of cubic surfaces. More specifically, we expand Deligne-Mostow's results by showing that the toroidal compactifications of their ball quotients are isomorphic to appropriate Hassett's moduli spaces of weighted stable rational curves. We pursue this isomorphism in detail for the case of eight labeled points in $\mathbb{P}^1$ by explicitly describing the Hodge theoretic boundary components. In another direction, as a consequence of our results, it follows that the KSBA compactification in \cite{MS18} of a certain family of K3 surfaces arising from eight points in the projective line is isomorphic to a toroidal compactification up to a finite group action. For cubic surfaces, we prove that Naruki's compactification is toroidal, and that it has a modular interpretation in terms of Koll\'ar-Shepherd--Barron-Alexeev stable pairs. Some of these results were expected (see \cite[\S1]{ACT02}, \cite[Remark 1.3 (4)]{HKT09}, and \cite[Appendix C]{CMGHL19}). However, their proofs were not pursued so far.
 

\subsection{Toroidal compactifications of Deligne-Mostow ball quotients}\label{s1.1}
In \cite{DM86}, Deligne and Mostow described certain moduli spaces of 
$n$ labeled points in $\mathbb{P}^1$ such that the
GIT compactification 
$
(\mathbb{P}^1)^n/ \! /_{\mathbf{w}} \SL_2
$
with respect to a specific linearization
$\mathbf{w}=(w_1, \ldots, w_n)$
is isomorphic, after possibly quotienting by the action on the labels of an appropriate symmetric group $S_m$, to the Baily-Borel compactification $\overline{\Gamma_{\mathbf{w}}\backslash\mathbb{B}_{n-3}}^{\bb}$ of a ball quotient, where $\Gamma_{\mathbf{w}}$ is arithmetic.
The general theory developed in \cite{AMRT75} provides us with an alternative compactification of the moduli space of $n$ points in $\mathbb{P}^1$: the unique toroidal compactification $\overline{\Gamma_{\mathbf{w}}\backslash\mathbb{B}_{n-3}}^{\tor}$, which is a blow up of the Baily-Borel compactification 
$\overline{\Gamma_{\mathbf{w}}\backslash\mathbb{B}_{n-3}}^{\bb}$ at the cusps. This toroidal compactification has divisorial boundary, with each boundary component being the quotient of a CM-abelian variety by a (possibly trivial) finite group. In general, toroidal compactifications have milder singularities, but they lack a geometrically modular interpretation.

On the other hand, a geometric compactification of the moduli space of $n$ points in $\mathbb{P}^1$ mapping birationally onto $(\mathbb{P}^1)^n/ \! /_{\mathbf{w}} \SL_2$ is provided by the Hassett weighted moduli space $\overline{\mathrm{M}}_{0,\mathbf{w}+\epsilon}$ of stable $n$-pointed rational curves with weights
$\mathbf{w}+\epsilon:=(w_1+\epsilon, \ldots, w_n+\epsilon)$ \cite{Has03}. It is very natural to ask how $\overline{\mathrm{M}}_{0,\mathbf{w}+\epsilon}$ and $\overline{\Gamma_{\mathbf{w}}\backslash\mathbb{B}_{n-3}}^{\tor}$ are related, which leads to our main result.

\begin{theorem}
\label{maintheorem}
Let $\mathbf{w}$ be a set of $n$ rational weights and $m$ a nonnegative integer for which we have the Deligne-Mostow isomorphism 
$$
\overline{\Gamma_{\mathbf{w}}\backslash\mathbb{B}_{n-3}}^{\bb}
\cong
(\mathbb{P}^1)^n/ \! /_{\mathbf{w}}\SL_2\rtimes S_m.
$$ 
\textup{(}For a complete list of these cases see Tables  ~\ref{table:DMcasesEisenstein} and 
\ref{table:DMcasesGaussian}; all satisfy $\sum w_i =2$.\textup{)} Then the toroidal compactification of 
$\overline{\Gamma_{\mathbf{w}}\backslash\mathbb{B}_{n-3}}^{\tor}$ is isomorphic to the quotient by $S_m$ of the Hassett moduli space of $n$-pointed rational curves with weights $\mathbf{w}+\epsilon$. In particular, we have the following commutative diagram:
\begin{center}
\begin{tikzpicture}[>=angle 90]
\matrix(a)[matrix of math nodes,
row sep=2em, column sep=2em,
text height=1.5ex, text depth=0.25ex]
{
\overline{\mathrm{M}}_{0, \mathbf{w + \epsilon}}/S_m
&
\overline{\Gamma_{\mathbf{w}}\backslash\mathbb{B}_{n-3}}^{\tor}
\\
(\mathbb{P}^1)^n/ \! /_{\mathbf{w}}\SL_2\rtimes S_m
&
\overline{\Gamma_{\mathbf{w}}\backslash\mathbb{B}_{n-3}}^{\bb},\\};
\path[->] (a-1-1) edge node[above]
{ $\overline{\Phi}_{\mathbf{w}}$ }(a-1-2); 
\path[->] (a-1-1) edge node[left]{}(a-2-1);
\path[->] (a-2-1) edge node[above]{}(a-2-2);
\path[->] (a-1-2) edge node[right]{$\varphi$}(a-2-2);
\end{tikzpicture}
\end{center}
where the horizontal arrows are isomorphisms. 
\end{theorem}

The main idea  for the  proof of Theorem~\ref{maintheorem}, see Section \ref{proofofthemaintheorem}, is to first prove the result for two special cases called the ancestral ones. Then we extend our result to the remaining Deligne-Mostow cases using work of Doran \cite{Dor04a}, which guarantees that the remaining ball quotients can be viewed as sub-ball quotients of the two ancestral cases.  Note that the common denominator $d$ of $\mathbf{w}$ is $3$, $4$, or $6$ in all cases, and when $d=4$ we always have $m=1$ (that is, $S_m$ is trivial).

The isomorphisms in Theorem \ref{maintheorem} compactify period maps $\Phi_{\mathbf{w}}\colon \mathrm{M}_{0,n}/S_m \hookrightarrow \Gamma_{\mathbf{w}}\backslash \mathbb{B}_{n-3}$ associated with certain weight-1 variations of Hodge structure $\widetilde{\mathcal{V}}_{\mathbf{w}}$ on $\mathrm{M}_{0,n}/S_m$ with monodromy group $\Gamma_{\mathbf{w}}$ and Hodge numbers $(n-2,n-2)$.  Namely, the fiber of $\widetilde{\mathcal{V}}_{\mathbf{w}}$ over (the $S_m$-orbit of)  $\mathbf{x}=(x_1,\ldots,x_n)\in \mathrm{M}_{0,n}$ is the subspace in $H^1$ of the curve $C_{\mathbf{w},\mathbf{x}}:=\{Y^d=\Pi_{j=1}^n (X-x_j Z)^{dw_j}\}\subset\mathbb{WP}[1:1:2]$ on which the automorphism $\rho^*$ defined by $Y\mapsto e^{\frac{2\pi i}{d}}Y$ acts through $e^{\pm\frac{2\pi i}{d}}$, cf. \cite[Thm. 8.4ff]{DK07}.  The injectivity of $\Phi_{\mathbf{w}}$ is thus transformed into a global Torelli theorem, and one wonders to what extent $\overline{\Phi}_{\mathbf{w}}$ underlies an \emph{extended} global Torelli theorem matching geometric and Hodge-theoretic moduli.

In Section \ref{sec:Matt}, we work this out for the case $\mathbf{w}=\mathbf{\tfrac{1}{4}}=(\tfrac{1}{4},\ldots,\tfrac{1}{4})$ ($n=8$, $m=1$), where the VHS has fibers $\widetilde{\mathcal{V}}_{\mathbf{\frac{1}{4}},\mathbf{x}}=H^1(C_{\mathbf{w},\mathbf{x}})^{-\rho^2}$.  To wit, we provide a mixed-Hodge-theoretic interpretation of the restriction of $\overline{\Phi}=\overline{\Phi}_{\mathbf{\frac{1}{4}}}$ to the exceptional divisor of $\varphi$, along which (generically) $C_{\mathbf{\frac{1}{4}},\mathbf{x}}$ degenerates to a pair of genus-3 curves.  That is, each component of this divisor has a Zariski open $\mathcal{S}$ parametrizing curves $\{D_{\nu}=D_{\nu}^{(1)}\cup D_{\nu}^{(2)}\}_{\nu\in\mathcal{S}}$ on which $\rho$ acts, with four fixed points on each $D_{\nu}^{(j)}$, and $D_{\nu}^{(1)}\cap D_{\nu}^{(2)}=\{q_{\ell}=\rho^{\ell}(q_0)\}_{\ell=0}^3$.  Write $N$ for the monodromy logarithm of $\widetilde{\mathcal{V}}_{\mathbf{\frac{1}{4}}}$ along $\mathcal{S}$, and $\mathcal{R}$ for the (finite) monodromy group of $H^1(D_{\nu}^{(j)},\mathbb{Z})^{-\rho^2}$ on $\mathcal{S}$.  We show in Propositions \ref{prop:Matt3} and \ref{prop:Matt4} that $\overline{\Phi}|_{\mathcal{S}}$ records the limiting mixed Hodge structure of $\widetilde{\mathcal{V}}$ at $\nu$ in the Hodge-theoretic boundary component $\Gamma_N\backslash B(N)\cong \times_{j=1}^2 \{\mathcal{R}\backslash (\Omega^1(D^{(j)}_{\nu})^{-\rho^2})^{\vee}/(1-\rho)H_1(D_{\nu}^{(j)},\mathbb{Z})\}\cong \mathbb{P}^2 \times\mathbb{P}^2$ (which is independent of $\nu\in\mathcal{S}$ despite appearances).  This leads to the following result, proved in Section \ref{sec:Matt5}:
\begin{theorem}\label{thm:Matt}
Given $\nu\in\mathcal{S}$, let $\sigma_+^{(j)}$ be any path on $D_{\nu}^{(j)}$ ($j=1,2$) from a fixed point of $\rho$ to a point of $D_{\nu}^{(1)}\cap D_{\nu}^{(2)}$, and $\sigma^{(j)}:=\sigma_+^{(j)}-\rho^2\sigma_+^{(j)}$.  Then the functional $(\int_{\sigma^{(1)}},\int_{\sigma^{(2)}})\in \oplus_{j=1}^2 (\Omega^1(D_{\nu}^{(j)})^{-\rho^2})^{\vee}$ becomes well-defined in $\Gamma_N\backslash B(N)$ and computes $\overline{\Phi}(\nu)$.
\end{theorem}


\subsection{Moduli of K3 surfaces from eight points in the projective line}
 The above ball quotients can also be identified with specific moduli spaces of surfaces (see \cite{DK07,Dor04b,DvGK05,Kon07,Moo18}). For instance, if $\mathbf{w}=\mathbf{\frac{1}{4}}$ as above, an open subset of the ball quotient $\Gamma_{\mathbf{\frac{1}{4}}} \backslash \cB_5$ parametrizes the K3 surfaces with order-four, purely non-symplectic automorphism and $U(2)\oplus D_4^{\oplus2}$ lattice polarization. These surfaces were studied in \cite{Kon07} from the point of view of automorphic forms, and they arise as the minimal resolution of the double cover $X\rightarrow\mathbb{P}^1\times\mathbb{P}^1$ branched along a specific curve of class $(4,4)$ which depends on the choice of eight points in $\mathbb{P}^1$. More precisely, if  
$[\lambda_1:1],\ldots,[\lambda_8:1]$ are the eight distinct points, then the equation of the branch curve is in the following form:
\[
y_0y_1 \left( y_0^2 \prod_{i=1}^4(x_0-\lambda_ix_1) + y_1^2 \prod_{i=5}^8(x_0-\lambda_ix_1)\right)=0.
\]
The involution of $\mathbb{P}^1\times\mathbb{P}^1$ given in an affine patch by $(x,y)\mapsto(x,-y)$ lifts to the K3 surface giving the order 4 purely non-symplectic automorphism. A compactification $\overline{\mathbf{K}}$ of such a family by KSBA stable pairs (see \cite{KSB88,Ale96,Kol18}) was studied in \cite{MS18}, where it is shown $\overline{\mathbf{K}}$ is isomorphic to the quotient of $\overline{\mathrm{M}}_{0,\mathbf{\frac{1}{4}}+\epsilon}$ by a finite group. Therefore, as an immediate consequence of Theorem~\ref{maintheorem}, we have the following result.

\begin{corollary}
\label{cor:KSBAToroK3}
Let $\overline{\bK}$ be the KSBA compactification  of the moduli space of K3 surfaces with a purely non-symplectic automorphism of order four and $U(2)\oplus D_4^{\oplus2}$ lattice polarization. Then, $\overline{\bK}$ is isomorphic to the quotient of the toroidal compactification $\overline{
\Gamma_{\mathbf{\frac{1}{4}}}
\backslash
\mathbb{B}_5}^{\tor}$ by $(S_4\times S_4)\rtimes S_2$.
\end{corollary}


\subsection{Compactifications of moduli of cubic surfaces}
Next, we discuss our work on cubic surfaces. Let $\mathbf{Y}$ be the moduli space parametrizing marked smooth cubic surfaces in \cite[\S6.3]{HKT09} (we recall the definition of marking in \S\ref{backgroundmoduliofmarkedcubicsurfacesandcompactifications}). This moduli space can be compactified using different techniques. A smooth normal crossing compactification $\mathbf{Y}\subseteq\overline{\mathbf{N}}$ was constructed by Naruki \cite{Nar82} using cross-ratios of tritangents to the cubic surfaces. Another perspective comes from GIT: an appropriate $W(E_6)$ cover of the GIT quotient of cubic surfaces (for details see Definition~\ref{GITcompndefandfacts}) provides a compactification $\mathbf{Y}\subseteq\overline{\mathbf{Y}}_{\GIT}$, which is related to Naruki's compactification by a birational morphism $\overline{\mathbf{N}}\rightarrow\overline{\mathbf{Y}}_{\GIT}$. By \cite[Theorem 3.17]{ACT02} we also know that there is an isomorphism between $\overline{\mathbf{Y}}_{\GIT}$ and the Baily-Borel compactification $\overline{\Gamma_c\backslash\cB_4}^{\bb}$ of an appropriate ball quotient (see also \cite{DvGK05}). In \cite{ACT02}, Allcock, Carlson, and Toledo assert the following result without pursuing its proof (see the introduction in \cite{ACT02}). To complete the picture, in Section~\ref{Narukicompactificationistoroidal} we give a proof of this expected isomorphism.

\begin{theorem}
\label{thm:maincubics}
The Naruki compactification $\overline{\mathbf{N}}$ of the moduli space $\mathbf{Y}$ of marked cubic surfaces is isomorphic to the toroidal compactification $\overline{\Gamma_c\backslash\cB_4}^{\tor}$ of its ball quotient. In particular, we have the following commutative diagram:
\begin{center}
\begin{tikzpicture}[>=angle 90]
\matrix(a)[matrix of math nodes,
row sep=2em, column sep=2em,
text height=2.0ex, text depth=0.25ex]
{\overline{\mathbf{N}}&
\overline{\Gamma_c\backslash\cB_4}^{\tor}
\\
\overline{\mathbf{Y}}_{\GIT}&
\overline{\Gamma_c\backslash\cB_4}^{\bb}
.\\};
\path[->] (a-1-1) edge node[above]{$\cong$}(a-1-2);
\path[->] (a-1-1) edge node[]{}(a-2-1);
\path[->] (a-2-1) edge node[above]{$\cong$}(a-2-2);
\path[->] (a-1-2) edge node[]{}(a-2-2);
\end{tikzpicture}
\end{center}
\end{theorem}

Another compactification of $\mathbf{Y}$ can be constructed using KSBA stable pairs. The marking of a smooth cubic surface $S$ induces a labeling of its $(-1)$-curves, and we denote by $B$ the divisor on $S$ given by their sum (note that $(S,B)$ is a stable pair if $B$ has normal crossings). In \cite{HKT09}, Hacking, Keel, and Tevelev studied the KSBA compactification of the moduli space parametrizing such stable pairs $(S,B)$. The authors also showed that Naruki's compactification $\overline{\mathbf{N}}$ is isomorphic to the log canonical model of $\mathbf{Y}$, and they believed that $\overline{\mathbf{N}}$ should parametrize stable pairs $\left(S,\left(\frac{1}{9}+\epsilon\right)B\right)$ and their degenerations, where $\epsilon\in\mathbb{Q}$, $0<\epsilon\ll1$. In the next theorem we confirm this belief. As a result, we give a modular interpretation of the toroidal compactification $\overline{\Gamma_c\backslash\mathbb{B}_4}^{\tor}$.

\begin{theorem}
\label{thm:NarukiModular}
The Naruki compactification $\overline{\mathbf{N}}$ is isomorphic to the normalization of the KSBA compactification $\overline{\mathbf{Y}}_{\frac{1}{9}+\epsilon}$ of the moduli space parametrizing marked smooth cubic surfaces with divisor given by the sum of the $27$ lines with weight $\frac{1}{9}+\epsilon$ and their degenerations. The stable pairs parametrized by the boundary of $\overline{\mathbf{N}}$ are in the form $\left(S_0,\left(\frac{1}{9}+\epsilon\right)B_0\right)$, where $S_0\subseteq\mathbb{P}^3$ is a singular cubic surface and $B_0$ is the sum of the lines in $S_0$, which we further describe as follows (see the corresponding picture in Table~\ref{ksbastabledegenerations27linesweight1/9+e}):
\begin{itemize}

\item $S_0$ has exactly one $A_1$ singularity and $B_0$ has six double lines passing through this singularity. The remaining $15$ lines have multiplicity one.

\item $S_0$ has exactly two $A_1$ singularities. $B_0$ has one quadruple line passing through the two singular points, four double lines passing through one $A_1$ singularity, and other four double lines passing through the other $A_1$ singularity. The remaining seven lines have multiplicity one.

\item $S_0$ has exactly three $A_1$ singularities. $B_0$ has three quadruple lines, each one containing a pair of $A_1$ singularities, and three pairs of double lines, where each pair passes through one of the three singular points. The remaining three lines have multiplicity one.

\item $S_0$ has exactly four $A_1$ singularities (this is known as the Cayley cubic surface). $B_0$ has six quadruple lines, each one containing a pair of $A_1$ singularities. The remaining three lines have multiplicity one.

\end{itemize}
In the remaining cases, $S_0$ is equal to the normal crossing union of three planes in $\mathbb{P}^3$, each one containing exactly nine of the lines of $B_0$, possibly with multiplicities. For the precise multiplicities and the incidences of the lines we refer to the pictures in the second column of Table~\ref{ksbastabledegenerations27linesweight1/9+e}.
\end{theorem}


\begin{table}[ht!]
\centering
\caption{
Pairs parametrized by Naruki's compactification --
see Theorem~\ref{thm:NarukiModular}. Continuous lines have multiplicity $1$, dashed lines have multiplicity $2$, and dotted lines have multiplicity $4$. In the pictures on the left we do not draw lines with multiplicity $1$ intentionally.
}
\label{ksbastabledegenerations27linesweight1/9+e}
\renewcommand{\arraystretch}{1.4}
\begin{tabular}{|>{\centering\arraybackslash}m{7.8cm}|>{\centering\arraybackslash}m{7.8cm}|}
\hline
$A_1$
&
$N$
\\
\hline
\begin{tikzpicture}[scale=0.5]

	\draw[dashed,line width=.5pt] (0,0) -- (-1.4,.5);
	\draw[dashed,line width=.5pt] (0,0) -- (-1,1);
	\draw[dashed,line width=.5pt] (0,0) -- (-.7,1.5);
	\draw[dashed,line width=.5pt] (0,0) -- (1.4,.5);
	\draw[dashed,line width=.5pt] (0,0) -- (1,1);
	\draw[dashed,line width=.5pt] (0,0) -- (.7,1.5);

	\fill (0,0) circle (5pt);

	\node at (0,-0.7) {$A_1$};

\end{tikzpicture}
&
\begin{tikzpicture}[scale=0.5]

	\draw[line width=1pt] (0,0) -- (0,4);
	\draw[line width=1pt] (0,0) -- (-4,-3);
	\draw[line width=1pt] (0,0) -- (4,-3);

	\draw[line width=.5pt] (0,1) -- (1,-3/4);
	\draw[line width=.5pt] (0,1) -- (2,-3/2);
	\draw[line width=.5pt] (0,1) -- (3,-9/4);
	\draw[line width=.5pt] (0,2) -- (1,-3/4);
	\draw[line width=.5pt] (0,2) -- (2,-3/2);
	\draw[line width=.5pt] (0,2) -- (3,-9/4);
	\draw[line width=.5pt] (0,3) -- (1,-3/4);
	\draw[line width=.5pt] (0,3) -- (2,-3/2);
	\draw[line width=.5pt] (0,3) -- (3,-9/4);

	\draw[line width=.5pt] (0,1) -- (-1,-3/4);
	\draw[line width=.5pt] (0,1) -- (-2,-3/2);
	\draw[line width=.5pt] (0,1) -- (-3,-9/4);
	\draw[line width=.5pt] (0,2) -- (-1,-3/4);
	\draw[line width=.5pt] (0,2) -- (-2,-3/2);
	\draw[line width=.5pt] (0,2) -- (-3,-9/4);
	\draw[line width=.5pt] (0,3) -- (-1,-3/4);
	\draw[line width=.5pt] (0,3) -- (-2,-3/2);
	\draw[line width=.5pt] (0,3) -- (-3,-9/4);

	\draw[line width=.5pt] (-1,-3/4) -- (1,-3/4);
	\draw[line width=.5pt] (-1,-3/4) -- (2,-3/2);
	\draw[line width=.5pt] (-1,-3/4) -- (3,-9/4);
	\draw[line width=.5pt] (-2,-3/2) -- (1,-3/4);
	\draw[line width=.5pt] (-2,-3/2) -- (2,-3/2);
	\draw[line width=.5pt] (-2,-3/2) -- (3,-9/4);
	\draw[line width=.5pt] (-3,-9/4) -- (1,-3/4);
	\draw[line width=.5pt] (-3,-9/4) -- (2,-3/2);
	\draw[line width=.5pt] (-3,-9/4) -- (3,-9/4);

\end{tikzpicture}
\\
\hline
\hline
$A_1^2$
&
$(A_1,N)$
\\
\hline
\begin{tikzpicture}[scale=0.5]

	\draw[dotted,line width=.5pt] (-3,0) -- (3,0);

	\draw[dashed,line width=.5pt] (0+1.5,0) -- (-1+1.5,1);
	\draw[dashed,line width=.5pt] (0+1.5,0) -- (-.7+1.5,1.5);
	\draw[dashed,line width=.5pt] (0+1.5,0) -- (1+1.5,1);
	\draw[dashed,line width=.5pt] (0+1.5,0) -- (.7+1.5,1.5);

	\draw[dashed,line width=.5pt] (0-1.5,0) -- (-1-1.5,1);
	\draw[dashed,line width=.5pt] (0-1.5,0) -- (-.7-1.5,1.5);
	\draw[dashed,line width=.5pt] (0-1.5,0) -- (1-1.5,1);
	\draw[dashed,line width=.5pt] (0-1.5,0) -- (.7-1.5,1.5);

	\fill (1.5,0) circle (5pt);
	\fill (-1.5,0) circle (5pt);

	\node at (1.5,-0.7) {$A_1$};
	\node at (-1.5,-0.7) {$A_1$};

\end{tikzpicture}
&
\begin{tikzpicture}[scale=0.5]

	\draw[line width=1pt] (0,0) -- (0,4);
	\draw[line width=1pt] (0,0) -- (-4,-3);
	\draw[line width=1pt] (0,0) -- (4,-3);

	\draw[line width=.5pt] (0,1) -- (1,-3/4);
	\draw[line width=.5pt] (0,1) -- (2,-3/2);
	\draw[line width=.5pt] (0,1) -- (3,-9/4);
	\draw[line width=.5pt] (0,2) -- (1,-3/4);
	\draw[line width=.5pt] (0,2) -- (2,-3/2);
	\draw[line width=.5pt] (0,2) -- (3,-9/4);
	\draw[line width=.5pt] (0,3) -- (1,-3/4);
	\draw[line width=.5pt] (0,3) -- (2,-3/2);
	\draw[line width=.5pt] (0,3) -- (3,-9/4);

	\draw[line width=.5pt] (0,1) -- (-1,-3/4);
	\draw[dashed,line width=.5pt] (0,1) -- (-3,-9/4);
	\draw[line width=.5pt] (0,2) -- (-1,-3/4);
	\draw[dashed,line width=.5pt] (0,2) -- (-3,-9/4);
	\draw[line width=.5pt] (0,3) -- (-1,-3/4);
	\draw[dashed,line width=.5pt] (0,3) -- (-3,-9/4);

	\draw[line width=.5pt] (-1,-3/4) -- (1,-3/4);
	\draw[line width=.5pt] (-1,-3/4) -- (2,-3/2);
	\draw[line width=.5pt] (-1,-3/4) -- (3,-9/4);
	\draw[dashed,line width=.5pt] (-3,-9/4) -- (1,-3/4);
	\draw[dashed,line width=.5pt] (-3,-9/4) -- (2,-3/2);
	\draw[dashed,line width=.5pt] (-3,-9/4) -- (3,-9/4);

\end{tikzpicture}
\\
\hline
\hline
$A_1^3$
&
$(A_1^2,N)$
\\
\hline
\begin{tikzpicture}[scale=0.5]

	\draw[dotted,line width=.5pt] (0,4) -- (-2,0);
	\draw[dotted,line width=.5pt] (0,4) -- (2,0);
	\draw[dotted,line width=.5pt] (-2,0) -- (2,0);

	\draw[dashed,line width=.5pt] (0,4) -- (-.3,2.5);
	\draw[dashed,line width=.5pt] (0,4) -- (.3,2.5);

	\draw[dashed,line width=.5pt] (-2,0) -- (-.7,.5);
	\draw[dashed,line width=.5pt] (-2,0) -- (-1,1);

	\draw[dashed,line width=.5pt] (2,0) -- (.7,.5);
	\draw[dashed,line width=.5pt] (2,0) -- (1,1);

	\fill (-2,0) circle (5pt);
	\fill (2,0) circle (5pt);
	\fill (0,4) circle (5pt);

	\node at (-2.7,0) {$A_1$};
	\node at (2.7,0) {$A_1$};
	\node at (0,4.5) {$A_1$};

\end{tikzpicture}
&
\begin{tikzpicture}[scale=0.5]

	\draw[line width=1pt] (0,0) -- (0,4);
	\draw[line width=1pt] (0,0) -- (-4,-3);
	\draw[line width=1pt] (0,0) -- (4,-3);

	\draw[line width=.5pt] (0,1) -- (1,-3/4);
	\draw[dashed,line width=.5pt] (0,1) -- (3,-9/4);
	\draw[line width=.5pt] (0,2) -- (1,-3/4);
	\draw[dashed,line width=.5pt] (0,2) -- (3,-9/4);
	\draw[line width=.5pt] (0,3) -- (1,-3/4);
	\draw[dashed,line width=.5pt] (0,3) -- (3,-9/4);

	\draw[line width=.5pt] (0,1) -- (-1,-3/4);
	\draw[dashed,line width=.5pt] (0,1) -- (-3,-9/4);
	\draw[line width=.5pt] (0,2) -- (-1,-3/4);
	\draw[dashed,line width=.5pt] (0,2) -- (-3,-9/4);
	\draw[line width=.5pt] (0,3) -- (-1,-3/4);
	\draw[dashed,line width=.5pt] (0,3) -- (-3,-9/4);

	\draw[line width=.5pt] (-1,-3/4) -- (1,-3/4);
	\draw[dashed,line width=.5pt] (-1,-3/4) -- (3,-9/4);
	\draw[dashed,line width=.5pt] (-3,-9/4) -- (1,-3/4);
	\draw[dotted,line width=.5pt] (-3,-9/4) -- (3,-9/4);

\end{tikzpicture}
\\
\hline
\hline
$A_1^4$
&
$(A_1^3,N)$
\\
\hline
\begin{tikzpicture}[scale=0.5]

	\draw[dotted,line width=0.5 pt] (-3,0) -- (3,0);
	\draw[dotted,line width=0.5 pt] (-3,3) -- (3,3);	
	\draw[dotted,line width=0.5 pt] (-2,-1) -- (-2,4);
	\draw[dotted,line width=0.5 pt] (2,-1) -- (2,4);	
	\draw[dotted,line width=0.5 pt] (-2,0) -- (2,3);	
	\draw[dotted,line width=0.5 pt] (2,0) -- (-2,3);	
	
	\fill (-2,0) circle (5pt);
	\fill (2,0) circle (5pt);
	\fill (-2,3) circle (5pt);
	\fill (2,3) circle (5pt);

	\node at (-2.7,-0.5) {$A_1$};
	\node at (2.7,-0.5) {$A_1$};
	\node at (2.7,3.5) {$A_1$};
	\node at (-2.7,3.5) {$A_1$};

\end{tikzpicture}
&
\begin{tikzpicture}[scale=0.5]

	\draw[line width=.5pt] (-1,-3/4) -- (0,1);
	\draw[line width=1pt] (0,0) -- (0,4);
	\draw[line width=1pt] (0,0) -- (-4,-3);
	\draw[line width=1pt] (0,0) -- (4,-3);

	\draw[line width=.5pt] (1,-3/4) -- (0,1);
	\draw[dashed,line width=.5pt] (0,1) -- (3,-9/4);
	\draw[dashed,line width=.5pt] (0,3) -- (1,-3/4);
	\draw[dotted,line width=.5pt] (0,3) -- (3,-9/4);

	\draw[dashed,line width=.5pt] (0,1) -- (-3,-9/4);
	\draw[dashed,line width=.5pt] (0,3) -- (-1,-3/4);
	\draw[dotted,line width=.5pt] (0,3) -- (-3,-9/4);

	\draw[line width=.5pt] (-1,-3/4) -- (1,-3/4);
	\draw[dashed,line width=.5pt] (-1,-3/4) -- (3,-9/4);
	\draw[dashed,line width=.5pt] (-3,-9/4) -- (1,-3/4);
	\draw[dotted,line width=.5pt] (-3,-9/4) -- (3,-9/4);

\end{tikzpicture}
\\
\hline
\end{tabular}
\end{table}


The main idea for the proof is to use the family over $\overline{\mathbf{N}}$ constructed by Naruki and Sekiguchi
in \cite{NS80} and endow it with the divisor intersecting the fibers of the family giving the $27$ lines and their degenerations. The main part of the argument is checking that the degenerations obtained this way are stable. This gives a morphism $\overline{\mathbf{N}}\rightarrow\overline{\mathbf{Y}}_{\frac{1}{9}+\epsilon}$, which we check is finite. The result then follows from Zariski's Main Theorem.



\subsection*{Acknowledgements}

We would like to thank Jeff Achter, Sebastian Casalaina-Martin, Eduardo Cattani, and Paul Hacking for insightful discussions, explanations, and for helping us correct imprecise statements in a preliminary draft of the paper. We also thank Han-Bom Moon for his feedback on the first draft of the paper. We are grateful for the working environments at the Department of Mathematics in Washington University in St. Louis and the University of Massachusetts Amherst where this research was conducted.


\section{Geometric interpretation of the toroidal compactifications of Deligne-Mostow ball quotients}
\label{proofofthemaintheorem}

In \S\ref{sec:Prelim} and \S\ref{sectiononhassettsmodulispaces} we recall the necessary background for the Deligne-Mostow ball quotients and Hassett's moduli spaces of weighted stable rational curves, which are necessary for the proof of Theorem~\ref{maintheorem}. The first step of the proof is carried out in \S\ref{prooftheoremintwoancestralcases}, where we show that our theorem holds for the so called \emph{ancestral cases} associated to $8$ and $12$ points in $\mathbb{P}^1$. Finally, using work of Doran (\cite[Theorem 4]{DDH18}), in \S\ref{sec:OtherDMCases} we show that these two ancestral cases imply all the others.


\subsection{Preliminaries on Deligne-Mostow ball quotients}
\label{sec:Prelim}

We start by reviewing the main result in \cite{DM86}, which we state following the exposition in \cite[\S8]{DK07} and \cite[\S2]{KLW87}. Let $n\geq5$ be an integer and let $\mathbf{w}:=(w_1,\ldots,w_n)$ be rational weights such that $1>w_{i}\geq w_{i+1}>0$ for all $i=1,\ldots,n-1$ and $w_1 + \ldots + w_{n} = 2$. Assume the weights satisfy the following condition:
\begin{enumerate}
\item[(a)] for any $i\neq j$ such that $w_i+w_j<1$, $(1-w_i-w_j)^{-1}\in\mathbb{Z}$.
\end{enumerate}
(This condition is called INT in \cite{DM86}.) Then there exists an arithmetic group $\Gamma_{\mathbf{w}}$ acting on a $(n-3)$-dimensional complex ball $\mathbb{B}_{n-3}$ such that $((\mathbb{P}^1)^n)^s/\!/_{\mathbf{w}}\SL_2$ is isomorphic to 
$
\Gamma_{\mathbf{w}} \backslash
\mathbb{B}_{n-3}$. Moreover, this isomorphism extends to the respective GIT and Baily-Borel compactifications. That is,
\[
(\mathbb{P}^1)^n/\!/_{\mathbf{w}}\SL_2\cong
\overline{\Gamma_{\mathbf{w}}\backslash\mathbb{B}_{n-3}}^{\bb}.
\]
There are then finitely many possibilities for this: $n$ can be $5,6,7$, or $8$, and the possible weights are the ones in \cite[Tables I and II]{KLW87} corresponding to rows where the entry in the column $\Sigma$ is empty (notice a typo: for $n=8$, the case $1^8$ should have empty entry under $\Sigma$). See also \cite[Appendix]{Thu98}.

The hypothesis (a) above can be relaxed still having an isomorphism after quotienting by the actions on some of the labels of the symmetric group $S_m$ for some $m\leq n$. More precisely, assume (a) is replaced by the following other condition: for any $i\neq j$ such that $w_i+w_j<1$,
\begin{enumerate}
\item[(b1)] if $w_i\neq w_j$, then $(1-w_i-w_j)^{-1}\in\mathbb{Z}$;
\item[(b2)] if $w_i=w_j$, then $2(1-w_i-w_j)^{-1}\in\mathbb{Z}$.
\end{enumerate}
(This condition is what is called $\Sigma$INT in 
\cite{Mos86}.) Then there exists an appropriate arithmetic group $\Gamma_{\mathbf{w}}$ acting on $\mathbb{B}_{n-3}$ and a positive integer $m\leq n$ such that
$((\mathbb{P}^1)^n)^s/\!/_{\mathbf{w}}\SL_2\rtimes S_m\cong
\Gamma_{\mathbf{w}}\backslash\mathbb{B}_{n-3}$. Moreover, this isomorphism extends to the compactifications
\begin{align*}
(\mathbb{P}^1)^n/\!/_{\mathbf{w}}\SL_2\rtimes S_m\cong
\overline{\Gamma_{\mathbf{w}}\backslash\mathbb{B}_{n-3}}^{\bb}.
\end{align*}
In this more general setting, $n$ is allowed to be also $9,10,11$, or $12$. The possible weights are the remaining ones in \cite[Tables I and II]{KLW87} (again, see also \cite[Appendix]{Thu98}).

\begin{remark}
\label{rmk:Table}
Tables~\ref{table:DMcasesEisenstein} and \ref{table:DMcasesGaussian} contain all the Deligne-Mostow cases that satisfy the hypotheses of Theorem~\ref{maintheorem}. These are obtained from \cite[Attached figure]{Dor04a}. The weights $\mathbf{w}$ corresponding to the marked points in $\mathbb{P}^1$ are normalized so that they add up to two. For simplicity we adopt the following convention: exponents denote how many times the corresponding weight appears. For instance,
\[
\left( \frac{1}{2} \right)
\left( \frac{1}{3} \right)^4
\left( \frac{1}{6} \right)
:=
\left( 
\frac{1}{2}, \frac{1}{3}, \frac{1}{3}, \frac{1}{3}, \frac{1}{3}, \frac{1}{6}
\right).
\]
In these tables, the entry under ``$m$" is left blank for the cases satisfying the INT condition, for which $S_m$ is trivial. In the cases satisfying $\Sigma$INT, the $m$ is given explicitly. The subdivision into two tables is related to the nature of the discrete arithmetic group $\Gamma_{\mathbf{w}}$. If $\Gamma_{\mathbf{w}}$ is realized as automorphisms of a lattice over the Gaussian integers $\mathbb{Z}[i]$, then we refer to $\Gamma_{\mathbf{w}}\backslash\mathbb{B}_n$ as \emph{Gaussian ball quotient} (Table~\ref{table:DMcasesGaussian}). If $\Gamma_{\mathbf{w}}$ is realized as automorphisms of a lattice over the Eisenstein integers $\mathbb{Z}[\omega]$, where $\omega$ is a primitive third root of unity, then we refer to $\Gamma_{\mathbf{w}}\backslash\mathbb{B}_n$ as \emph{Eisenstein ball quotient} (Table~\ref{table:DMcasesEisenstein}).
\end{remark}


\begin{remark}
The period map has a very explicit interpretation. Let $d$ be the least positive integer such that for all $i$, $w_i=m_i/d$ where $m_i\in\mathbb{Z}$. The period map uses curves given by the degree $d$ cover of $\mathbb{P}^1$ fully branched along $n$ points, c.f. \S\ref{s1.1}.
\end{remark}


\begin{table}
\caption{Deligne-Mostow Eisenstein cases in Remark~\ref{rmk:Table}.}
\label{table:DMcasesEisenstein}
\setlength{\tabcolsep}{10pt} 
\renewcommand{\arraystretch}{1.5}
\begin{tabular}{|c|c|c||c|c|c|}
\hline
Number of points & Weights 
& $m$
& Number of points & Weights & $m$
\\
\hline
12  &  $\left(\frac{1}{6} \right)^{12}$ 
& 12 & 6 & $\left(\frac{5}{6} \right)\left(\frac{1}{2} \right)\left(\frac{1}{6} \right)^{4}$
& $4$
\\
\hline
11  &  $\left(\frac{1}{3} \right)\left(\frac{1}{6} \right)^{10}$ & 10 & 6 & $\left(\frac{5}{6} \right)\left(\frac{1}{3} \right)^2\left(\frac{1}{6} \right)^{3}$
& $3$
\\
\hline
10  &  $\left(\frac{1}{2} \right)\left(\frac{1}{6} \right)^{9}$ & 9 & 6 & $\left(\frac{2}{3} \right)\left(\frac{1}{3} \right)^3\left(\frac{1}{6} \right)^{2}$
& $2$
\\
\hline
10  &  $\left(\frac{1}{3} \right)^2\left(\frac{1}{6} \right)^{8}$ & 8 & 6 & $\left(\frac{2}{3} \right)\left(\frac{1}{2} \right)\left(\frac{1}{3} \right)\left(\frac{1}{6} \right)^{3}$
& $3$
\\
\hline
9  &  $\left(\frac{2}{3} \right)\left(\frac{1}{6} \right)^{8}$ & 8 & 6 & $\left(\frac{2}{3} \right)^2\left(\frac{1}{6} \right)^{4}$
&
$4$
\\
\hline
9  &  $\left(\frac{1}{2} \right)\left(\frac{1}{3} \right)\left(\frac{1}{6} \right)^{7}$ & 7 & 6 & $\left(\frac{1}{2} \right)^3\left(\frac{1}{6} \right)^{3}$
& $3$
\\
\hline
9  &  $\left(\frac{1}{3} \right)^3\left(\frac{1}{6} \right)^{6}$ & 6 & 6 & $\left(\frac{1}{2} \right)^2 \left(\frac{1}{3} \right)^2\left(\frac{1}{6} \right)^{2}$
& $2$
\\
\hline
8  &  $\left(\frac{1}{3} \right)^4\left(\frac{1}{6} \right)^{4}$ & 4 & 6 & $\left(\frac{1}{2} \right)\left(\frac{1}{3} \right)^{4}\left(\frac{1}{6} \right)$
&
\\
\hline
8  &  $\left(\frac{1}{2} \right)\left(\frac{1}{3} \right)^2\left(\frac{1}{6} \right)^{5}$ & 5 & 6 & $\left(\frac{1}{3} \right)^{6}$
& 
\\
\hline
8  &  $\left(\frac{1}{2} \right)^2\left(\frac{1}{6} \right)^{6}$ & 6 & 5 & $\left(\frac{2}{3} \right)\left(\frac{1}{3} \right)^{4}$
&
\\
\hline
8  &  $\left(\frac{2}{3} \right)\left(\frac{1}{3} \right)\left(\frac{1}{6} \right)^{6}$ & 6 & 5 & $\left(\frac{1}{2} \right)^2\left(\frac{1}{3} \right)^3$
&
\\
\hline
8  &  $\left(\frac{5}{6} \right)\left(\frac{1}{6} \right)^{7}$ 
&
7
& 5 & $\left(\frac{1}{2} \right)^3\left(\frac{1}{3} \right)\left(\frac{1}{6} \right)$
&
\\ 
\hline
7  &  $\left(\frac{5}{6} \right)\left(\frac{1}{3} \right)\left(\frac{1}{6} \right)^{5}$ 
&
5
& 5 & $\left(\frac{2}{3} \right)^2\left(\frac{1}{3} \right)\left(\frac{1}{6} \right)^2$
& $2$
\\
\hline
7  &  $\left(\frac{2}{3} \right)\left(\frac{1}{3} \right)^2\left(\frac{1}{6} \right)^{4}$ 
&
4
& 5 & $\left(\frac{2}{3} \right)\left(\frac{1}{2} \right)^2\left(\frac{1}{6} \right)^2$
& $2$
\\
\hline
7  &  $\left(\frac{2}{3} \right)\left(\frac{1}{2} \right)\left(\frac{1}{6} \right)^{5}$ 
&
5
& 5 & $\left(\frac{2}{3} \right)\left(\frac{1}{2} \right)\left(\frac{1}{3} \right)^2\left(\frac{1}{6} \right)$
&
\\
\hline
7  &  $\left(\frac{1}{2} \right)^2\left(\frac{1}{3} \right)\left(\frac{1}{6} \right)^{4}$ 
&
4
& 5 & $\left(\frac{5}{6} \right)\left(\frac{1}{3} \right)^3\left(\frac{1}{6} \right)$
&
\\
\hline
7  &  $\left(\frac{1}{2} \right)\left(\frac{1}{3} \right)^3\left(\frac{1}{6} \right)^{3}$ &
3
& 5 & $\left(\frac{5}{6} \right)\left(\frac{1}{2} \right)    \left(\frac{1}{3} \right)\left(\frac{1}{6} \right)^2$
&
$2$
\\
\hline
7  &  $\left(\frac{1}{3} \right)^5\left(\frac{1}{6} \right)^2$ & 2 & 5 & $\left(\frac{5}{6} \right)\left(\frac{2}{3} \right)\left(\frac{1}{6} \right)^3$
&
$3$
\\ 
\hline
\end{tabular}
\end{table}


\begin{table}
\caption{Deligne-Mostow Gaussian cases in Remark~\ref{rmk:Table}.}
\label{table:DMcasesGaussian}
\setlength{\tabcolsep}{10pt} 
\renewcommand{\arraystretch}{1.5}
\begin{tabular}{|c|c|c|c|}
\hline
Number of points & Weights & $m$
\\
\hline
8 & $\left(\frac{1}{4} \right)^{8}$ & 
\\
\hline
7 & $\left(\frac{1}{2} \right)\left(\frac{1}{4} \right)^{6}$  &
\\
\hline
6 & $\left(\frac{3}{4} \right)\left(\frac{1}{4} \right)^{5}$ 
&
\\
\hline
6 & $\left(\frac{1}{2} \right)^{2}\left(\frac{1}{4} \right)^{4}$
&
\\
\hline
5 & $\left(\frac{3}{4} \right)\left(\frac{1}{2} \right)\left(\frac{1}{4} \right)^{3}$ &
\\
\hline
5 & $\left(\frac{1}{2} \right)^{3}\left(\frac{1}{4} \right)^{2}$  &
\\
\hline
\end{tabular}
\end{table}



The following result due to Doran (see \cite[Theorem 5 and Corollary 9]{Dor04a}) will be one of the ingredients in the proof of Theorem~\ref{maintheorem}. We recall it as recounted in \cite{DDH18}.

\begin{lemma}
[{\cite[Theorem 4]{DDH18}}]
\label{lemma:DoranBallQ}
The GIT moduli space of eight points in $\mathbb{P}^1$ with $\mathbf{w}=(\frac{1}{4},\ldots,\frac{1}{4})$ is the largest dimensional Gaussian Deligne-Mostow ball quotient $\Gamma_{\mathbf{w}}\backslash\mathbb{B}_5$. All the other Gaussian Deligne-Mostow ball quotients (see Table~\ref{table:DMcasesGaussian}) arise as sub-ball quotients of $\Gamma_{\mathbf{w}}\backslash\mathbb{B}_5$.

The GIT moduli space of $12$ points in $\mathbb{P}^1$ with $\mathbf{w}=\left(\frac{1}{6},\ldots,\frac{1}{6}\right)$ modulo $S_{12}$ is the largest dimensional Eisenstein Deligne-Mostow ball quotient $\Gamma_{\mathbf{w}}\backslash\mathbb{B}_9$. All the other Eisenstein Deligne-Mostow ball quotients (see Table~\ref{table:DMcasesEisenstein}) arise as sub-ball quotients of $\Gamma_{\mathbf{w}}\backslash\mathbb{B}_9$.
\end{lemma}

\begin{remark}
The cases corresponding to $n=8$ with weights
$\mathbf{w} = \left( \frac{1}{4} \right)^8$ and $n=12$ with weights 
$\mathbf{w} = \left( \frac{1}{6} \right)^{12}$ are known as the \emph{ancestral cases}. The ancestral case for $n=8$ cannot be deduced from the ancestral case for $n=12$ because the monodromy group for $n=8$ is defined over the Gaussian integers, while the monodromy group for $n=12$ is defined over the Eisenstein integers.
\end{remark}

\begin{example}
We can recover the GIT quotient of six points with respect to the weights 
$\mathbf{\frac{1}{3}} = \left(  \frac{1}{3} \right)^6 $
from the  GIT quotient of twelve points with respect to the weights $\mathbf{\frac{1}{6}} = \left(  \frac{1}{6} \right)^{12}$. 
Indeed,  configurations of twelve points with two colliding points are stable. Therefore, the locus where there are six pairs of colliding points, say $p_1=p_7,p_2=p_8,\ldots,p_6=p_{12}$, can be used to define an embedding
\begin{align*}
( \mathbb{P}^1)^6/\!/_{ \mathbf{\frac{1}{3}} } \SL_2 
\hookrightarrow
( \mathbb{P}^1)^{12}/\!/_{ \mathbf{\frac{1}{6}}} \SL_2.
\end{align*}
Now consider the Deligne-Mostow isomorphisms $(\mathbb{P}^1)^{12}/\!/_{ \mathbf{\frac{1}{6}}}\SL_2\rtimes S_{12}\cong\overline{\Gamma_{ \mathbf{\frac{1}{6}} } \backslash\cB_{9}}^{\bb}$ and $(\mathbb{P}^1)^6/\!/_{ \mathbf{\frac{1}{3}}}\SL_2\cong\overline{\Gamma_{ \mathbf{\frac{1}{3}} } \backslash\cB_{3}}^{\bb}$. We have the following commutative diagram:
\begin{center}
\begin{tikzpicture}[>=angle 90]
\matrix(a)[matrix of math nodes,
row sep=2em, column sep=2em,
text height=2.5ex, text depth=1.25ex]
{( \mathbb{P}^1)^6/\!/_{ \mathbf{\frac{1}{3}} } \SL_2&( \mathbb{P}^1)^{12}/\!/_{ \mathbf{\frac{1}{6}} } \SL_2\rtimes S_{12}\\
\overline{
\Gamma_{ \mathbf{\frac{1}{3}} } \backslash
\cB_{3}
}^{\bb}
&
\overline{
\Gamma_{ \mathbf{\frac{1}{6}} } \backslash
\cB_{9}
}^{\bb},
\\};
\path[->] (a-1-1) edge node[above]{}(a-1-2);
\path[->] (a-1-1) edge node[left]{$\cong$}(a-2-1);
\path[->] (a-2-1) edge node[above]{}(a-2-2);
\path[->] (a-1-2) edge node[right]{$\cong$}(a-2-2);
\end{tikzpicture}
\end{center}
where the horizontal maps are degree $6!$ covers.
\end{example}

\begin{remark}
There are some Deligne-Mostow cases which are not contained in one of the two ancestral ones. These special cases are not considered in this paper because either the corresponding Baily-Borel compactifications have no cusps, or the monodromy group is not arithmetic.
\end{remark}


\subsection{Weighted stable rational curves}
\label{sectiononhassettsmodulispaces}

In \cite{Has03}, Hassett introduced the moduli spaces of \emph{weighted stable curves}. In the current paper we are interested in the genus zero case, which we now recall. Let $n$ be a positive integer and let $0<b_1,\ldots,b_n\leq1$ be rational numbers.

\begin{definition}[{\cite[\S2]{Has03}}]
A \emph{weighted stable rational curve for the weight} $\mathbf{b}=(b_1,\ldots,b_n)$ is a pair $(X,\sum_{i=1}^nb_ix_i)$, where
\begin{enumerate}

\item $X$ is a nodal connected curve of arithmetic genus zero;

\item $x_1,\ldots,x_n$ are smooth points of $X$;

\item If $x_{i_1}=\ldots=x_{i_r}$, then $b_{i_1}+\ldots+b_{i_r}\leq1$;

\item The divisor $K_X+\sum_{i=1}^nb_ix_i$ is ample. In other words, for each irreducible component $C\subseteq X$, if $N$ denotes the number of nodes of $X$ on $C$,  we have that
\[
N+\sum_{x_i\in C}b_i>2.
\]
\end{enumerate}
\end{definition}

\begin{theorem}[{\cite[Theorem 2.1]{Has03}}]
There exists a fine smooth projective moduli space $\overline{\mathrm{M}}_{0,\mathbf{b}}$ parametrizing weighted stable rational curves of weight $\mathbf{b}$, and containing $\mathrm{M}_{0,n}$ as a Zariski open subset. In particular, for $\mathbf{b}=(1, \ldots, 1)$, this space is the Deligne-Mumford-Knudsen compactification $\overline{\mathrm{M}}_{0,n}$.
\end{theorem}

\begin{remark}
Let $\mathbf{b}=(b_1,\ldots,b_n)$ and $\mathbf{c}=(c_1,\ldots,c_n)$ be weights such that $b_i\geq c_i$ for all $i$. Then there exists a birational morphism
\[
\rho_{\mathbf{b},\mathbf{c}}\colon\overline{\mathrm{M}}_{0,\mathbf{b}}\rightarrow\overline{\mathrm{M}}_{0,\mathbf{c}}
\]
called the reduction morphism (see \cite[]{Has03}). The locus where $\rho_{\mathbf{b},\mathbf{c}}$ is an isomorphism contains the locus parametrizing $n$ distinct points in $\mathbb{P}^1$.
\end{remark}

\begin{example}\label{ex:Reduc}
Define the collection of eight weights $\mathbf{\frac{1}{4}}+\epsilon=\left( \frac{1}{4}+\epsilon, \ldots, \frac{1}{4}+\epsilon\right)$, where $\epsilon\ll1$ is a positive rational number. Consider the reduction morphism $\rho\colon\overline{\mathrm{M}}_{0,8}\rightarrow\overline{\mathrm{M}}_{0,\mathbf{\frac{1}{4}}+\epsilon}$. Let $D_{123,45678} \subseteq\overline{\mathrm{M}}_{0,8}$ be the boundary divisor generically parametrizing the gluing of two copies of $\mathbb{P}^1$ at one point, where one copy supports the points $\{ p_1, p_2, p_3\}$ while the other one supports the points $\{ p_4, p_5, p_6, p_7, p_8 \}$. The morphism $\rho$ contracts the divisor $D_{123,45678}$ onto a codimension two locus generically parametrizing $\mathbb{P}^1$ supporting the configurations of points $\{p_1=p_2=p_3,p_4, p_5, p_6, p_7, p_8\}$.
\end{example}

\begin{remark}
Given a collection of positive weights $\mathbf{a}=(a_1, \ldots, a_n)$ such that $a_i\leq1$ and $a_1 + \ldots + a_n=2$, then these do not define a moduli space of weighted stable rational curves, but rather the GIT quotient $(\mathbb{P}^1)^n / \! /_{\mathbf{a}} \text{SL}_2$. Now let $\mathbf{b}=(b_1,\ldots,b_n)$ be a second collection of weights such that $b_i\geq a_i$ for all $i$ and $b_1+\ldots+b_n>2$. Then there exists a birational morphism
\[
\overline{\mathrm{M}}_{0,\mathbf{b}}\rightarrow
(\mathbb{P}^1)^n / \! /_{\mathbf{a}} \SL_2.
\]
Note that if the GIT quotient does not have strictly semistable points and $b_i$ is sufficiently close to $a_i$ for all $i$, then the birational morphism above is an isomorphism. On the other hand, if the GIT quotient has strictly semistable points and again the weights $b_i$ are sufficiently close to the $a_i$, then $\overline{\mathrm{M}}_{0,\mathbf{b}}\rightarrow
(\mathbb{P}^1)^n / \! /_{\mathbf{a}} \SL_2$ is Kirwan's partial desingularization.
\end{remark}

\begin{example}\label{ex:8points}
Given the collection of eight weights $\mathbf{\frac{1}{4}}=\left( 
\frac{1}{4}, \ldots, \frac{1}{4}
\right)$, we have a birational morphism 
\[
f\colon\overline{\mathrm{M}}_{0,\mathbf{\frac{1}{4}} + \epsilon}
\rightarrow
(\mathbb{P}^1)^8 / \! /_{\mathbf{\frac{1}{4}}} \SL_2,
\]
which is a blow up at the strictly semi-stable points of the GIT quotient. The exceptional divisors are isomorphic to $\mathbb{P}^2 \times \mathbb{P}^2$. From the moduli perspective, the morphism $f$ is interpreted as follows. A point $x$ in a dense open subset of the exceptional divisors parametrizes an $8$-pointed stable rational curve as shown on the left of Figure~\ref{modularinterpretationreductiontogit}. Then $f(x)$ is the point in $(\mathbb{P}^1)^8/\!/_{\mathbf{\frac{1}{4}}}\SL_2$ corresponding to the strictly semi-stable configuration on the right of Figure~\ref{modularinterpretationreductiontogit}.
\end{example}

\begin{figure}
\begin{tikzpicture}[scale=0.3]

	\draw[line width=1.0pt] (-2,-1) -- (10,5);

	\draw[line width=1.0pt] (6,5) -- (18,-1);

	\fill (0,0) circle (8pt);
	\fill (2,1) circle (8pt);
	\fill (4,2) circle (8pt);
	\fill (6,3) circle (8pt);

	\fill (10,3) circle (8pt);
	\fill (12,2) circle (8pt);
	\fill (14,1) circle (8pt);
	\fill (16,0) circle (8pt);

	\node at (0,0+1) {$p_1$};
	\node at (2,1+1) {$p_2$};
	\node at (4,2+1) {$p_3$};
	\node at (6,3+1) {$p_4$};

	\node at (10,3+1) {$p_5$};
	\node at (12,2+1) {$p_6$};
	\node at (14,1+1) {$p_7$};
	\node at (16,0+1) {$p_8$};

	\draw[line width=1.0pt] (22,2) -- (22+12,2);

	\fill (22+2,2) circle (8pt);
	\fill (22+10,2) circle (8pt);

	\node at (22+2,2+1.5) {$p_1,\ldots,p_4$};
	\node at (22+10,2+1.5) {$p_5,\ldots,p_8$};

\end{tikzpicture}
\caption{Modular interpretation of the morphism $\overline{\mathrm{M}}_{0,\mathbf{\frac{1}{4}}+\epsilon}\rightarrow(\mathbb{P}^1)^n/\!/_{\mathbf{\frac{1}{4}}}\SL_2$ on a dense open subset of the exceptional divisors.}
\label{modularinterpretationreductiontogit}
\end{figure}
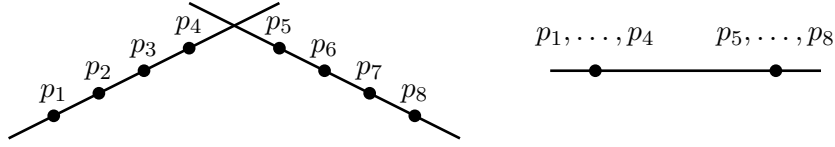


\subsection{Proof of the main theorem for the ancestral cases}
\label{prooftheoremintwoancestralcases}

The first step is to prove the following proposition.

\begin{proposition}
\label{lemma:OpenExt}
The following hold:
\begin{enumerate}

\item The rational map
$
\overline{\mathrm{M}}_{0,\mathbf{\frac{1}{4}}+\epsilon}
\dashrightarrow 
\overline{
\Gamma_{\mathbf{\frac{1}{4}}}
\backslash
\mathbb{B}_5}^{\tor}
$ 
extends to an isomorphism away from closed subsets of codimension at least two.

\item The rational map
$
\overline{\mathrm{M}}_{0,\mathbf{\frac{1}{6}}+\epsilon}/S_{12}
\dashrightarrow 
\overline{
\Gamma_{\mathbf{\frac{1}{6}}}
\backslash
\mathbb{B}_9}^{\tor}
$ 
extends to an isomorphism away from closed subsets of codimension at least two.

\end{enumerate}
\end{proposition}

\begin{proof}
We first discuss the case $n=8$ and $\mathbf{w}=\mathbf{\frac{1}{4}}$. For simplicity of notation we set $\mathbb{B}=\mathbb{B}_5$ and $\Gamma=\Gamma_{\mathbf{\frac{1}{4}}}$. Consider the following morphisms:
\begin{enumerate}

\item The birational reduction morphism $\rho\colon\overline{\mathrm{M}}_{0,8}\rightarrow\overline{\mathrm{M}}_{0,\mathbf{\frac{1}{4}}+\epsilon}$.

\item The birational map $\overline{\mathrm{M}}_{0,8}\dashrightarrow\overline{\Gamma\backslash\mathbb{B}}^{\tor}$, which by Lemma~\ref{extensiontotoroidalwithnccompactificationandfreeactionandhyperplanearrangement} extends to a birational morphism $\ell\colon\overline{\mathrm{M}}_{0,8}\rightarrow\overline{\Gamma\backslash\mathbb{B}}^{\tor}$. Let us explain why we can apply Lemma~\ref{extensiontotoroidalwithnccompactificationandfreeactionandhyperplanearrangement}. First of all, using the isomorphism $\Gamma\backslash\mathbb{B}\cong((\mathbb{P}^1)^8)^s/\!/_{\mathbf{\frac{1}{4}}}\SL_2$, let $\Gamma\backslash\mathcal{H}$ be the locus where at least two points coincide. By \cite[\S2]{MY93}, we have that $\Gamma$ acts freely on $\mathbb{B}\setminus\mathcal{H}$. Finally $\overline{\mathrm{M}}_{0,8}$ is a normal crossing compactification of $\Gamma\backslash(\mathbb{B}\setminus\mathcal{H})\cong\mathrm{M}_{0,8}$.

\item The composition $\pi\colon\overline{\mathrm{M}}_{0, \mathbf{\frac{1}{4}}+\epsilon}\rightarrow
(\mathbb{P}^1)^8/\!/_{\mathbf{\frac{1}{4}}}\SL_2\cong\overline{\Gamma\backslash\mathbb{B}}^{\bb}$, which is the regular blow up at the cusps (see
\cite[Theorem 1.1]{KM11}).

\item The blow up $\varphi\colon\overline{\Gamma\backslash\mathbb{B}}^{\tor}\rightarrow\overline{\Gamma\backslash\mathbb{B}}^{\bb}$.

\end{enumerate}

We can fit all these maps in the following commutative diagram:

\begin{center}
\begin{tikzpicture}[>=angle 90]
\matrix(a)[matrix of math nodes,
row sep=2em, column sep=2em,
text height=2.0ex, text depth=0.50ex]
{&\overline{\mathrm{M}}_{0,8}&\\
\overline{\mathrm{M}}_{0,\mathbf{\frac{1}{4}}+\epsilon}&&\overline{\Gamma\backslash\mathbb{B}}^{\tor}\\
&\overline{\Gamma\backslash\mathbb{B}}^{\bb}.&\\};
\path[->] (a-1-2) edge node[above left]{$\rho$}(a-2-1);
\path[->] (a-1-2) edge node[above right]{$\ell$}(a-2-3);
\path[->] (a-2-1) edge node[below left]{$\pi$}(a-3-2);
\path[->] (a-2-3) edge node[below right]{$\varphi$}(a-3-2);
\end{tikzpicture}
\end{center}
The above diagram commutes because the two compositions $\pi\circ\rho$ and $\varphi\circ\ell$ agree on the open subset $\mathrm{M}_{0,8}\subseteq\overline{\mathrm{M}}_{0,8}$, over which they are an isomorphism with $\Gamma\backslash(\mathbb{B}\setminus\mathcal{H})$.

We now construct a birational map $f\colon\overline{\mathrm{M}}_{0,\mathbf{\frac{1}{4}}+\epsilon}\dashrightarrow\overline{\Gamma\backslash\mathbb{B}}^{\tor}$ defined on an open subset whose complement has codimension at least two. Let $U\subseteq\overline{\mathrm{M}}_{0,8}$ be the maximal dense open subset where $\rho$ is an isomorphism. Notice that $\mathrm{M}_{0,8}\subsetneq U$: for instance, a dense open subset of the boundary divisors $D_{ij}$ is contained in $U$. Define
\[
f_1=\ell\circ\rho|_U^{-1}\colon \rho(U)\rightarrow\overline{\Gamma\backslash\mathbb{B}}^{\tor}.
\]
Moreover, define $f_2\colon\pi^{-1}(\Gamma\backslash\mathbb{B})\rightarrow\overline{\Gamma\backslash\mathbb{B}}^{\tor}$ to be the composition $\varphi|_{\Gamma\backslash\mathbb{B}}^{-1}\circ\pi|_{\pi^{-1}(\Gamma\backslash\mathbb{B})}$. If we take $P^0$ to be the union $\rho(U)\cup\pi^{-1}(\Gamma\backslash\mathbb{B})$, then we can define $f\colon P^0\rightarrow\overline{\Gamma\backslash\mathbb{B}}^{\tor}$ by gluing $f_1,f_2$. To be able to do this, we have to check that $f_1$ and $f_2$ agree on $\rho(U)\cap\pi^{-1}(\Gamma\backslash\mathbb{B})$. Let $V\subseteq\rho(U)\cap\pi^{-1}(\Gamma\backslash\mathbb{B})$ be the dense open subset corresponding to eight distinct points in $\mathbb{P}^1$. Notice that $f_1|_V=f_2|_V$ by the commutativity of the diagram above. This implies that $f_1$ and $f_2$ agree on the whole $\rho(U)\cap\pi^{-1}(\Gamma\backslash\mathbb{B})$ by \cite[Chapter II, Exercise 4.2]{Har77}.

Let $D$ be the exceptional locus of the morphism $\pi$. Notice that $D\cap P^0=D\cap\rho(U)\neq\emptyset$, so $P^0$ intersects $D$ in a dense open subset. This implies that
\[
\codim\left(\overline{\mathrm{M}}_{0, \mathbf{\frac{1}{4}} + \epsilon}\setminus P^0\right) \geq 2.
\]
Moreover, let $B$ be the exceptional locus of the morphism $\varphi$. We have that
\[
f(D\cap P^0)=f(D\cap\rho(U))=\ell(\rho|_U^{-1}(D\cap\rho(U)))=\ell(\rho^{-1}(D)\cap U)
\]
gives a dense open subset of $B$ because $\ell$ is surjective ($\ell$ is dominant and it is a morphism of projective varieties), hence $f|_D\colon D\dashrightarrow B$ is a dominant map. In conclusion, it follows by Lemma~\ref{extensionOpen} that $f$ restricts to an isomorphism between $\overline{\mathrm{M}}_{0,\mathbf{\frac{1}{4}}+\epsilon}$ and $\overline{\Gamma\backslash\mathbb{B}}^{\tor}$ away from closed subsets of codimension at least two.

We now move to the $12$ points case, where again for simplicity of notation we set $\mathbb{B}=\mathbb{B}_9$ and $\Gamma=\Gamma_{\mathbf{\frac{1}{6}}}$. Let us start by considering the Deligne-Mostow isomorphism $\overline{\Gamma\backslash\mathbb{B}}^{\bb}\cong(\mathbb{P}^1)^{12}/\!/_{\mathbf{\frac{1}{6}}}\SL_2\rtimes S_{12}$. By \cite[Theorem 7.8]{HL00}, the locus in the GIT quotient where at least two points coincide corresponds under this isomorphism to $\Gamma\backslash\mathcal{H}$, where $\mathcal{H}$ can be described as follows. Let $\mathbb{Z}[\omega]$ be the ring of Eisenstein integers, where $\omega$ is a primitive cube root of $1$. The ball $\mathbb{B}$ is obtained from a specific hyperbolic lattice $L$ over the Eisenstein integers (see \cite[\S5]{All00} for an explicit description of $L$). If $r\in L$ is an element such that $r^2=-3$ and $r\in(\omega-\overline{\omega})L^*$, then the isometry
\[
x\mapsto x-(1-\omega)\frac{x\cdot r}{r^2}r
\]
has order $3$ and is called a \emph{triflection}. For a triflection we define the \emph{mirror} as the orthogonal complement of $r$. Then $\mathcal{H}$ is given by the mirrors of these triflections. These are all the mirrors by \cite[\S2.2]{Bas07}. Then $\Gamma$ acts freely on $\mathbb{B}\setminus\mathcal{H}$ by \cite[Theorem 1.5]{Ste64} (see also \cite{Leh04}).

By taking the composition of the maps
\[
\overline{\mathrm{M}}_{0,12}\rightarrow\overline{\mathrm{M}}_{0,\mathbf{\frac{1}{6}} + \epsilon }\rightarrow(\mathbb{P}^1)^{12}/\!/_{\mathbf{\frac{1}{6}}}\SL_2 \rtimes S_{12}\xrightarrow{\cong}\overline{\Gamma\backslash\mathbb{B}}^{\bb},
\]
we have a generically finite morphism 
$
\overline{\mathrm{M}}_{0,12} \longrightarrow
\overline{ \Gamma \backslash \cB}^{\bb}
$ 
such that the preimage of the 
$\Gamma
\backslash\mathcal{H}$ is the boundary of $\overline{\mathrm{M}}_{0,12}$, which we know is normal crossing. Therefore by Lemma~\ref{extensiontotoroidalwithnccompactificationandfreeactionandhyperplanearrangement} we have a morphism $\overline{\mathrm{M}}_{0,12}\rightarrow\overline{\Gamma\backslash\mathbb{B}}^{\tor}$.

We now construct a birational map $f\colon\overline{\mathrm{M}}_{0,\mathbf{\frac{1}{6}}+\epsilon}\dashrightarrow\overline{\Gamma\backslash\mathbb{B}}^{\tor}$ defined on an open subset whose complement has codimension at least two. Consider the commutative diagram
\begin{center}
\begin{tikzpicture}[>=angle 90]
\matrix(a)[matrix of math nodes,
row sep=2em, column sep=2em,
text height=2.0ex, text depth=0.50ex]
{&\overline{\mathrm{M}}_{0,12}&\\
\overline{\mathrm{M}}_{0,\mathbf{\frac{1}{6}}+\epsilon}&&\overline{\Gamma\backslash\mathbb{B}}^{\tor}\\
&\overline{\Gamma\backslash\mathbb{B}}^{\bb},&\\};
\path[->] (a-1-2) edge node[above left]{$\rho$}(a-2-1);
\path[->] (a-1-2) edge node[above right]{$\ell$}(a-2-3);
\path[->] (a-2-1) edge node[below left]{$\pi$}(a-3-2);
\path[->] (a-2-3) edge node[below right]{$\varphi$}(a-3-2);
\end{tikzpicture}
\end{center}
where we note that $\pi$ is the composition of the quotient by the $S_{12}$-action and a blow down. The above diagram commutes because the two compositions $\pi\circ\rho$ and $\varphi\circ\ell$ agree on the open subset $\mathrm{M}_{0,12}\subseteq\overline{\mathrm{M}}_{0,12}$, over which they are the quotient by the $S_{12}$-action.

Following the same strategy as for eight points, we construct a rational map $f\colon\overline{\mathrm{M}}_{0,\mathbf{\frac{1}{6}}+\epsilon}\dashrightarrow\overline{\Gamma\backslash\mathbb{B}}^{\tor}$. Let $U\subseteq\overline{\mathrm{M}}_{0,12}$ be the maximal dense open subset where $\rho$ is an isomorphism. Define
\[
f_1=\ell\circ\rho|_U^{-1}\colon \rho(U)\rightarrow\overline{\Gamma\backslash\mathbb{B}}^{\tor}.
\]
Moreover, define $f_2\colon\pi^{-1}(\Gamma\backslash\mathbb{B})\rightarrow\overline{\Gamma\backslash\mathbb{B}}^{\tor}$ to be the composition $\varphi|_{\Gamma\backslash\mathbb{B}}^{-1}\circ\pi|_{\pi^{-1}(\Gamma\backslash\mathbb{B})}$. If we take $P^0$ to be the union $\rho(U)\cup\pi^{-1}(\Gamma\backslash\mathbb{B})$, then we can define $f\colon P^0\rightarrow\overline{\Gamma\backslash\mathbb{B}}^{\tor}$ by gluing $f_1,f_2$ in the same way as we did for eight points.

Let $D$ be the exceptional locus of the morphism $\pi$. Notice that $D\cap P^0=D\cap\rho(U)\neq\emptyset$, so $P^0$ intersects $D$ in a dense open subset. This implies that
\[
\codim\left(\overline{\mathrm{M}}_{0, \mathbf{\frac{1}{6}} + \epsilon}\setminus P^0\right) \geq 2.
\]
Moreover, let $B$ be the exceptional locus of the morphism $\varphi$. Also in this case we have that $f(D\cap P^0)$ is a dense open subset of $B$ because $\ell$ is surjective, hence $f|_D\colon D\dashrightarrow B$ is a dominant map. The key fact now is that $P^0$ is an $S_{12}$-invariant subset, so there is an induced birational map $\overline{f}\colon\overline{\mathrm{M}}_{0,\mathbf{\frac{1}{6}}+\epsilon}/S_{12}\dashrightarrow\overline{\Gamma\backslash\mathbb{B}}^{\tor}$ which satisfies the assumptions of Lemma~\ref{extensionOpen}. So $\overline{f}$ restricts to an isomorphism between $\overline{\mathrm{M}}_{0,\mathbf{\frac{1}{6}}+\epsilon}/S_{12}$ and $\overline{\Gamma\backslash\mathbb{B}}^{\tor}$ away from closed subsets of codimension at least two.
\end{proof}

The next step is to show that the above isomorphisms away from closed subsets of codimension at least two extends to global isomorphisms in our cases of interest. This is nontrivial, and relies on the geometry of Hassett's moduli spaces and toroidal compactifications. Let us start by recalling some relevant definitions and facts.

\begin{definition}
Let $\phi\colon X \to S$ be a proper surjective morphism of varieties with $X$ 
normal. By $N_1(X/S)$ we denote the real vector space generated by numerical equivalence classes of reduced irreducible curves that are mapped to points in $S$ by $\phi$. Let
$\overline{\NE}(X/S)\subseteq N_1(X/S)$ be the 
closed convex cone generated by the classes of reduced irreducible curves and let $H\in\Pic(X)$. By Kleiman's criterion for ampleness, $H\in\Pic(X)$ is \emph{relatively ample} if and only if $H$ gives a positive function on $\overline{\NE}(X/S) \setminus \{ 0 \}$ (see \cite[Theorem 0-1-2]{KMM87}).
\end{definition}

\begin{definition}
Let $V \subseteq \mathbb{P}^n$  be a projective scheme  and $f_1, \ldots,f_s$  generators of its homogeneous ideal. The \emph{classical affine cone over $V$} is the variety $C_a(V) \subseteq \mathbb{A}^{n+1}$ defined by the equations of $V$.
The notions of affine cone can be generalized as follows, 
see \cite[\S3.1]{Kol13}.  Let $L$ be an ample line bundle on $V$. Then the \emph{affine cone over $V$ with conormal bundle $L$} is
\begin{align*}
    \label{eq:affineCone}
C_a(V,L)=\Spec_\mathbb{C}\left(\sum_{m\geq0}H^0(V,L^{\otimes m})\right).
\end{align*}
\end{definition}

\begin{proposition}
\label{pre-relativeamplenesshassettside}
Let $V$ be a smooth projective Fano variety and let $r$ be any positive rational number.  We denote the affine cone $C_a(V,-rK_V)$ simply as $C_r(V)$. Assume that 
\begin{enumerate}

\item $C_r(V)$ is $\mathbb{Q}$-Gorenstein, and

\item $\pi\colon X\rightarrow C_r(V)$ is a blow up at the vertex with $X$ smooth and exceptional divisor $E\cong V$.

\end{enumerate}
Then $-E$ is relatively ample.
\end{proposition}

\begin{proof}
Equivalently, we show that $-E|_E$ is ample. Let ``$\sim_\mathbb{Q}$" denote linear equivalence of $\mathbb{Q}$-divisors (that is, $D_1\sim_{\mathbb{Q}}D_2$ provided there exists a nonzero integer $m$ such that $mD_1$ and $mD_2$ are linearly equivalent). We have that $K_X \sim_\mathbb{Q}\varphi^*K_{C_r(V)}+\alpha E$. After adding $E$ to both sides and by using the projection formula, we obtain that
\begin{equation*}
K_X+E\sim_\mathbb{Q}\varphi^*K_{C_r(V)}+(1+\alpha)E\implies(K_X+E)|_E\sim_\mathbb{Q}(1+\alpha)E|_E\implies-K_E\sim_\mathbb{Q}(1+\alpha)(-E)|_E.
\end{equation*}
Notice that $1+\alpha>0$ by \cite[ Lemma 3.1 (3)]{Kol13} and $-K_E$ is ample because $E\cong V$, which is Fano. So $-E|_E$ is ample as claimed.
\end{proof}

\begin{proposition}\label{relativeamplenesshassettside}
For even $n$, let $K_1,\ldots, K_m$ be the exceptional divisors of the blow up
\[
\overline{\mathrm{M}}_{0,\mathbf{\frac{2}{n}} + \epsilon}
\rightarrow
(\mathbb{P}^1)^n/ \! /_{\mathbf{\frac{2}{n}}} \SL_2,
\]
and denote by $K$ their disjoint union. Then $-K$ is a relatively ample divisor.
\end{proposition}

\begin{proof}
By \cite[Theorem 1.1]{KM11}, we know that the morphism in the statement is the blow up at the strictly semi-stable points of the GIT quotient. So let $p=(p_1,\ldots,p_n)$ be one of these strictly semi-stable point configurations in $\mathbb{P}^1$. Notice that, up to permuting the indices, $p$ satisfies $p_1=\ldots=p_\frac{n}{2}\neq p_{\frac{n}{2}+1}=\ldots=p_n$. The idea is to apply Proposition~\ref{pre-relativeamplenesshassettside}, but to be able to do this we need a local description of the singularity at $p$. Such a description is provided by Luna's slice theorem \cite{Lun75}. 

More precisely, the completion $\widehat{\mathcal{O}}_{p}$ of the local ring at $p$ is isomorphic to the the completion of the local ring at the origin of the affine quotient $\mathcal{N}_p / \! / G_p$, where $G=\SL_2$ and $\mathcal{N}_p$ is the fiber at $p$ of the normal bundle to the orbit $G\cdot p$ (see \cite[\S1.1 and \S1.2]{Dre04}).
Notice that the stabilizer $G_p$ is isomorphic to $\mathbb{C}^*$.
By the discussion within the proofs
of \cite[Proposition 1.10 and Lemma 6.1]{KLW87}, it holds that 
\begin{align*}
\mathcal{N}_p / \! / G_p
\cong
\mathbb{C}^{n-2} / \! /_{\chi} \mathbb{C}^*,
\end{align*}
where $\chi$ is an appropriate character of $\mathbb{C}^*$ for which the $\mathbb{C}^*$-action is explicitly given by
\begin{align*}
t \cdot (u_1, \ldots u_{b}, v_1, \ldots v_{b})
=
(tu_1, \ldots tu_{b}, t^{-1}v_1, \ldots t^{-1}v_{b}),~b=\frac{n-2}{2}.
\end{align*}
This affine GIT quotient is studied in \cite[Theorem 6.27 (i)]{Muk03}.
The result in [loc. cit.] implies that our GIT quotient at $p$ is locally isomorphic to the affine cone over the Segre variety 
$\mathbb{P}^{b-1} \times \mathbb{P}^{b-1}\subseteq\mathbb{P}^{b^2-1}$. Our result now follows from Proposition~\ref{pre-relativeamplenesshassettside}.
\end{proof}

\begin{proposition}
\label{relativeamplenesstoroidalside}
Let $\overline{\mathcal{B}/\Gamma}^{\tor}$ be the unique toroidal compactification of a ball quotient $\mathcal{B}/\Gamma$. Let $E_1,\ldots,E_m$ be the irreducible exceptional divisors of $\overline{\mathcal{B}/\Gamma}^{\tor}\rightarrow\overline{\mathcal{B}/\Gamma}^{\bb}$, and denote by $E$ their disjoint union. Then $-E$ is a relatively ample divisor.
\end{proposition}

\begin{proof}
We show that $-E|_E$ is ample. Let $\Gamma_0\subseteq\Gamma$ be a finite index normal neat subgroup (notice that $\overline{\mathcal{B}/\Gamma_0}^{\tor}$ is automatically smooth by \cite[Theorem 7.26]{Nam80}). Let $\Gamma_1\subseteq\Gamma_0$ be a finite index subgroup such that $\overline{X}=\overline{\mathcal{B}/\Gamma_1}^{\tor}$ has $K_{\overline{X}}$ ample \cite[Theorem 1.3]{DD17}. Denote by $F$ the preimage of $E$ under the map $\overline{\mathcal{B}/\Gamma_1}^{\tor}\rightarrow\overline{\mathcal{B}/\Gamma}^{\tor}$. Then it suffices to prove that $-F|_F$ is ample by \cite[Corollary 1.2.28]{Laz04}. But $F$ is a disjoint union of abelian varieties, so by adjunction formula
\[
0=K_F=(K_{\overline{X}}+F)|_F\implies-F|_F=K_{\overline{X}}|_F,
\]
Since $K_{\overline{X}}|_F$ is ample, this proves what we want.
\end{proof}

\begin{lemma}[{\cite[Theorem 5.14]{Kov09}}]
\label{kovacs}
Let $S$ be a scheme and $f_i:X_i \to S$ two proper $S$-schemes, $\mathcal{L}_i$ relatively ample line bundles on 
$X_i/S$ and $j_i:U_i \to X_i$ open immersions with complement $Z_i = X_i \setminus U_i$ for $i = 1,2$.
Assume that
\begin{itemize}
    \item there exists an $S$-isomorphism 
    $\alpha:U_1/S \to U_2 /S$ such that 
    $\alpha^* \left( \mathcal{L}_2\right) = 
    \mathcal{L}_1, and
    $
    \item For $i=1,2$, $X_i$ is normal and $\codim(Z_i, X_i) \geq 2$.
\end{itemize}
Then $\alpha$ extends to $X_1$ to give an isomorphism 
$X_1/S \cong X_2/S$. 
\end{lemma}

\begin{proposition}
\label{globalextensionresult}
The rational maps $\overline{\mathrm{M}}_{0,\mathbf{\frac{1}{4}}+\epsilon}
\dashrightarrow 
\overline{
\Gamma_{\mathbf{\frac{1}{4}}} \backslash
\cB_5}^{\tor}$ and  $\overline{\mathrm{M}}_{0,\mathbf{\frac{1}{6}}+\epsilon}/S_{12}
\dashrightarrow 
\overline{
\Gamma_{\mathbf{\frac{1}{6}}}
\backslash
\cB_{9}}^{\tor}$ extend to isomorphisms.
\end{proposition}

\begin{proof}
We first discuss the case $n=8$. By Proposition~\ref{lemma:OpenExt} the rational map $\overline{\mathrm{M}}_{0,\mathbf{\frac{1}{4}}+\epsilon}\dashrightarrow\overline{\Gamma_{\mathbf{\frac{1}{4}}}\backslash\cB_5}^{\tor}$ extends to an isomorphism $\alpha$ away from closed subsets of codimension at least two. Consider the two blow ups
\begin{align*}
\overline{\mathrm{M}}_{0,\mathbf{\frac{1}{4}}+\epsilon}
\longrightarrow 
\overline{
\Gamma_{\mathbf{\frac{1}{4}}}
\backslash \cB_5
}^{\bb}
& &
\overline{
\Gamma_{\mathbf{\frac{1}{4}}}
\backslash
\cB_5
}^{\tor}
\longrightarrow
\overline{ \Gamma_{\mathbf{\frac{1}{4}}}
\backslash
\cB_5
}^{\bb}.
\end{align*}
By Propositions~\ref{relativeamplenesshassettside} and \ref{relativeamplenesstoroidalside}, negative the exceptional divisors of these blow ups correspond to relatively ample line bundles $\mathcal{L}_1,\mathcal{L}_2$ on the Hassett moduli space and the toroidal compactification respectively. Finally, observe that $\alpha^*(\mathcal L_2)=\mathcal L_1$ because $\alpha$ sends the exceptional divisors in Hassett's compactification to the respective ones in the toroidal compactification. Now, all the hypotheses of Lemma~\ref{kovacs} are satisfied, so the isomorphism follows.

A similar argument hold for $n=12$, but with the following change. Let $E$ be the disjoint union of the exceptional divisors of $\overline{\mathrm{M}}_{0,\mathbf{\frac{1}{6}}+\epsilon}\rightarrow\overline{
\Gamma_{\mathbf{\frac{1}{6}}}  \backslash 
\cB_{9}
}^{\bb}$ 
and denote by $H$ the quotient of $E$ by $S_{12}$. We know that $-E$ is relatively ample by Proposition~\ref{relativeamplenesshassettside}, hence $-E|_E$ is ample. This implies that $-H|_H$ is ample by \cite[Corollary 1.2.28]{Laz04}, hence $-H$ is relatively ample.
\end{proof}


\subsection{Extension to the remaining Deligne-Mostow cases}
\label{sec:OtherDMCases}

The following result allows us to reduce the proof of Theorem~\ref{maintheorem} to the Gaussian and Eisenstein cases $n=8$ and $n=12$ with symmetric weights, that have just been covered in \S\ref{prooftheoremintwoancestralcases}.

\begin{lemma}[{\cite[Lemma 5.4]{YZ18}}]
\label{lemma:extSubcase}
Let $f_1\colon Z_1 \rightarrow Y$ and 
$f_2\colon Z_2 \rightarrow Y$ be finite morphisms between irreducible algebraic varieties. Suppose $Z_1$, $Z_2$ are normal. Moreover, assume there exist Zariski open subsets $U_i\subseteq Z_i$, $i = 1,2$, with a biholomorphic map 
$g\colon U_1 \rightarrow U_2$ such that 
$f_1 = f_2  \circ g$. Then $g$ extends to an algebraic isomorphism $Z_1 \rightarrow Z_2$.
\end{lemma}

\begin{proof}[Proof of Theorem~\ref{maintheorem}]
We already proved the result for the two ancestral cases. Next, we consider one of the cases with $h$ points and weights $\mathbf{w}$ contained in one of the two ancestral cases with $n$ points and weights $\mathbf{ \frac{2}{n}}$.

By allowing points to collide, we can view $\overline{\mathrm{M}}_{\mathbf{w}+\epsilon}$ as a sub-moduli space of $\overline{\mathrm{M}}_{0,\mathbf{\frac{2}{n}}+\epsilon}$. This could be done in different ways, so we fix a choice. According to this choice, identify $(( \mathbb{P}^1)^{h})^{s}/\!/_{ \mathbf{w}}\SL_2$ as a sub-quotient of $(( \mathbb{P}^1)^{n})^{s}/\!/_{ \mathbf{\frac{2}{n}}}\SL_2$. Let $G$ (respectively, $K$) be the finite group giving the Deligne-Mostow isomorphism $((( \mathbb{P}^1)^{h})^{s}/\!/_{ \mathbf{w}}\SL_2)/G\cong \Gamma_{\mathbf{w}}\backslash\mathbb{B}_{h-3}$ (respectively, $((( \mathbb{P}^1)^{n})^{s}/\!/_{ \mathbf{\frac{2}{n}}}\SL_2)/K\cong \Gamma_{\mathbf{w}}\backslash\mathbb{B}_{n-3}$). Recall that $G$ is the trivial group if $\mathbf{w}$ satisfies INT, or $G=S_m$ for some positive integer $m\leq n$ if $\mathbf{w}$ satisfies $\Sigma$INT. For the same reason, $K$ is trivial for $n=8$ and $K=S_{12}$ for $n=12$. Notice that the inclusion $\overline{\mathrm{M}}_{\mathbf{w}+\epsilon}\subseteq\overline{\mathrm{M}}_{0,\mathbf{\frac{2}{n}}+\epsilon}$ induces a finite morphism $\overline{\mathrm{M}}_{\mathbf{w}+\epsilon}/G\rightarrow\overline{\mathrm{M}}_{0,\mathbf{\frac{2}{n}}+\epsilon}/K$. Then we have the following commutative diagram:
\begin{center}
\begin{tikzpicture}[>=angle 90]
\matrix(a)[matrix of math nodes,
row sep=2.5em, column sep=4em,
text height=2.0ex, text depth=1.0ex]
{((( \mathbb{P}^1)^{h})^{s}/\!/_{ \mathbf{w}}\SL_2)/G&\overline{\mathrm{M}}_{0,\mathbf{w}+\epsilon}/G&\overline{\mathrm{M}}_{0,\mathbf{\frac{2}{n}}+\epsilon}/K\\
\Gamma_{\mathbf{w}} \backslash \cB_{h-3}&
\overline{\Gamma_{\mathbf{w}}\backslash\cB_{h-3}}^{\tor}&\overline{\Gamma_{\mathbf{\frac{2}{n}}}\backslash\cB_{n-3}}^{\tor},\\};
\path[right hook->] (a-1-1) edge node[]{}(a-1-2);
\path[->] (a-1-2) edge node[]{}(a-1-3);
\path[->] (a-1-1) edge node[left]{$\cong$}(a-2-1);
\path[dashed,->] (a-1-2) edge node[]{}(a-2-2);
\path[->] (a-1-3) edge node[right]{$\cong$}(a-2-3);
\path[right hook->] (a-2-1) edge node[]{}(a-2-2);
\path[->] (a-2-2) edge node[]{}(a-2-3);
\end{tikzpicture}
\end{center}
where the lower-right map between toroidal compactifications is a finite morphism which comes from combining \cite[Definition 13 and Theorem 5]{Dor04a} and \cite[Proposition 3.4]{Har89}.
In conclusion, Lemma~\ref{lemma:extSubcase} implies that the middle vertical arrow extends, so Theorem \ref{maintheorem} follows for all the cases contained within the ancestral ones (see Eisenstein and Gaussian cases in Tables~\ref{table:DMcasesEisenstein} and \ref{table:DMcasesGaussian} respectively).
\end{proof}


\section{Hodge theoretic interpretation for eight points in the line} \label{sec:Matt}

The purpose of the following section is to give a precise mixed-Hodge-theoretic description of the extension of the period map associated to 8 points in $\mathbb{P}^1$.  That is, we show how to reinterpret the isomorphism $\overline{\mathrm{M}}_{0,\mathbf{\frac{1}{4}}+\epsilon}\overset{\cong}{\longrightarrow}
\overline{\Gamma_{\mathbf{\frac{1}{4}}} \backslash\cB_5}^{\tor}$
as an extended global Torelli result, matching the geometrically-modular description of the source with the description of the target in terms of equivalence-classes of mixed Hodge structures.  

Throughout, we shall denote field extensions by subscripts, viz. $V_{\mu,\mathbb{C}}$ rather than $V_{\mu}\otimes \mathbb{C}$.

\subsection{Period map}\label{sec:Matt1}
Writing $\mathcal{M}:=\mathrm{M}_{0,8}=\left\{\underline{x}=(x_1,\ldots,x_8)\in (\mathbb{P}^1)^8\setminus\cup\{x_i=x_j\}\right\}/\mathrm{PGL}_2(\mathbb{C})$, consider the family $\mathcal{C} \overset{\pi}{\to} \mathcal{M}$ of genus-nine curves, with fiber
$$
C_{\mu}:=\pi^{-1}(\mu)\cong \overline{\{y^4=\Pi_{j=1}^8 (x-x_j)\}}\subseteq \mathbb{WP}[1:1:2]
$$
over $\mu=[\underline{x}] \in \mathcal{M}$.  This family has an automorphism by $\rho\colon y\mapsto iy$; and denoting the quotient of $C_{\mu}$ by $\rho^2$ by
$$
B_{\mu}:=\overline{\lbrace \tilde{y}^2 =\Pi_{j=1}^8 (x-x_j) \rbrace} \subseteq \mathbb{WP}[1:1:4],
$$
we have the decomposition
\begin{align*}
H^1 \left(C_{\mu}, \mathbb{Q} \right)    
\cong
H^1 \left(B_{\mu}, \mathbb{Q} \right)    
\oplus 
H^1 \left( C_{\mu}, \mathbb{Q}\right)^{-}    
\end{align*}
into $\pm 1$-eigenspaces for $\rho^2$.  Setting $\tilde{V}_{\mu}:= H^1 \left( C_{\mu}, \mathbb{Q}\right)^{-}$, we have the further decomposition
$$
\tilde{V}_{\mu,\mathbb{Q}(i)} \cong V_{\mu} \oplus \overline{V}_{\mu}
$$
into $\pm i$-eigenspaces for $\rho$, so that $\tilde{V}_{\mu}\cong \mathrm{Res}_{\mathbb{Q}(i)/\mathbb{Q}}V_{\mu}$.
One easily checks that the Hodge numbers 
$\left( h^{0,1}, h^{1,0} \right)$
of the Hodge structures
$H^1 \left(B_{\mu}\right)    $, $V_{\mu}$, and $ \overline{V}_{\mu}$ are 
$(3,3)$, $(5,1)$, and $(1,5)$ respectively; 
in particular, $V_{\mu}^{1,0} \subseteq V_{\mu,\mathbb{C}}$ is spanned by  $\omega = \frac{dx}{y}$.

By \cite[Proposition 2.2]{MT04} and its preliminary discussion, the $\rho^2$-anti-invariant part of homology is 
$
H_1(C_{\mu}, \mathbb{Z})^{-}
\cong
\mathbb{Z}^{\oplus 12},
$
with generators 
$\{\textsc{a}_1, \ldots, \textsc{a}_6, \textsc{b}_1, \ldots, \textsc{b}_6 \}$
where 
$\textsc{b}_j = \rho_{*}(\textsc{a}_j)$ and 
$-\textsc{a}_j=\rho_{*}(\textsc{b}_j)$.
The Poincar\'e duality isomorphism 
$\textsc{P.D.}\colon H_1(C_{\mu}, \mathbb{Z})\overset{\cong}{\to}H^1(C_{\mu},\mathbb{Z})$ (defined by $\gamma\mapsto\delta_{\gamma}:=$ current of integration over $\gamma$) sends
$H_1(C_{\mu},\mathbb{Z})^{-}\overset{\cong}{\to} H^1(C_{\mu}, \mathbb{Z})^{-} :=
H^1 \left( C_{\mu}, \mathbb{Q} \right)^{-} \cap H^1(C_{\mu}, \mathbb{Z})$.
So we obtain identifications of 
\begin{align*}
\widetilde{V}_{\mu, \mathbb{Z}}
:=
H^1(C_{\mu}, \mathbb{Z})^{-}
=
 \text{Res}_{\mathbb{Z}[i]/\mathbb{Z}}
 \left(
 V_{\mu, \mathbb{Z}}
 \right),
\end{align*}
where the last equality defines $ V_{\mu, \mathbb{Z}}$, with
\begin{align*}
\mathbb{Z}^{\oplus 12}
=
 \text{Res}_{\mathbb{Z}[i]/\mathbb{Z}}
 \left(
\mathbb{Z}[i]^{\oplus 6}
 \right).
\end{align*}
In particular, a $\mathbb{Z}[i]$-basis of 
$ V_{\mu, \mathbb{Z}}$ is given by 
$\{\textsc{a}_1, \ldots, \textsc{a}_6 \}$ (i.e. $\delta_{\textsc{a}_1},\ldots,\delta_{\textsc{a}_6}$). 
The action of $\rho^*$ on $ V_{\mu, \mathbb{Z}}$ (and on $\tilde{V}_{\mu,\mathbb{Z}}$) is identified with  multiplication by $(-i)$ on $\mathbb{Z}[i]^{\oplus 6}$, and so 
$\textsc{b}_j$ identifies with $i\textsc{a}_j$.

Now let $\tilde{\mathcal{V}} \to \mathcal{M}$ be the vector bundle with fibers $\tilde{V}_{\mu,\mathbb{C}}$, and
$\widetilde{\mathbb{V}}$ [resp. $\mathbb{V}$] be the 
$\mathbb{Z}$- [resp. $\mathbb{Z}[i]$-] local system with fibers $\tilde{V}_{\mu,\mathbb{Z}}$ [resp. $V_{\mu,\mathbb{Z}}$] and intersection pairing
$Q\colon \tilde{\mathbb{V}} \otimes \tilde{\mathbb{V}}\to \mathbb{Z}$.
Writing $\tilde{V} (\cong \mathbb{Q}^{\oplus 12})$ [resp. $V(\cong \mathbb{Q}(i))$] for a fixed fiber of $\tilde{\mathbb{V}}_{\mathbb{Q}}$ [resp. $\mathbb{V}_{\mathbb{Q}(i)}^{\oplus 6}$], the $\mathbb{Z}$-VHS
$(\tilde{\mathcal{V}}, F^{\bullet} \tilde{\mathcal{V}},\tilde{\mathbb{V}},\nabla,Q)$
has Hodge group 
$$
G = \text{Res}_{\mathbb{Q}(i)/\mathbb{Q}} \left(\mathrm{GL}(V) \right)
\cap
\mathrm{Sp}(\widetilde V, Q),
$$
which is a $\mathbb{Q}$-form of $U(5,1)$.

\begin{definition}
We take
\begin{align*}
\Phi\colon \mathcal{M} \longrightarrow   
\Gamma \backslash G(\mathbb{R})/G^0(\mathbb{R})   
\; \cong \;
\Gamma \backslash 
U(5,1)
/ \left(U(5) \times U(1)  \right) 
\; \cong \;
\Gamma \backslash 
\mathbb{B}_5
\end{align*}
to be the period map associated to this $\mathbb{Z}$-VHS, where $\Gamma \leq G(\mathbb{Z})$ is the monodromy group of $\widetilde{\mathbb{V}}$.  Note that this is just a detailed description of $\Phi_{\mathbf{\frac{1}{4}}}$ from $\S$\ref{s1.1}.\end{definition}

There are two important reinterpretations of this ``$\rho^2$-anti-invariant'' period map. First, we may view 
$\Phi(\mu)$ as a $\Gamma$-equivalence-class of polarized $\mathbb{Z}[i]$-Hodge structures on 
$V = \mathbb{Q}[i]^{\oplus6}$:  the underlying lattice is just $V_{\mathbb{Z}} = \mathbb{Z}[i]^{\oplus 6}$; and by making the identification $V_{\mathbb{Z}} \overset{\cong}{\to} V_{\mu,\mathbb{Z}}$ via $\{\textsc{a}_j\}$ as above (well-defined up to $\Gamma$), we obtain a Hodge flag
$F^{\bullet}_{\mu} V_{\mathbb{C}}$ by pulling back the line $F^{1}V_{\mu,\mathbb{C}}= V^{1,0}_{\mu}
\subset V_{\mu, \mathbb{C}}
$. 
The polarization is given by the $\mathbb{Z}[i]$-Hermitian form $h:V_{\mathbb{Z}} \times V_{\mathbb{Z}} \longrightarrow \mathbb{Z}[i]$ defined (using the identification 
$V_{\mathbb{Z}} \overset{\cong}{\to} \tilde{V}_{\mathbb{Z}} \cong \mathbb{Z}^{\oplus 12}$ via $v \mapsto \tilde v$) by
$$
h(v,w):= Q(\tilde{v}, \rho \tilde{w})- i Q(\tilde{v}, \tilde{w}),
$$
with signature $(5,1)$ and associated matrix 
\begin{align*}
[h]=\begin{pmatrix}
-2& 1-i & &0 & 0 & 0 \\
1+i & -2  & 1-i & 0 & 0 & 0  \\
0 & 1+i & -2  & 1-i & 0 & 0  \\
0 & 0 & 1+i & -2  & 1-i & 0\\
0 & 0 & 0 & 1+i & -2  & 1-i  \\
0 & 0 & 0 & 0 & 1+i & -2   \\
\end{pmatrix}.   
\end{align*}
By \cite[Thm 4.1]{MY93}, $\Gamma$ is precisely the arithmetic subgroup
$$
U(h; (1-i))
:=
\left\lbrace
g \in \mathrm{GL}_6(\mathbb{Z}[i]) \; | \;
{}^t \bar{g}[h]g=[h] ,\; g \equiv \mathrm{I}_6 \; \text{mod}\; (1-i)
\right\rbrace
$$
of $G(\mathbb{Q})$ associated to the ideal $(1-i)\subseteq \mathbb{Z}[i]$.

A second interpretation of $\Phi( \mu )$ is as the modulus of the abelian variety
$$
J(C_{\mu})^{-}:=\Ext^{1}_{\mathrm{MHS}}( \mathbb{Z}, \tilde{V}_{\mu,\mathbb{Z}}(1))\cong\frac{(\tilde{V}_{\mu}^{1,0})^{\vee}}{H_1\left( C_{\mu}, \mathbb{Z} \right)^{-}},
$$
which is naturally a sub-abelian variety of the Jacobian $$
J(C_{\mu}):=\Ext^{1}_{\mathrm{MHS}}\left( \mathbb{Z}, H^1\left(C_{\mu}, \mathbb{Z}(1) \right)\right).
$$
\begin{remark}
In \cite[Sec 2]{MT04}, $J(C_{\mu})^{-}$ is called the ``Prym variety'', although $C_{\mu}\overset{2:1}{\twoheadrightarrow} B_{\mu}$ is not \'etale.  We remark that in its definition, one is not dividing out by \emph{all} periods of the $\rho^2$-anti-invariant holomorphic 1-forms, but only by the finite-index sublattice of periods over $\rho^2$-anti-invariant cycles.
\end{remark}

\subsection{Geometric boundary}\label{sec:Matt2}
By a \emph{codimension-one stratum} in a compactification $\bar{X}\supseteq X$ we shall mean the open part $\mathcal{S}$ of an irreducible boundary divisor $\bar{\mathcal{S}}$; that is, $\mathcal{S}=\bar{\mathcal{S}}\setminus \{\bar{\mathcal{S}}\cap \mathrm{sing}(\bar{X}\setminus X)\}$.  In the KSBA compactification $\overline{\mathcal{M}}^{\mathrm{KSBA}}:=\overline{\mathrm{M}}_{0, \mathbf{\frac{1}{4}} + \epsilon}$ of $\mathcal{M}$, there are two types of codimension-one strata:
\begin{itemize}
    \item[(A)] those arising from a collision $x_i=x_j$ ($i\neq j$), and parametrizing (a $\mathbb{P}^1$ with) seven ordered points with one of multiplicity two; and
    \item[(B)] those parametrizing two copies of $\mathbb{P}^1$ glued at a point $p$, with four ordered points on each component $\mathbb{P}^1\setminus\{p\}$.
\end{itemize}
There are $\binom{8}{2}=28$ strata of type (A) and $\binom{8}{4}/2=35$ strata of type (B).  The closure of each type (B) stratum arises as the ``KSBA replacement'' of a pair of colliding 4-tuples; and taken together, they constitute the exceptional divisor of the morphism $\overline{\mathrm{M}}_{0,\mathbf{\frac{1}{4}}+\epsilon} \to \overline{\left(\Gamma \backslash \mathbb{B}_5\right)}^{\mathrm{bb}}$.  

The VHS $\tilde{\mathcal{V}}$ degenerates along each of these strata, and we write $T=e^N T_{\mathrm{ss}}$ for the Jordan decomposition of the local monodromy operator into unipotent and (finite) semi-stable parts.  In this subsection, we shall describe the LMHSs (limiting mixed Hodge structures) along both types of components, together with the action on them of $T_{\mathrm{ss}}$, $N$ and $\rho$.  This is done by computing the decomposition $$\sigma_f = \sigma_f^{-i} + \sigma_f^{-1} + \sigma_f^{i}$$ of the usual spectrum (of an isolated singularity locally described by some $f(x,y)=y^4+p(x)$ specified below) into \emph{eigenspectra} under the action of $\rho\colon y\mapsto iy$, followed by a base-change.  We refer the reader to \cite[$\S\S$1-2]{KL20} for a description of how the spectra are calculated and the vanishing cycle sequence used to relate them to the LMHS.

For type (A), start with the tacnode degeneration 
$y^4+x^2=t$ and base-change by $t\mapsto t^2$. (This models the degeneration of $C_{\mu}$ as $x_i\sim t$ and $x_j \sim -t$ collide; the base-change is necessary to preserve the order.  While the resulting total space is singular, the base-change does not affect the LMHS.)  The tacnode spectrum is 
\begin{align*}
\sigma^{-i} = [\frac{3}{4}] 
& &
\sigma^{-1} = [1] 
& &
\sigma^{i} = [\frac{5}{4}] 
\end{align*}
and the effect of the base-change is to square the action of $T_{ss}$. This leads to LMHS types
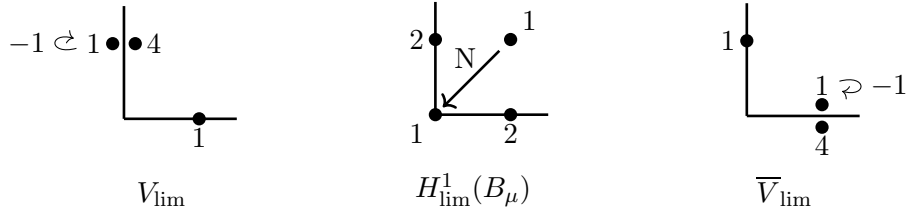
\begin{align*}
\begin{tikzpicture}[scale=0.5]
   	\node at (-0.8,2) (foo) {$1$};
    \path[->] (foo) edge  [loop left] node {$-1$} ();
    \draw[-,line width=1.0pt] (0,0) -- (0,3);
    \draw[-,line width=1.0pt] (0,0) -- (3,0);
	\fill (2,0) circle (5pt);
	\node at (2,-0.5)  {$1$};		
	\fill (-0.3,2) circle (5pt);
	\fill (0.3,2) circle (5pt);	
	\node at (0.8,2) {$4$};
	\node at (1,-2) {$V_{\lim}$};					
\end{tikzpicture}
& &
\begin{tikzpicture}[scale=0.5]
    \draw[-,line width=1.0pt] (0,0) -- (0,3);
    \draw[-,line width=1.0pt] (0,0) -- (3,0);
	\fill (2,0) circle (5pt);
	\fill (2,2) circle (5pt);
	\fill (0,0) circle (5pt);
	\fill (0,2) circle (5pt);	
    \draw[->,line width=1.0pt] (1.7, 1.7) -- (0.2,0.2);	
	\node at (2.5,2.5) {$1$};				
	\node at (-0.5,-0.5) {$1$};				
	\node at (2,-0.5) {$2$};				
	\node at (-0.5,2) {$2$};
	\node at (1,-2) {$H_{\lim}^1(B_{\mu})$};		\node at (0.8,1.5) {\text{N}};				
\end{tikzpicture}
& &
\begin{tikzpicture}[scale=0.5]
   	\node at (2,0.8) (foo) {$1$};
    \path[->] (foo) edge  [loop right] node {$-1$} ();
    \draw[-,line width=1.0pt] (0,0) -- (0,3);
    \draw[-,line width=1.0pt] (0,0) -- (3,0);
	\fill (2,-0.3) circle (5pt);
	\fill (2,0.3) circle (5pt);	
	\fill (0, 2) circle (5pt);
	\node at (-0.5,2) {$1$};			
	\node at (2,-0.8) {$4$};
	\node at (1,-2) {$\overline{V}_{\lim}$};
\end{tikzpicture}
\end{align*}
where the non-trivial $T_{ss}$ action is shown with an arrow.  In particular, from the left-hand picture, we find that on $V$ each $T_{ss}$ has order two. These are nothing but the reflections $\alpha(ij)$ in \cite[Prop 3.1]{MY93} which are shown to generate $\Gamma$.

For type (B), we start with two copies of $y^4+x^4=t$
and base-change by $t \mapsto t^4$.  (This reflects how the type (B) components arise, see above; once again the base-change is done to preserve the order.)
The pre-base-change spectra are
\begin{align*}
\sigma^{-i}
=
[\frac{1}{2}] + \big[\frac{3}{4}\big] +  \big[ 1
\big]
& &
\sigma^{-1}
=
[\frac{3}{4}] + [1] + [\frac{5}{4}]
& &
\sigma^{i}
=
[1] + \big[\frac{5}{4}\big] +  \big[ \frac{3}{2}
\big]
\end{align*}
for each of the two singular points, so that the base-change renders the monodromy unipotent (i.e. $T_{ss} = I$).  The sum of the $\mathrm{Gr}^W_1$'s in the resulting LMHS types
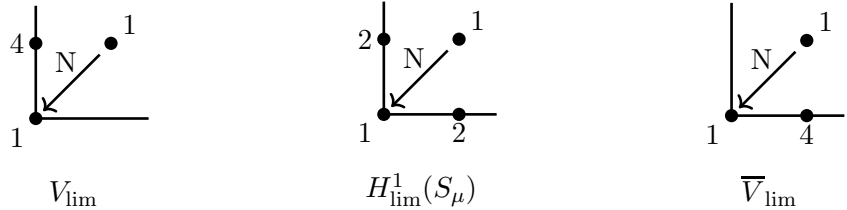
\begin{align*}
\begin{tikzpicture}[scale=0.5]
    \draw[-,line width=1.0pt] (0,0) -- (0,3);
    \draw[-,line width=1.0pt] (0,0) -- (3,0);
	\fill (2,2) circle (5pt);
	\fill (0,0) circle (5pt);
	\fill (0,2) circle (5pt);	
    \draw[->,line width=1.0pt] (1.7, 1.7) -- (0.2,0.2);	
	\node at (2.5,2.5) {$1$};				
	\node at (-0.5,-0.5) {$1$};				
	\node at (-0.5,2) {$4$};
	\node at (0.8,1.5) {\text{N}};	
	\node at (1,-2) {$V_{\lim}$};			
\end{tikzpicture}
& &
\begin{tikzpicture}[scale=0.5]
    \draw[-,line width=1.0pt] (0,0) -- (0,3);
    \draw[-,line width=1.0pt] (0,0) -- (3,0);
	\fill (2,0) circle (5pt);
	\fill (2,2) circle (5pt);
	\fill (0,0) circle (5pt);
	\fill (0,2) circle (5pt);	
    \draw[->,line width=1.0pt] (1.7, 1.7) -- (0.2,0.2);	
	\node at (2.5,2.5) {$1$};				
	\node at (-0.5,-0.5) {$1$};				
	\node at (2,-0.5) {$2$};				
	\node at (-0.5,2) {$2$};
	\node at (1,-2) {$H_{\lim}^1(S_{\mu})$};			\node at (0.8,1.5) {\text{N}};			
\end{tikzpicture}
& &
\begin{tikzpicture}[scale=0.5]
    \draw[-,line width=1.0pt] (0,0) -- (0,3);
    \draw[-,line width=1.0pt] (0,0) -- (3,0);
	\fill (2,0) circle (5pt);
	\fill (2,2) circle (5pt);
	\fill (0,0) circle (5pt);
    \draw[->,line width=1.0pt] (1.7, 1.7) -- (0.2,0.2);	
	\node at (2.5,2.5) {$1$};				
	\node at (-0.5,-0.5) {$1$};				
	\node at (2,-0.5) {$4$};				
	\node at (1,-2) {$ \overline{V}_{\lim}$};			\node at (0.8,1.5) {\text{N}};			
\end{tikzpicture}
\end{align*}
comprise two copies of $H^1(\mathfrak{F})$, where $\mathfrak{F}$ is the Fermat quartic plane curve (i.e.  $x^4 + y^4 = t^4$); and the extension classes (in the category MHS) are all trivial.

But this describes a 0-dimensional locus in the GIT compactification.  The KSBA replacement lets us deform \emph{off} this Fermat/trivial-extension locus. The decomposition under $\rho$,  the LMHS \emph{types}, and the fact that the monodromy is unipotent \emph{do not change}; but the singular object parametrized by the the given type (B) stratum is now a pair of genus-three curves 
\begin{align*}
D_{\nu}^{(j)} := \lbrace Y^4 = \Pi_{k=1}^4(X - \xi_k^{(j)}Z) \rbrace \subseteq \mathbb{P}^2,
& 
~(j \in \{1,2\})
\end{align*}
glued along four points 
$\{q_{\ell} = [i^{\ell}:1:0]\}_{\ell=0}^3$  
with $\rho$ acting by $Y \mapsto iY$ (and 
$q_{\ell} =\rho^{\ell}(q_0)$), to form 
$D_{\nu} = D_{\nu}^{(1)} \cup D_{\nu}^{(2)} $.  

\begin{remark}
Here ``$\nu$'' is simply the parameter on the type (B) stratum, for which we shall not need a precise geometric description.  However, one sees at once that it must ``take in'' more than just the isomorphism classes of $D_{\nu}^{(1)}$ and $D_{\nu}^{(2)}$ (i.e. the cross-ratios $\textsc{cr}(\underline{\xi}^{(1)})$ and $\textsc{cr}(\underline{\xi}^{(2)})$), since the strata are 4-dimensional.  Moreover, the limiting period map $\Phi_{\lim}$ for $\tilde{\mathcal{V}}$ won't even \emph{directly} see the isomorphism classes of $D_{\nu}^{(1)}$  and $D_{\nu}^{(2)}$, the variation in which is recorded by the $\rho^2$-\emph{invariant} part of the limiting period map (ignored by $\Phi_{\lim}$).  More on this in $\S$\ref{sec:Matt5}.
\end{remark}

\subsection{Extending the period map}\label{sec:Matt3}
We now turn to the existence of an extension
\begin{align*}
\overline{\Phi}:
\overline{\mathcal{M}}^{\text{KSBA}} \to 
\overline{
\left(
\Gamma \backslash \mathbb{B}_5    
\right)}^{\tor}
\end{align*}
of $\Phi$, together with a Hodge-theoretic interpretation at the boundary (in codimension one only).  More precisely, denote by $\overline{\mathcal{M}}^{\circ}\subsetneq\overline{\mathcal{M}}^{\text{KSBA}}$ the union of $\mathcal{M}$ and all codimension-one strata (of type (A) or (B)); and simply define $\overline{\Phi}$ to be the isomorphism $\overline{\Phi}_{\mathbf{\frac{1}{4}}}$ from Theorem \ref{maintheorem}, which we already know restricts to $\Phi$ on $\mathcal{M}$ (see $\S$\ref{s1.1}).  Then we claim:

\begin{proposition}\label{prop:Matt3}
The restriction $\overline{\Phi}|_{\mathcal{S}}\colon \nu \mapsto \overline{\Phi}(\nu)$ to a codimension-one stratum computes the LMHS of $\tilde{\mathcal{V}}$ at $\nu$, modulo the action of $e^{\mathbb{C}N}$ \textup{(}choice of local parameter\textup{)} and $\Gamma$. 
\end{proposition}

\begin{proof}
One way to see this is by Theorem B of \cite{KU08}, as we demonstrate below.  Rather than offering a narrow proof, we spell out in some detail \emph{how} the extension records the LMHS along $\mathcal{S}$, concentrating on strata of type (B).  

First note that, by \cite[Sec 6]{KP16}, the results of \cite{KU08} (stated for period domains) carry over to the Mumford-Tate domain setting, which includes all connected Shimura varieties.  Moreover, in this ``classical'' case, their compactifications coincide with those of \cite{AMRT75}, see \cite[Rmk. 8.2.7]{KU08}.
Since $G_{\mathbb{R}} \cong U(5,1)$ has real rank 1, one obtains a complete fan $\Sigma$ 
 in the sense of Kato-Usui by taking the $\mathrm{Ad}(\Gamma)$-orbit of the monodromy logarithms 
$N \in \mathrm{End}(\widetilde{V}, Q, \rho)$ 
attached to all $35$ type (B) strata in $\S$\ref{sec:Matt2}. (One needs to parallel-transport each $N$ from the nearby fiber of $\mathrm{End}(\widetilde{\mathbb{V}})$ at each boundary component to the fiber identified with $\mathrm{End}(\widetilde{V})$).  In particular, we have $|\Gamma\backslash\Sigma|=35$.

As the theory in \cite{KU08} is stated in terms of a neat group, we first pass to a normal neat subgroup $\Gamma_0$ of $\Gamma$. Writing $(\Gamma_0)_N \leq \Gamma_0$ for the centralizer of $N$, the toroidal compactification is 
\begin{align*}
\overline{
\left( 
\Gamma_0 \backslash \mathbb{B}_5
\right)   
}^{\Sigma}
=
\left(
\Gamma_0 \backslash \mathbb{B}_5
\right)
\amalg
\coprod_{
\alpha \in 
\Gamma_0 \backslash \Sigma 
}
\left. (\Gamma_0)_{N_{\alpha}} \right\backslash B(N_{\alpha})\; ,
\end{align*}
where the Hodge-theoretic boundary component $B(N_{\alpha})$ is defined as in \cite[$\S$7]{KP16}, as a set of $N_{\alpha}$-nilpotent orbits $e^{\mathbb{C}N_{\alpha}}F^{\bullet}\subseteq \check{\mathbb{B}}_5$ in the compact dual (or equivalently, limiting MHS modulo choice of local coordinate).
More precisely, let 
$Z_N = M_{N} \rtimes G_N$ denote the centralizer of $N$, where $M_N$ is the unipotent radical and $G_N$ is reductive; inside this, we have the Hodge group $M_{B(N)} = M_N \rtimes H_N$
of the boundary component (with $H_N \leq G_N$ again reductive).  Defining
$\mathfrak{M}_{B(N)} := M_N(\mathbb{C}) \rtimes H_N(\mathbb{R})$, and 
$\mathcal{K}_N \leq \mathfrak{M}_{B(N)}$ the stabilizer of a $\mathbb{Q}$-split base-point $F_0^{\bullet}$ (i.e. special limiting Hodge flag), we may present the toroidal boundary component as a fibration \begin{align*}
(\Gamma_0)_N \backslash B(N)
=
(\Gamma_0)_Ne^{\mathbb{C}N}\backslash \mathfrak{M}_{B(N)} / \mathcal{K}_N 
\;\;{\relbar\joinrel\relbar\joinrel\twoheadrightarrow}\;\;
\overline{(\Gamma_0)_N}
\backslash H_N(\mathbb{R}) / \overline{\mathcal{K}_N}
=
\overline{(\Gamma_0)_N} \backslash  D(N)
\end{align*}
over a Baily-Borel boundary component.  This map may be interpreted simply as forgetting all extension classes in a MHS on the left-hand side; see [loc. cit.] and \cite[Thm. 5.21]{KP14} for details.

For the specific case at hand (for type (B)), $H_N \cong U(1)$ because the LMHS on $\widetilde{V}$ has associated-graded $\mathrm{Gr}^{W(N)}_{\bullet}\tilde{V}$ a direct sum of Tate-HS and $\mathbb{Q}(i)$-CM-HS. Hence $D(N)$ is a point, while 
$\Gamma_{0, N} \backslash B(N)$ is $\mathrm{CM}$-abelian fourfold parametrizing extension classes in the LMHS $\tilde{V}_{\lim}=(\tilde{V}_{\mathbb{Z}},F^{\bullet}_{\lim}\tilde{V}_{\mathbb{C}},W(N)_{\bullet}\tilde{V})$ with $\rho$-action.  Indeed, viewing $\tilde{V}_{\lim}$ as $\mathrm{Res}_{\mathbb{Z}[i]/\mathbb{Z}}$ of
\begin{align*}
\begin{tikzpicture}[scale=0.5]
	\node at (-2,1) {$V_{\lim}\;=$};		
    \draw[-,line width=1.0pt] (0,0) -- (0,3);
    \draw[-,line width=1.0pt] (0,0) -- (3,0);
	\fill (2,2) circle (5pt);
	\fill (0,0) circle (5pt);
	\fill (0,2) circle (5pt);	
    \draw[->,line width=1.0pt] (1.7, 1.7) -- (0.2,0.2);	
	\node at (2.5,2.5) {$1$};				
	\node at (-0.5,-0.5) {$1$};				
	\node at (-0.5,2) {$4$};
	\node at (1.4,0.7) {\text{N}};	
\end{tikzpicture}
\end{align*}
(in which $\rho$ replaces multiplication by $-i$), we have $\mathrm{Ext}^1_{\mathbb{Z}[\rho]\text{-MHS}}(\mathbb{Z}[\rho](-1),\mathrm{Gr}^W_1 \tilde{V}_{\lim})\cong $
$$
\mathrm{Ext}^1_{\mathbb{Z}[i]\text{-MHS}}\left(
\mathbb{Z}[i](-1), 
\mathrm{Gr}_1^{W}V_{\lim}
\right)
\cong
\frac{Gr_1^W V_{\mathbb{C}}}{
F^1 Gr_1^{W} V_{\mathbb{C}}
+
Gr^W_1 V_{\mathbb{Z}}
}
\cong
\frac{Gr_1^W V_{\mathbb{C}}}{
Gr^W_1 V_{\mathbb{Z}}
}
\cong 
\frac{\mathbb{C}^4}{\mathbb{Z}[i]^{\oplus 4}
},
$$
see $\S$\ref{sec:Matt4} for more details.

On the domain side, we let $\mathcal{M}_{0}$ 
[resp. $\overline{\mathcal{M}}_{0},\overline{\mathcal{M}}_0^{\circ}$] 
be the normalization of $\mathcal{M}$ [resp. $\overline{\mathcal{M}},\overline{\mathcal{M}}^{\circ}$]
in the function field of  $\Gamma_0 \backslash \mathbb{B}_5$. Since the singularities of $\overline{\mathcal{M}}_0$ are in codimension $\geq 2$ (over singularities of $\overline{\mathcal{M}}\setminus \mathcal{M}$), $\overline{\mathcal{M}}_0^{\circ}$ is smooth. 
The resulting period map $\Phi_0\colon \mathcal{M}_0\longrightarrow \Gamma_0 \backslash \mathbb{B}_5$ therefore extends to 
$$
\overline{\Phi}_0^{\circ}:
\overline{\mathcal{M}}_0^{\circ}
\longrightarrow
\overline{\left(
\Gamma_0 \backslash \mathbb{B}_5    
\right)}^{\Sigma}
$$
by \cite[$\S$4.2.1 (Thm. B)]{KU08}, which says moreover that for $\nu_0 \in \overline{\mathcal{M}}_0^{\circ} \setminus \mathcal{M}_0$,
the point $\overline{\Phi}^{\circ}_0(\nu_0)$ records the LMHS of $\widetilde{\mathcal{V}}$ at $v$ in the corresponding 
$(\Gamma_0)_{N_{\alpha}} \backslash B(N_{\alpha})$.
Since $\Gamma_0 \trianglelefteq \Gamma$, and LMHSs are invariant under base-change, we can now quotient both sides by $\Gamma/\Gamma_0$ to get a codimension-one extension
$$
\overline{\Phi}^{\circ}:
\overline{\mathcal{M}}^{\circ}
\longrightarrow
\overline{
\left(
\Gamma \backslash \mathbb{B}_5
\right)
}^{\Sigma}
=
\left(
\Gamma \backslash \mathbb{B}_5
\right)
\amalg
\coprod_{
\bar{\alpha} \in 
\Gamma \backslash \Sigma 
}
\left. \Gamma_{N_{\bar{\alpha}}} \right\backslash B(N_{\bar{\alpha}})
$$
with the LMHS interpretation along strata $\mathcal{S}$ of type (B).  (For strata of type (A) we get this too, with the LMHS pure hence belonging to $\Gamma\backslash \mathbb{B}_5$.)  By uniqueness of extensions, $\overline{\Phi}^{\circ}$ is the restriction of the $\overline{\Phi}$ from Thm. \ref{maintheorem}.\end{proof}

\begin{remark}\label{rem:gluingLMHS}
In the statement of Prop. \ref{prop:Matt3}, ``LMHS'' is meant in the marked sense:  that is, a pair of flags $F^{\bullet}$ and $W_{\bullet}$ on $\tilde{V}$ with its preordained integral basis.  Heuristically, while $e^{\mathbb{C}N}$ kills off the extension information of $\mathrm{Gr}^W_2 \tilde{V}$ by $\mathrm{Gr}^W_0\tilde{V}$, the action of $\Gamma$ equates (some \emph{but not all}) LMHSs which are \emph{isomorphic} as MHS.  The $\Gamma/\Gamma_N$ part of this gets absorbed in equating the infinitely many $\{B(N')\}$ for $N'\in \mathrm{Ad}(\Gamma).N$, while the $\Gamma_N$ part acts on $B(N)$.  
\end{remark}

We can therefore characterize the content of the Proposition for type (B) strata more precisely in the language introduced in its proof.  Namely, fixing a representative $N$ of the orbit $\mathrm{Ad}(\Gamma).N$ attached to a stratum $\mathcal{S}$, $\overline{\Phi}(\nu)$ yields a LMHS up to the action of $\Gamma_N e^{\mathbb{C}N}$, which is exactly to say \textbf{a point of the boundary component} $\mathbf{\Gamma_N\backslash B(N)}$.  Of course, the latter is a quotient of $(\Gamma_0)_N\backslash B(N)$, which is to say a CM abelian fourfold isogenous to $(\mathbb{C}/\mathbb{Z}[i])^{\times 4}$; but we shall need a more precise identification.

\subsection{Hodge-theoretic boundary}\label{sec:Matt4}

To describe $\Gamma_N\backslash B(N)$, recall from $\S$\ref{sec:Matt2} the pair of families of genus-three curves $\{D_{\nu}^{(j)}\}$ ($j=1,2$, $\nu\in\mathcal{S}$) over a type (B) stratum, with the automorphism $\rho$ of degree four.  

Fix $j=1$ or $2$.  As a VHS on $\mathcal{S}$, the $\rho^2$-anti-invariants $H^1(D_{\nu}^{(j)})^-$ are --- unlike the full $H^1(D_{\nu}^{(j)})$ --- isotrivial (i.e. locally constant), since the eigenspaces for $\rho$ have Hodge numbers $(0,2)$ (for $i$) and $(2,0)$ (for $-i$).  But isotrivial VHSs admit finite monodromy, and exactly the same computation as for type (A) in $\S$\ref{sec:Matt2} shows that the $(-i)$-eigenspace $H^{0,1}(D_{\nu}^{(j)})^-$ has order two monodromy (with $T-I$ of rank 1) when two ramification points $\xi_k^{(j)},\xi_{\ell}^{(j)}$ ($k,\ell\in\{1,2,3,4\}$) collide.  (After all, such collision points are where the closures of type (A) and type (B) strata meet.)  Let $\mathcal{R}$ be the finite group generated by these six reflections.

\begin{proposition}\label{prop:Matt4}
We have $$\Gamma_N\backslash B(N)\cong \times_{j=1}^2 \mathcal{R}\backslash \tilde{J}(D_{\nu}^{(j)})^- \cong \mathbb{P}^2\times \mathbb{P}^2\,,$$ where $\tilde{J}(D_{\nu}^{(j)})^-:= \mathrm{Ext}^1_{\mathbb{Z}[\rho]\textup{-MHS}}\left(\mathbb{Z}[\rho],(1-\rho)H_1(D_{\nu}^{(j)},\mathbb{Z})\right)$.
\end{proposition}
\begin{proof}
Begin by using the fact that 
$\overline{\Phi}(\nu)$ is the LMHS at $\nu \in \mathcal{S}$ to compute its 
weight-one part:  by Clemens-Schmid and $\S\S$\ref{sec:Matt2}-\ref{sec:Matt3}, 
we have integrally (with $W_{\bullet}= W(N)_{\bullet}$) that
\begin{align*}
\mathrm{Gr}_1^W \overline{\Phi}(v)
\cong 
\mathrm{Gr}_1^W \widetilde{V}_{\lim,\nu}
\cong 
\mathrm{Gr}_1^W H^1(D_v)^{-}
\cong
H^1(D_v^{(1)})^{-}
\oplus
H^1(D_v^{(2)})^{-}\;,
\end{align*}
where the superscript ``$-$'' denotes the anti-invariants under $\rho^2$.  We can compute
$
H^1 (D_{\nu}^{(j)}, \mathbb{Z})^{-}
\underset{\cong}{\overset{\mathrm{P.D.}}{\longleftarrow}}
H_1 (D_{\nu}^{(j)}, \mathbb{Z})^{-}
$
as in Section \ref{sec:Matt1}, with basis $\{\textsc{a}_1, \textsc{a}_2, \textsc{b}_1, \textsc{b}_2\}$ (where $\textsc{b}_j=\rho_*(\textsc{a}_j)=-\rho^*(\textsc{a}_j)$) and intersection form
\begin{align*}
    \begin{pmatrix}
    0 & -1 & 2 & -1 \\
    -1 & 0 &-1 & 2 \\
    -2 & 1 & 0 & 1 \\
    1 & -2 & -1 & 0
    \end{pmatrix}.
\end{align*}
That is, as a $\mathbb{Z}[i]$-module we have
$
H^1(D_{\nu}^{(j)}, \mathbb{Z})^{-}
\cong 
\mathbb{Z}[i]^{\oplus 2},
$
with Hermitian form 
\begin{align*}
\mathfrak{h}(X,Y)\,=\,{}^t \bar{X}
\begin{pmatrix}
-2 & 1-i \\
i+1 & -2
\end{pmatrix}Y.
\end{align*}
Its Jacobian is 
\begin{align*}
J(D_{\nu}^{(j)})^{-}
=
\mathrm{Ext}^1_{\mathbb{Z}[\rho]\text{-MHS}}
\left(\mathbb{Z}[\rho],
H^1 ( D_{\nu}^{(j)}, \mathbb{Z})^{-}(1)\right)
\cong 
\frac{\mathbb{C}^2}{\mathbb{Z}[i]^{\oplus 2}}\;,
\end{align*}
and we define
\begin{align*}
J(D_{\nu})^{-}
:=
\mathrm{Ext}^1_{\mathbb{Z}[\rho]\text{-MHS}}
\left(
\mathbb{Z}[\rho],
\mathrm{Gr}_1^W \widetilde{V}_{\lim, v}
\right)
\cong 
\oplus_{j=1}^2
J(D_{\nu}^{(j)})^{-}
=
\frac{\mathbb{C}^4}{\mathbb{Z}[i]^{\oplus 4}}.
\end{align*}

To compare this with $\Gamma_N \backslash B(N)$, we use the fact that 
$$
B(N) = \mathbb{C}\langle N\rangle \backslash \mathrm{Lie} (M_N(\mathbb{C}))/
F^0 \mathrm{Lie}\left( M_N(\mathbb{C})\right)
$$
(cf. \cite[$\S$7]{KP16}), which in our case identifies with 
$$
\mathrm{Gr}_1^{W} \widetilde{V}_{\mathbb{C}}
/ F^{1}\mathrm{Gr}_1^W \widetilde{V}_{\mathbb{C}}
\cong
\widetilde{V}_{\lim}^{0,1}
=
V_{\lim}^{0,1}
\cong
\mathbb{C}^4.
$$
We break the action of $\Gamma_N$ on this space into that of 
$$
\Gamma_N^{(-1)}:=
\Gamma_N \cap M_N(\mathbb{Z})
\trianglelefteq
\Gamma_N,
$$
then $\Gamma_N^{(0)} := \Gamma_N / \Gamma_N^{(-1)}$. 

The first group is $P_1(1-i)$ in \cite[$\S$6]{MY93},
acting through $\pi$; that is, the group (of \emph{translations} by) 
$(1-i)\mathbb{Z}[i]^{\oplus 4}$.
This identifies
$
\Gamma_N^{(-1)} \backslash B(N)
$
with the $16:1$ cover
$$
\widetilde{J}(D_{\nu})^{-}
:=
\oplus_{j=1}^2
\widetilde{J}(D_{\nu}^{(j)})^{-}
:=
\bigoplus_{j=1}^2
\frac{
(\Omega^1 (D_{\nu}^{(j)})^{-})^{\vee}
}
{
(1-\rho)
H_1 (D_{\nu}^{(j)}, \mathbb{Z})
}
\cong 
\bigoplus_{j=1}^2
\frac{\mathbb{C}^2}{
(1-i)\mathbb{Z}[i]^{\oplus 2}
}
$$
of $J(D_v)^{-}$.  Note that $\rho^*$ acts as multiplication by $i$ on each $\Omega^1(D^{(j)}_{\nu})^-$.

On the other hand, in the notation of [loc. cit.]
$$
\Gamma_N^{(0)}
\cong 
\frac{
P(1-i)
}{
P_1(1-i)
}
\cong
U\left(
\mathfrak{h}^{\oplus 2}
, (1-i)
\right)
\cong 
\mathcal{R}^{\times 2}
$$
where $\mathcal{R} \cong 
\left( \mathbb{Z}_4 \times \mathbb{Z}_2\right)
\rtimes\mathbb{Z}_2
$
is the order-16 group generated by the reflections
\begin{align*}
\alpha
=
\begin{pmatrix}
-1 & -1 +i \\
0 & 1
\end{pmatrix}
& &
\beta =
\begin{pmatrix}
1 & 0 \\
-1-i & -1
\end{pmatrix}
& &
\gamma
=
\begin{pmatrix}
i & 1+i \\
1-i & -i
\end{pmatrix}
\end{align*}
acting on $\widetilde{J}(D_{\nu}^{(j)})^-$, cf. \cite[Lemma 6.2.(2)]{MY93}.  Moreover, \cite[Lemma 6.4]{MY93} says that $\pm\alpha,\pm\beta,\pm\gamma$ are exactly the six reflections from the beginning of this section.  Finally, noting that $\widetilde{J}(D_{\nu}^{(j)})^-\cong \mathbb{C}^2/\mathbb{Z}[i]^{\oplus 2}=(\mathbb{C}/\mathbb{Z}[i])^2$, we have by \cite[Thm 1.1]{AA18} that
$$\mathcal{R} \backslash \widetilde{J}(D_{\nu}^{(j)})^-
\cong \mathbb{P}^2$$
for $j=1,2$.
\end{proof}

\begin{remark}
Notice that Proposition \ref{prop:Matt4} yields a Hodge-theoretic interpretation of the $\mathbb{P}^2\times \mathbb{P}^2$ already encountered in the proof of Proposition \ref{relativeamplenesshassettside}.
\end{remark}

\begin{remark}\label{rem:powers}
One easily checks that $\mathcal{R}$ contains the subgroup $\{I,-I,iI,-iI\}$ generated by $\rho$.  This is relevant below.
\end{remark}

\subsection{From geometric to Hodge-theoretic boundary}
\label{sec:Matt5}

We can now compute $\overline{\Phi}|_{\mathcal S}$.
Given $\nu \in \mathcal{S}$, heuristically speaking $\overline{\Phi}(v)$ must record: (i) the LMHS of $\widetilde{\mathcal{V}}$ at $v$ modulo reparametrization; plus (ii) some finite level-structure data, since $\Gamma=U(h;(1-i))$ is a proper subgroup of $U(h;\mathbb{Z}[i])$. 

\begin{proof}[Proof of Theorem \ref{thm:Matt}]
By Clemens-Schmid, (i) is equivalent to the MHS on
$H_1(D_{\nu})^{-}$, which takes the form (after Tate-twisting by $(-1)$)
\begin{align*}
\begin{tikzpicture}[scale=0.5]
    \begin{scope}[shift={(-6,0)}]
    \draw[-,line width=1.0pt] (0,0) -- (0,3);
    \draw[-,line width=1.0pt] (0,0) -- (3,0);
	\fill (2,2) circle (5pt);
	\fill (2,0) circle (5pt);
	\fill (0,2) circle (5pt);	
	\node at (2.5,2.5) {$2$};				
	\node at (-0.5,2) {$4$};
	\node at (2,0.5) {$4$};	
	\node at (4.5,1.5) {=};		
    \end{scope}	
    \draw[-,line width=1.0pt] (0,0) -- (0,3);
    \draw[-,line width=1.0pt] (0,0) -- (3,0);
	\fill (2,2) circle (5pt);
	\fill (0,2) circle (5pt);	
	\node at (2.5,2.5) {$1$};				
	\node at (-0.5,2) {$4$};
	\node at (4,1.5) {$\bigoplus$};	
    \begin{scope}[shift={(5,0)}]
    \draw[-,line width=1.0pt] (0,0) -- (0,3);
    \draw[-,line width=1.0pt] (0,0) -- (3,0);
	\fill (2,0) circle (5pt);
	\fill (2,2) circle (5pt);
	\node at (2.5,2.5) {$1$};				
	\node at (2,-0.5) {$4$};				
    \end{scope}	
\end{tikzpicture}
\end{align*}
where the diagram represents the eigenspace decomposition into the $(-i)$ and $(i)$ eigenspaces, respectively.  The class of this extension of $\mathbb{Z}[\rho]$-MHS 
\begin{align} \label{eq:sequenceMHS}
    0 \longrightarrow
\oplus_{j=1}^2 
H_1( D_{\nu}^{(j)})^{-}
\longrightarrow
H_1(D_{\nu})^{-}
\longrightarrow
\mathbb{Z} [\rho]
\longrightarrow
0
\end{align}
is computed by taking the Abel-Jacobi invariant of a $\rho^2$-anti-invariant cycle supported on 
$\mathrm{sing}(D_{\nu}) = \{ q_0,q_1,q_2,q_3\}$ that generates
$\mathbb{Z}[\rho]$. There are four candidates: 
$[q_2]-[q_0]$ and its translates by $\rho^k$ with $k\in \{ 0,1,2,3\}$;
the choice will be erased later, since $\mathcal{R}$ contains $\langle\rho\rangle :=\{1,\rho,\rho^2,\rho^3\}$ (cf. Remark \ref{rem:powers}), and thus is immaterial.  So writing $\sigma^{(j)}$ for a path in $D_{\nu}^{(j)}$ with boundary $\partial \sigma^{(j)}=[q_2]-[q_0]$, we have the functional
$\int_{\sigma^j} \in(\Omega^1 (D_{\nu}^{(j)})^{-})^{\vee}$.
Going modulo ambiguities in the choice of $\sigma^{(j)}$ --- in the sense that the difference of two choices (with the same endpoints) is a topological 1-cycle --- forces us to divide not by 
$H_1( D_{\nu}^{(j)}, \mathbb{Z})^{-}$, but by the image of $H_1(D_{\nu}^{(j)},\mathbb{Z})$ in $(\Omega^1(D_{\nu}^{(j)})^-)^{\vee}$ (which is the strictly larger lattice $\tfrac{1-\rho}{2}H_1(D_{\nu}^{(j)},\mathbb{Z})^-$).  Writing $\overline{J}(D_{\nu}^{(j)})^{-}$ for the resulting isogenous quotient of $J( D_{\nu}^{(j)})^{-}$, we get a point in $\overline{J}(D_{\nu})^{-}/ \langle\rho \rangle$
from the extension \eqref{eq:sequenceMHS}.  This is a little different from what we want, and the way to fix this is to take (ii) into account.

Namely, we use the fact that the four branch points  $\{\xi^{(j)}_1,\xi^{(j)}_2,\xi^{(j)}_3,\xi^{(j)}_4\}$  of $D_{\nu}^{(j)} \twoheadrightarrow \mathbb{P}^1$ are ordered. Choose one, say $\xi_1^{(j)}$. Draw a path 
$\sigma_{+}^{(j)}$ from $\xi_1^{(j)}$ to $q_2$, and define $\sigma_{-}^{(j)} := \rho^2 ( \sigma_{+}^{(j)} )$.  Finally, put $\sigma^{(j)} :=\sigma_{+}^{(j)}-\sigma_{-}^{(j)}$ and observe that $\partial \sigma^{(j)}=q_2 - q_0$.
The ambiguity in the choice of $\sigma_+^{(j)}$ (with the same endpoints) produces an ambiguity of 
$
(1 - \rho^2)H_1( D_{\nu}^{(j)}, \mathbb{Z})
=
(1- \rho)H_1( D_{\nu}^{(j)}, \mathbb{Z})^{-}
$
in $\sigma^{(j)}$. (Changing $\xi_1^{(j)}$ to some other $\xi_{\ell}^{(j)}$ has the same effect.)
Therefore $\int_{\sigma^{(j)}}$ yields a well-defined point of 
$\widetilde{J}( D_{\nu}^{(j)}, \mathbb{Z})^{-}$, which we can project to 
$\mathcal{R} \setminus \widetilde{J}( D_{\nu}^{(j)}, \mathbb{Z})^{-}$ (thereby also removing the ambiguity from the choice of $[q_2]-[q_0]$).  Doing this for $j=1,2$ yields the desired point  
$\overline{\Phi}(\nu)\in\Gamma_N \backslash B(N).$
\end{proof}

\begin{remark}
As alluded to in Remark \ref{rem:gluingLMHS}, not all isomorphic LMHSs are glued together --- just those for which the isomorphism is given by a $\gamma \in \Gamma$. This includes those points of 
$\widetilde{J}(D_{\nu}^{(1)})\times\widetilde{J}(D_{\nu}^{(2)})$
which are equated by $\mathcal{R} \times \mathcal{R}$.
\end{remark}


\section{Moduli of cubic surfaces: Naruki's compactification is toroidal}
\label{Narukicompactificationistoroidal}

The goal of this section is to show that the Naruki compactification $\overline{\mathbf{N}}$ of the moduli space of smooth marked cubic surfaces is isomorphic to the toroidal compactification of an appropriate ball quotient (see \cite{ACT02,DvGK05}). We start by briefly recalling the necessary background.


\subsection{Background on moduli of marked cubic surfaces and compactifications}
\label{backgroundmoduliofmarkedcubicsurfacesandcompactifications}

\begin{definition}
A smooth cubic surface $S\subseteq\mathbb{P}^3$ can be realized as the blow up of $\mathbb{P}^2$ at six points in general linear position not lying on a conic. $S$ is called \emph{marked} if we have a labeling of the six points blown up in $\mathbb{P}^2$. This induces the following labeling of the $27$ lines on $S$, which arise from the blow up construction: Let $E_i$ be the exceptional divisor over $p_i$, $L_{ij}$ the strict transform of the line passing through $p_i,p_j$, and $C_i$ the strict transform of the conic passing through $p_1,\ldots,\widehat{p}_i,\ldots,p_6$. The notation $E_i,L_{ij},C_i$ is what gives the marking of the cubic surface $S$. It follows that the moduli space of marked cubic surfaces, which we denote by $\mathbf{Y}$, is the following quotient:
\[
((\mathbb{P}^2)^6\setminus\Delta)/\SL_3,
\]
where $\Delta\subseteq(\mathbb{P}^2)^6$ is the closed subset parametrizing $6$-tuples of points in $\mathbb{P}^2$ where two coincide, or three are on a line, or all six lie on a conic.
\end{definition}




\subsubsection{Naruki's compactification $\overline{\mathbf{N}}$}
\label{narukicompactificationdefandfacts}

In \cite{Nar82}, Naruki defined a compactification $\overline{\mathbf{N}}$ of $\mathbf{Y}$ as follows. Given a smooth cubic surface $S\subseteq\mathbb{P}^3$, a \emph{tritangent} is a plane intersecting $S$ in three distinct lines. Considering the cross-ratios of $45$ specific quadruples of colinear tritangents of a cubic surface (for details we refer to \cite{Nar82}) defines an embedding $\mathbf{Y}\hookrightarrow(\mathbb{P}^1)^{45}$. The Zariski closure of $\mathbf{Y}$ in $(\mathbb{P}^1)^{45}$ under this embedding gives a compactification $\overline{\mathbf{N}}$ which is referred to as the \emph{Naruki compactification}.

The Naruki compactification is smooth, and the boundary $\overline{\mathbf{N}}\setminus\mathbf{Y}$ is normal crossing and decomposes into the union of $76$ irreducible divisors (see \cite[Theorem 1.1]{Nar82}). $36$ of these divisors are associated to marked cubic surfaces with $A_1$ singularities by \cite[Proposition 11.1]{Nar82}, and are called of \emph{type} $A_1$. The number of $A_1$ singularities on a cubic surface defines higher codimension strata called of \emph{type} $A_1^2,A_2^3$, and $A_1^4$. 
The remaining $40$ divisors are pairwise disjoint and are isomorphic to $(\mathbb{P}^1)^3$ by \cite[Proposition 11.2]{Nar82}. These divisors are called of \emph{type} $N$.


\subsubsection{GIT compactification $\overline{\mathbf{Y}}_{\GIT}$}
\label{GITcompndefandfacts}

Let $\mathbf{Y}$ be the moduli space of marked smooth cubic surfaces. Following \cite[\S2.8]{DvGK05}, the field of rational functions $\mathbb{C}(\mathbf{Y})$ is an extension of the field of rational functions of the GIT quotient
$$
\mathbb{P} (
H^0 ( 
\mathbb{P}^3, \mathcal{O}(3)
)
)
/ \! /_{\mathcal{O}(1)} \SL_4.
$$
This field extension has Galois group the Weyl group $W(E_6)$. Define the compactification $\overline{\mathbf{Y}}_{\GIT}$ to be the normalization of $\mathbb{P} (H^0 ( \mathbb{P}^3, \mathcal{O}(3)))/ \! /_{\mathcal{O}(1)} \SL_4$ in the field of rational functions
$\mathbb{C}\left(\mathbf{Y}\right)$. By construction, $W(E_6)$ acts on $\overline{\mathbf{Y}}_{\GIT}$ and the quotient of $\overline{\mathbf{Y}}_{\GIT}$ by this action recovers the above symmetric GIT quotient.

The GIT compactification is related the Naruki compactification $\overline{\mathbf{N}}$ as follows. Each one of the $40$ divisors of type $N$ of $\overline{\mathbf{N}}$ can be blown down as described in the statement of \cite[Proposition 12.1]{Nar82}. Blowing down all the type $N$ divisors defines a morphism $\overline{\mathbf{N}}\rightarrow\overline{\mathbf{Y}}_{\GIT}$ (see the introduction in \cite{Nar82} and also \cite[\S2.9]{DvGK05}).


\subsubsection{Baily-Borel compactification}
\label{BBcompndefandfacts}

In \cite{ACT02}, the authors described a ball quotient $\Gamma_c\backslash\cB_4$ whose Baily-Borel compactification $\overline{\Gamma_c\backslash\cB_4}^{\bb}$ is isomorphic to $\overline{\mathbf{Y}}_{\GIT}$ (see \cite[Theorem 3.17]{ACT02}). Due to this isomorphism, the singular points of the Baily-Borel compactification are locally isomorphic to the vertex of the cone over the Veronese embedding of $(\mathbb{P}^1)^3$ into $\mathbb{P}^7$. This is true because these are the singularities of the strictly semi-stable points in the GIT quotient 
$\overline{\mathbf{Y}}_{\GIT}$ (see the proof of \cite[Proposition C.4]{CMGHL19}).


\subsection{Proof of Theorem \ref{thm:maincubics}}

By the discussion in \S\ref{narukicompactificationdefandfacts}, \S\ref{GITcompndefandfacts}, and \S\ref{BBcompndefandfacts} we have the following diagram:
\begin{center}
\begin{tikzpicture}[>=angle 90]
\matrix(a)[matrix of math nodes,
row sep=2em, column sep=2em,
text height=2.0ex, text depth=0.25ex]
{\overline{\mathbf{N}}&\overline{\Gamma_c\backslash\mathbb{B}_4}^{\tor}\\
\overline{\mathbf{Y}}_{\GIT}&\overline{\Gamma_c\backslash\mathbb{B}_4}^{\bb},\\};
\path[->] (a-1-1) edge node[]{}(a-2-1);
\path[dashed,->] (a-1-1) edge node[]{}(a-1-2);
\path[->] (a-1-2) edge node[]{}(a-2-2);
\path[->] (a-2-1) edge node[above]{$\cong$}(a-2-2);
\end{tikzpicture}
\end{center}
where the rational map $\overline{\mathbf{N}}\dashrightarrow\overline{\Gamma_c\backslash\mathbb{B}_4}^{\tor}$ is an isomorphism on $\Gamma_c\backslash\cB_4$. We first show that $\overline{\mathbf{N}}\dashrightarrow\overline{\Gamma_c\backslash\mathbb{B}_4}^{\tor}$ extends to a morphism $\overline{\mathbf{N}}\rightarrow\overline{\Gamma_c\backslash\mathbb{B}_4}^{\tor}$. Let $\mathcal{H}\subseteq\mathbb{B}_4$ be the hyperplane arrangement described in \cite[\S2.19]{ACT02}. Note that $\Gamma_c\backslash\mathcal{H}$ corresponds in $\overline{\mathbf{Y}}_{\GIT}$ to cubic surfaces with $A_1$ singularities by \cite[Theorem 2.20]{ACT02}. We have that $\Gamma_c$ acts freely on $\mathbb{B}_4\setminus\mathcal{H}$ by \cite[Lemma 7.28]{ACT02}. Finally, $\overline{\mathbf{N}}$ is a normal crossing compactification of $\Gamma_c\backslash(\mathbb{B}_4\setminus\mathcal{H})\cong\mathbf{Y}$ as we already discussed in \S\ref{narukicompactificationdefandfacts}. Therefore, by Lemma~\ref{extensiontotoroidalwithnccompactificationandfreeactionandhyperplanearrangement} we have the desired extension $\overline{\mathbf{N}}\rightarrow\overline{\Gamma_c\backslash\mathbb{B}_4}^{\tor}$. Moreover, notice that by Lemma~\ref{extensionOpen} the morphism $\overline{\mathbf{N}}\rightarrow\overline{\Gamma_c\backslash\mathbb{B}_4}^{\tor}$ is an isomorphism away from closed subsets of codimension at least $2$.

On the other hand, let $E_1,\ldots,E_{40}$ be the exceptional divisors of $\overline{\Gamma_c\backslash\mathbb{B}_4}^{\tor}\rightarrow\overline{\Gamma_c\backslash\mathbb{B}_4}^{\bb}$, and denote by $E$ their disjoint union. Then $-E$ is a relatively ample divisor by the argument of Proposition~\ref{relativeamplenesstoroidalside}. Let $F_1,\ldots,F_{40}$ be the type $N$ exceptional divisors over the cusps $\overline{\mathbf{N}}\rightarrow\overline{\Gamma_c\backslash\mathbb{B}_4}^{\bb}$, and denote by $F$ their disjoint union. Then $-F$ is a relatively ample divisor by the argument of Proposition~\ref{pre-relativeamplenesshassettside} (note that the singularities blown up are locally isomorphic to the singularity of the cone over the Veronese embedding of $(\mathbb{P}^1)^3$ into $\mathbb{P}^7$ by \S\ref{BBcompndefandfacts}). Therefore, our claim follows by Lemma~\ref{kovacs}.\qed


\section{Naruki's compactification is a moduli space of KSBA stable pairs}

In this section, our goal is to show that Naruki's compactification $\overline{\mathbf{N}}$ has a modular interpretation in terms of KSBA stable pairs. We start by briefly recalling the necessary background.


\subsection{Stable pairs and moduli of cubic surfaces}

\begin{definition}
Let $X$ be a variety and $D$ a $\mathbb{Q}$-divisor on $X$ with coefficients in $(0,1]$. The pair $(X,D)$ is called \emph{semi-log canonical} if the following conditions hold:

\begin{enumerate}

\item $X$ is demi-normal, that is $X$ is $S_2$ and its codimension $1$ points are either regular or ordinary nodes;

\item If $\nu\colon X^\nu\rightarrow X$ is the normalization with conductors $E\subseteq X$ and $E^\nu\subseteq X^\nu$, then the support of $E$ does not contain any irreducible component of $D$;

\item $K_X+D$ is $\mathbb{Q}$-Cartier;

\item The pair $(X^\nu,E^\nu+\nu_*^{-1}D)$ is log canonical. More precisely, for each connected component $Z$ of $X^\nu$, the pair $(Z,(E^\nu+\nu_*^{-1}D)|Z)$ is log canonical (see \cite[Definition 2.8]{Kol13}), where $\nu_*^{-1}D$ denotes the strict transform of $D$.

\end{enumerate}
A pair $(X,D)$ is called \emph{stable} if it is log canonical and $K_X+D$ is ample.
\end{definition}

\begin{example}
\label{stablepaircubicsurfacewith27lines}
Let $S\subseteq\mathbb{P}^3$ be a smooth cubic surface whose $27$ lines are normal crossings. If $B\subseteq S$ is the divisor given by the sum of these lines, then $(S,cB)$ is log canonical for all $c\in\mathbb{Q}\cap(0,1]$. Let us now determine for which $c$ the log canonical divisor $K_S+cB$ is ample. If $H\subseteq\mathbb{P}^3$ is a hyperplane, then
\[
K_S+cB\sim-H|_S+9cH|_S=(9c-1)H|_S,
\]
which implies that $(S,cB)$ is stable for all rational $\frac{1}{9}<c\leq1$. Let us explain in detail why $B\sim9H|_S$. We can subdivide the $27$ lines into nine sets of three distinct coplanar lines. The hyperplanes associated to such sets of lines are called tritangents (see \cite[page 250-251]{Cay49} for a list of them, and the lines that they contain). In more detail, label the $27$ lines as follows: considering the blow up model $S\rightarrow\mathbb{P}^2$ at six general points $p_1,\ldots,p_6$, let $a_i$ be the exceptional divisors over $p_i$, $b_i$ the strict transforms of the conic not passing through $p_i$, and $c_{ij}$, $1\leq i<j\leq6$, the strict transform of the line through $p_i,p_j$. The $45$ tritangents are the planes
\[
(ij)=\langle a_i,b_j,c_{ij}\rangle,~(ij,k\ell,mn)=\langle c_{ij},c_{k\ell},c_{mn}\rangle,
\]
where $\{i,j,k,\ell,m,n\}=\{1,\ldots,6\}$ and $i<j,k<\ell,m<n$. Note that $(ij) \neq (ji)$. The $27$ lines on $S$ are then contained in the following tritangents (we included the notation of both Cayley and Schl\"afli):
\begin{align*}
(16)=(w)
= \{ a_1, b_6, c_{16} \}
& &
(12,34,56)=(\theta)
= \{ c_{12}, c_{34}, c_{56}\}
& & 
(52)=(\overline{\theta})
=\{ a_5, b_2, c_{25}\}
\\
(64) = (\overline{l})= \{a_6, b_4,c_{46}\}
& &
(15, 24, 36) = (\overline{m}) = 
\{c_{15}, c_{24}, c_{36}\} 
& &
(23) =  (\overline{n}) = \{a_2, b_3, c_{23}\}
\\
(45) = (l) = \{ a_4, b_5, c_{45} \}
& &
(14, 26, 35) =  (m) = \{ c_{14}, c_{26}, c_{35} \}
& &
(31) =(n) = \{ a_3, b_1, c_{13}\}.
\end{align*}
\end{example}


In the current paper, we are interested in considering the following alternative compactification by stable pairs.

\begin{definition}
\label{definitionksbaweight1/9+epsilon}
Let $\mathbf{Y}_\times\subseteq\mathbf{Y}$ be the open subset parametrizing marked smooth cubic surfaces $S$ such that the divisor $B$ on it consisting of the sum of the $27$ lines is normal crossing. Consider the family $(\mathcal{X},\mathcal{B})\rightarrow\mathbf{Y}_\times$ of pairs $(S,B)$, whose existence is guaranteed by \cite[Theorem 1.1]{HKT09}.
Fix a rational number $0<\epsilon\ll1$. Then $\left(\mathcal{X},\left(\frac{1}{9}+\epsilon\right)\mathcal{B}\right)\rightarrow\mathbf{Y}_\times$ is a family of stable pairs by Example~\ref{stablepaircubicsurfacewith27lines}. It follows that $\left(\mathcal{X},\left(\frac{1}{9}+\epsilon\right)\mathcal{B}\right)\rightarrow\mathbf{Y}_\times$ induces an embedding of $\mathbf{Y}_\times$ into an appropriate projective moduli space of stable pairs. Denote by $\overline{\mathbf{Y}}_{\frac{1}{9}+\epsilon}$ the Zariski closure of the image of $\mathbf{Y}_\times$ under this embedding. 
\end{definition}

The ultimate goal is to show that $\overline{\mathbf{N}}$ is isomorphic to the normalization of 
$\overline{\mathbf{Y}}_{\frac{1}{9}+\epsilon}$. To achieve this, we construct an appropriate family of stable pairs over $\overline{\mathbf{N}}$. Our starting point is to consider the family of surfaces over $\overline{\mathbf{N}}$ constructed by Naruki and Sekiguchi in \cite{NS80}  modifying the family of cubic surfaces originally constructed by Cayley in \cite{Cay49}.


\subsection{Naruki's family of cubic surfaces on \texorpdfstring{$\overline{\mathbf{N}}$}{Lg}}

\begin{definition}
\label{def:FamilyCubic}
Let $[x_0:x_1:x_2:x_3]$ denote the coordinates of $\mathbb{P}^3$ and let $(\lambda,\mu,\nu,\rho)$ be the coordinates of the torus $T=(\mathbb{C}^*)^4$. Let $\mathcal{S}\subseteq\mathbb{P}^3\times T$ be the family over $T$ of cubic surfaces defined by the following equation (\cite[Equation (5.1)]{Nar82}):
\begin{align*}
\rho x_3\bigg(&\lambda x_0^2+\mu x_1^2+\nu x_2^2+(\rho-1)^2(\lambda \mu \nu \rho-1)^2x_3^2\\
+&(\mu \nu +1)x_1x_2+(\lambda \nu +1) x_0x_2+(\lambda\mu +1)x_0x_1
\\
-&(\rho -1)(\lambda\mu\nu\rho-1)x_3\big((\lambda+1)x_0+(\mu+1)x_1+(\nu+1)x_2\big)\bigg)+x_0x_1x_2=0.
\end{align*}
The family $\mathcal{S}\rightarrow T$ extends to a family of cubic surfaces $\overline{\mathcal{S}}\rightarrow\overline{\mathbf{N}}$ by \cite[Proposition 12.2]{Nar82}.
\end{definition}
\begin{remark}
Let $\overline{\mathcal{S}}\rightarrow\overline{\mathbf{N}}$ be the family in Definition~\ref{def:FamilyCubic}. By the discussion after \cite[Proposition 12.2]{Nar82}, $\overline{\mathcal{S}}\rightarrow\overline{\mathbf{N}}$ has fibers equal to $\{x_0x_1x_2=0\}$ over the exceptional divisors of the map $\overline{\mathbf{N}}\rightarrow\overline{\mathbf{Y}}_{\GIT}$. In the next lemma we describe how the $27$ lines degenerate on such reducible fibers.
\end{remark}


\begin{lemma}\label{limit27}
Let $S(\mathbf{r})$ be the cubic surface of the family $\mathcal{S}\rightarrow T$ parametrized by the point $\mathbf{r}=(\lambda,\mu,\nu,\rho)\in T$. The limit for $\rho\to0$ of this surface is the union of three coordinate hyperplanes $x_0x_1x_2=0$, and the limits of the $27$ lines on $S(\mathbf{r})$ are the lines $\overline{A_iB_k},\overline{A_iC_k},\overline{B_jC_k}$ spanned by the points
$A_i,B_j,C_k$ given by
\begin{align*}
A_1 := [1:0:0:0] & & A_2 := [1:0:0:1] & & A_3 := [1:0:0:\lambda]
\\
B_1 := [0:1:0:0] & & B_2 := [0:1:0:1] & & B_3 := [0:1:0:\mu]
\\
C_1 := [0:0:1:0] & & C_2 := [0:0:1:1] & & C_3 := [0:0:1:\nu].
\end{align*}
\end{lemma}

\begin{proof}
Each one of the $27$ lines of $S(\mathbf{r})$ is the complete intersection of two tritangents (see Table~\ref{tab:27lines}, where we use the notational conventions of \cite{Cay49}). Then, for each one of the intersections, we compute the limit for $\rho\to0$ of the two tritangents. The explicit equations of the tritangents can be found in \cite[Table 1]{Nar82} (note, the coordinates there are $[X:Y:Z:W]$, while we use $[x_0:x_1:x_2:x_3]$). The limit lines can be found in the columns ``Limit" of Table~\ref{tab:27lines}.

We illustrate this calculation explicitly for two of the $27$ lines. The line $a_1$ is cut out by $x_0 = x_3 = 0$. Therefore, its limit is the line $\overline{B_1C_1}$.
The line $a_2$ is the intersection of the tritangents $(p_{,})$ and $(\theta)$, whose equations are
\begin{align*}
(p_{,}) 
&\colon
x_0 + \mu \rho x_1 + \nu \rho x_2 -
\rho (\rho - 1)(
\lambda \mu \nu \rho + \mu \nu - \mu - \mu
)x_3=0,
\\
(\theta) &\colon
\lambda x_0
+
\mu x_1 
+ 
\nu x_2
-
\left( 
(\rho -1)(\lambda \mu \nu \rho -1)
-
\rho(\lambda -1)(\mu -1)(\nu -1)
\right)
x_3=0.
\end{align*}
Then, it holds that 
\[
\lim_{\rho \to 0}  a_2 
=
\lim_{\rho \to 0}  ((p_{,}) \cap (\theta)) 
=
\{
x_0 = 
\lambda x_0+\mu x_1 + \nu x_2 - x_3
=
0
\}
=
\overline{B_3C_3}.\qedhere
\]
\end{proof}


\begin{table}
    \centering
    \caption{Using Cayley's notation in \cite{Cay49}, the table lists the $27$ lines, realizes each line as the intersection of two trintagents, and gives the limit for $\rho\to0$.
    }
    \renewcommand{\arraystretch}{1.4}
    \begin{tabular}{ |c|c|c|| c|c|c|}
    \hline   
    Lines & Tritangents & Limit & Lines & Tritangent & Limit
        \\ \hline 
        $a_1$ & (w)  $(x)$ & $\overline{B_1C_1}$
        & $a_6$ & (x) $(r_{,})$ & $\overline{A_2B_1}$
        \\ \hline
        $b_1$ & (w)  $(y)$ & $\overline{A_1C_1}$
&         
               $b_6$ & (y) $(p_{,})$ & $\overline{B_2C_1}$
        \\ \hline
        $c_1$ & (w)  $(z)$ & $\overline{A_1B_1}$
&       $c_6$ & (z) $(q_{,})$ & $\overline{A_1C_2}$
        \\ \hline \hline
        $a_2$ & $(p_{,})$  $(\theta)$ & $\overline{B_3C_3}$
        & 
        $a_7$ &         $(q_{,})$ (x)        & $\overline{A_2C_1}$
        \\ \hline
        $b_2$ & $ (q_{,}) $  
        $(\theta)$ & $\overline{A_3C_3}$
        & 
        $b_7$ & $(r_{,})$  (y) & $\overline{A_1B_2}$         
        \\ \hline
        $c_2$ & $(r_{,})$ 
        $(\theta)$ & $\overline{A_3B_3}$
        & 
        $c_7$ & $(p_{,})$ (z) & $\overline{B_1C_2}$
        \\ \hline \hline 
        $a_3$ & 
        $(\overline{p}_{,})$ $(\overline{\theta})$
        & $\overline{B_2C_2}$
        & 
        $a_8$ &         $(\overline{q}_{,})$ $(\overline{\mathrm{x}})$
        & $\overline{A_3C_1}$

        \\ \hline
        $b_3$ & $(\overline{q}_{,})$  $(\overline{\theta})$ & $\overline{A_2C_2}$ 
        & 
        $b_8$ & $(\overline{r}_{,})$  $(\overline{\mathrm{y}})$ & $\overline{A_1B_3}$ 
        \\ \hline  
        $c_3$ & $(\overline{r}_{,})$ $(\overline{\theta})$& $\overline{A_2B_2}$
        & 
        $c_8$ & $(\overline{p}_{,})$ $(\overline{\mathrm{z}})$& $\overline{B_1C_3}$
        \\ \hline \hline
        $a_4$ & $(x)$ $(g)$ & $\overline{B_3C_2}$
        & 
         $a_9$ & $(\overline{r}_{,})$ 
        $(\overline{\mathrm{x}})$
        &  $\overline{A_3B_1}$        
        \\ \hline
        $b_4$ & $(y)$ $(\overline f)$ & $\overline{A_2C_3}$
        & 
          $b_9$ & $(\overline{p}_{,})$
        $(\overline{\mathrm{y}})$
        & $\overline{B_3C_1}$
        \\ \hline
        $c_4$ & $(z)$ $(f)$ & $\overline{A_3B_2}$ & 
         $c_9$ &$(\overline{q}_{,})$ 
        $(\overline{\mathrm{z}})$
        & 
        $\overline{A_1C_3}$
        \\ \hline \hline
        $a_5$ & $(x)$ $(\overline{g})$ & $\overline{B_2C_3}$ &
        & &
        \\ \hline
               $b_5$ & $(y)$ $(f)$ & $\overline{A_3C_2}$
               & & &
        \\ \hline
        $c_5$ & $(z)$ $(\overline{f})$ & $\overline{A_2B_3}$
        & & &
         \\ \hline        
    \end{tabular}
    \label{tab:27lines}
\end{table}


\subsection{Family of stable pairs over \texorpdfstring{$\overline{\mathbf{N}}$}{Lg}}

\begin{definition}\label{def:Definition}
Let $\overline{\mathcal{S}}\rightarrow\overline{\mathbf{N}}$ be Naruki's family, which agrees with $\mathcal{X}$ over $\mathbf{Y}_\times$ (see Definition~\ref{definitionksbaweight1/9+epsilon}). Endow $\overline{\mathcal{S}}$ with a divisor $\overline{\mathcal{B}}$ given by the Zariski closure of $\mathcal{B}$ in $\overline{\mathcal{S}}$. In what follows, we show that
\[
\left(\overline{\mathcal{S}},\left(\frac{1}{9}+\epsilon\right)\overline{\mathcal{B}}\right)\rightarrow\overline{\mathbf{N}}
\]
is a proper flat family of stable pairs, and $\overline{\mathcal{B}}$ is relative over the base $\overline{\mathbf{N}}$.
\end{definition}

\begin{proposition}
\label{narukifamilyisafamilyofstablepairs}
Consider the family $\left(\overline{\mathcal{S}},\left(\frac{1}{9}+\epsilon\right)\overline{\mathcal{B}}\right)\rightarrow\overline{\mathbf{N}}$. Then for all $x\in\overline{\mathbf{N}}$, the fiber $\left(\overline{\mathcal{S}}_x,\left(\frac{1}{9}+\epsilon\right)\overline{\mathcal{B}}_x\right)$ is a stable pair.
\end{proposition}

\begin{proof}
For simplicity of notation, denote $\left(\overline{\mathcal{S}}_x,\left(\frac{1}{9}+\epsilon\right)\overline{\mathcal{B}}_x\right)$ by $\left(S_0,\left(\frac{1}{9}+\epsilon\right)B_0\right)$. If $x\in\mathbf{Y}_\times$, then we already know that $\left(S_0,\left(\frac{1}{9}+\epsilon\right)B_0\right)$ is a stable pair by Example~\ref{stablepaircubicsurfacewith27lines}, so from now on in the proof we consider $x\in\overline{\mathbf{N}}\setminus\mathbf{Y}_\times$. Given the length of the proof, we split it into two parts: in part $1$ we assume $x$ is not contained in a type $N$ divisor, and in part $2$ we assume otherwise.

\textbf{Part 1.} Assume $x$ is not in a type $N$ divisor, in particular $S_0$ is an irreducible cubic surface with at most four $A_1$ singularities. In all these cases, it holds that the log canonical divisor of the pair is ample because
\[
K_{S_0}+\left(\frac{1}{9}+\epsilon\right)B_0\sim-H|_{S_0}+\left(\frac{1}{9}+\epsilon\right)H|_{S_0}=9\epsilon H|_{S_0},
\]
where $H\subseteq\,\mathbb{P}^3$ is a hyperplane. $K_{S_0}\sim-H|_{S_0}$ follows from adjunction, which can be applied because $S_0$ has at worst $A_1$ singularities. To prove that $B_0\sim-9H|_{S_0}$ we argue as follows. The equation for $S_0$ is given in Definition~\ref{def:FamilyCubic} for appropriate values of $\lambda,\mu,\nu,\rho$. The tritangents of $S_0$ can also be expressed in terms of these constants by \cite[Table 1]{Nar82}. As we discussed in Example~\ref{stablepaircubicsurfacewith27lines}, in the general case there are $9$ specific tritangents such that each one contains exactly three of the $27$ lines. By specializing the equations of these planes to the $\lambda,\mu,\nu,\rho$ giving $S_0$, we obtain that $B_0$ is given by nine hyperplane sections (notice that in this case lines can possibly acquire multiplicities).

In what follows we prove that $\left(S_0,\left(\frac{1}{9}+\epsilon\right)B_0\right)$ is log canonical. Let us first consider the case where $S_0$ is smooth. By \cite[Proposition 4.1 (iii)]{Tu05} we know that every line in $B_0$ has multiplicity $1$. Moreover, for every point $p\in S_0$, we have at most three lines in $S_0$ passing through $p$. This is because the lines in $S_0$ through $p$ are contained in the tangent plane to $S_0$ at $p$, and this plane intersects $S$ in a curve of degree $3$. Therefore, by Lemma~\ref{logcanonicityconcurrentlines} we have that $\left(S_0,\left(\frac{1}{9}+\epsilon\right)B_0\right)$ is log canonical, and hence stable.

Now assume $x$ is in a boundary stratum of type $A_1,A_1^2,A_1^3$, or $A_1^4$. Recall that these correspond to $S_0$ having exactly $1,2,3$, or $4$ singular points of type $A_1$, and no other singularities. To prove that $\left(S_0,\left(\frac{1}{9}+\epsilon\right)B_0\right)$ is log canonical, the following facts are crucial. Let $\ell$ be a line in $S_0$.

\begin{itemize}

\item If $\ell$ contains exactly one singular point, then $\ell$ is of multiplicity $2$ (\cite[Proposition 4.1 (i) (a)]{Tu05}).

\item If $\ell$ contains two singular points, then $\ell$ is of multiplicity $4$ (\cite[Proposition 4.1 (ii) (a)]{Tu05}).

\item If $\ell$ does not contain any singularity, then $\ell$ is of multiplicity $1$ (\cite[Proposition 4.1 (iii)]{Tu05}).

\end{itemize}
We now argue by cases on the number of $A_1$ singularities $S_0$ has. In what follows, we refer to the proof of \cite[Proposition 4.1]{Tu05}.

\begin{enumerate}

\item If $S_0$ has exactly one $A_1$ singularity, then there exist exactly six lines of multiplicity two by \cite{Tu05}. Using the above facts, these double lines have to pass through the $A_1$ singularity. Moreover, the remaining $15$ lines avoid the $A_1$ singularity. It now follows that the pair $\left(S_0,\left(\frac{1}{9}+\epsilon\right)B_0\right)$ is log canonical by Lemma~\ref{logcanonicityconcurrentlinesthroughA1} and Lemma~\ref{logcanonicityconcurrentlines}, because the sum of the weights of the lines through the $A_1$ singularity add up to strictly less than $2$, and the same holds for any smooth point of $S_0$.

\item Assume $S_0$ has exactly two $A_1$ singularities. Then by \cite{Tu05} there exist exactly eight lines of multiplicity two, one line of multiplicity four, and the remaining lines avoid the singularities and have multiplicity one. The quadruple line has to pass through both the $A_1$ singularities, while each one of the double lines passes through exactly one of the two $A_1$ singularities. Next we want to understand how many double lines pass through each $A_1$ singularity. The cubic surface $S_0$ arises as follows. Let $S_1$ be the blow up $\mathbb{P}^2$ at four points in general linear position. Let $S_2$ be the blow up of $S_1$ at two points on two different exceptional divisors. Then $S_0$ is obtained by contracting the two $(-2)$-curves on $S_2$ (for this construction, see \cite[\S9.2.2]{Dol12}). From the symmetry of this construction, it follows that we must have exactly four double lines through each of the two $A_1$ singularities. As in the previous case, we can conclude that $\left(S_0,\left(\frac{1}{9}+\epsilon\right)B_0\right)$ is log canonical again by applying Lemma~\ref{logcanonicityconcurrentlinesthroughA1} and Lemma~\ref{logcanonicityconcurrentlines}, regardless of the behavior of the remaining seven multiplicity one lines away from the singularities.

\item Now consider the case where $S_0$ has exactly three $A_1$ singularities. Then by \cite{Tu05} there exist exactly six lines with multiplicity two, three lines with multiplicity four, and the remaining lines avoid the singularities and have multiplicity one. The three lines with multiplicity four have to contain two $A_1$ singularities each, while each one of the double lines has to pass through exactly one of the $A_1$ singularities. In particular, we must have exactly one quadruple line through each pair of $A_1$ singularities (otherwise, we would have a line with multiplicity at least eight). Let us understand more precisely how many double lines pass through each $A_1$ singularity. By \cite[\S9.22]{Dol12}, such a cubic surface arises as follows. Let $S_1$ be the blow up $\mathbb{P}^2$ at three points in general linear position. Let $S_2$ be the blow up of $S_1$ at three points on the three different exceptional divisors. Then $S_0$ is obtained by contracting the three $(-2)$-curves on $S_2$. By the symmetry of this construction, we must have that there are exactly two double lines through each $A_1$ singularity. Again, we can conclude that $\left(S_0,\left(\frac{1}{9}+\epsilon\right)B_0\right)$ is log canonical by applying Lemma~\ref{logcanonicityconcurrentlinesthroughA1} and Lemma~\ref{logcanonicityconcurrentlines}, regardless of the behavior of the remaining three multiplicity one lines away from the singularities.

\item Finally, assume $S_0$ has exactly four $A_1$ singularities. Then by \cite{Tu05} there are exactly six lines of multiplicity four. The remaining 3 lines are reduced and avoid the $A_1$ singularities. By Lemma~\ref{logcanonicityconcurrentlinesthroughA1} and Lemma~\ref{logcanonicityconcurrentlines} we can conclude that the $\left(S_0,\left(\frac{1}{9}+\epsilon\right)B_0\right)$ is log canonical.

\end{enumerate}

\textbf{Part 2.} Assume $x$ is in a type $N$ divisor. Before analyzing the the pairs parametrized by divisors of type $N$, we set up the following notation. On a divisor of type $N$ the surface $S_0$ is given by $x_0x_1x_2=0$, and denote by $H_i$ the hyperplane given by $x_i=0$. By Lemma~\ref{limit27}, each one of these irreducible components contains exactly 9 of the $27$ lines. More precisely, $H_0$ contains
\[
\overline{B_1C_1},~\overline{B_1C_2},~\overline{B_1C_3},~\overline{B_2C_1},~\overline{B_2C_2},~\overline{B_2C_3},~\overline{B_3C_1},~\overline{B_3C_2},~\overline{B_3C_3},
\]
$H_1$ contains
\[
\overline{A_1C_1},~\overline{A_1C_2},~\overline{A_1C_3},~\overline{A_2C_1},~\overline{A_2C_2},~\overline{A_2C_3},~\overline{A_3C_1},~\overline{A_3C_2},~\overline{A_3C_3},
\]
and finally $H_2$ contains
\[
\overline{A_1B_1},~\overline{A_1B_2},~\overline{A_1B_3},~\overline{A_2B_1},~\overline{A_2B_2},~\overline{A_2B_3},~\overline{A_3B_1},~\overline{A_3B_2},~\overline{A_3B_3}.
\]
The equations of these lines, in terms of $\lambda,\mu,\nu$, in the corresponding plane containing them are listed in Table~\ref{limitlinesonthreeplanes}.

For simplicity of notation, let us set $L_{ij}^0=\overline{B_iC_j},L_{ij}^1=\overline{A_iC_j},L_{ij}^2=\overline{A_iB_j}$. The goal is to show that each pair $\left(S_0,\left(\frac{1}{9}+\epsilon\right)\sum_{i,j,k}L_{ij}^k\right)$ is stable. Let $\{a,b,c\}=\{0,1,2\}$ and define $D_a$ to be the conductor $(H_b+H_c)|_{H_a}$. We need to show that, for $a=0,1,2$, the following pair is stable:
\[
\left(H_a,D_a+\left(\frac{1}{9}+\epsilon\right)\sum_{i,j}L_{ij}^a\right).
\]
Notice that the log canonical divisor is ample because, if $\ell$ is a line in $H_a$, then
\[
K_{H_a}+D_a+\left(\frac{1}{9}+\epsilon\right)\sum_{i,j}L_{ij}^a\sim-3\ell+2\ell+9\left(\frac{1}{9}+\epsilon\right)\ell=9\epsilon\ell,
\]
which is ample. So we only need to check that the pair is log canonical. We distinguish the cases where $x$ is in the open part of a stratum of type $N,(A_1,N),(A_1^2,N)$, or $(A_1^3,N)$. The corresponding line arrangements on $S_0$ are pictured in Table~\ref{ksbastabledegenerations27linesweight1/9+e}.

\begin{enumerate}

\item Type $N$. By \cite[Proof of Lemma 11.4]{Nar82}, this corresponds to considering $\lambda,\mu,\nu\neq0,1$. As it can be argued from Table~\ref{limitlinesonthreeplanes}, for $a=0,1,2$, the lines $L_{ij}^a$ are all distinct. Moreover, the overall line arrangement has only six multiple points: three triple points along each component of the conductor $D_a$. Therefore, by Lemma~\ref{logcanonicityconcurrentlines} below, the pair $\left(H_a,D_a+\left(\frac{1}{9}+\epsilon\right)\sum_{i,j}L_{ij}^a\right)$ is log canonical.

\item Type $(A_1,N)$. By \cite[Proof of Lemma 11.4]{Nar82}, this corresponds to $\lambda=0$ and $\mu,\nu\neq0,1$. Specializing the lines in Table~\ref{limitlinesonthreeplanes} we can argue the following:
\begin{itemize}

\item On $H_0$ the situation is analogous to $H_0,H_1,H_2$ in the type $N$ case.

\item On $H_1$ we have exactly three multiple lines. These have multiplicity $2$ and pass through $A_1=A_3$. The pair $\left(H_a,D_a+\left(\frac{1}{9}+\epsilon\right)\sum_{i,j}L_{ij}^a\right)$ is log canonical again by Lemma~\ref{logcanonicityconcurrentlines}.

\item On $H_2$, the situation is analogous to $H_1$ above.

\end{itemize}

\item Type $(A_1^2,N)$. By \cite[Proof of Lemma 11.4]{Nar82}, this corresponds to considering $\lambda,\mu=0$ and $\nu\neq0,1$. Specializing the lines in Table~\ref{limitlinesonthreeplanes} we can argue the following:

\begin{itemize}

\item On $H_0$ and $H_1$ the situation is analogous to the arrangement on $H_1,H_2$ when we are in type $(A_1,N)$.

\item On $H_2$ we have one quadruple line, two double lines, and one other reduced line in general linear position with respect to the others. The two double lines intersect the quadruple line in two distinct points lying on the two irreducible components of the conductor. The pair $\left(H_a,D_a+\left(\frac{1}{9}+\epsilon\right)\sum_{i,j}L_{ij}^a\right)$ is log canonical again by Lemma~\ref{logcanonicityconcurrentlines}.

\end{itemize}

\item Type $(A_1^3,N)$. By \cite[Proof of Lemma 11.4]{Nar82}, this corresponds to considering $\lambda,\mu,\nu=0$. Specializing the lines in Table~\ref{limitlinesonthreeplanes} we can see that on each plane $H_a$ the line arrangement is analogous to the one on $H_2$ in the type $(A_1^2,N)$ case.\qedhere

\end{enumerate}
\end{proof}

\begin{table}
\centering
\caption{Limit lines on $S_0=H_0\cup H_1\cup H_2$ for pairs parametrized by a divisor of type $N$.}
\label{limitlinesonthreeplanes}
\renewcommand{\arraystretch}{1.4}
\begin{tabular}{|c|c|c|c|c|c|}
\hline
\multicolumn{2}{|c|}{Lines on $H_0=\{x_0=0\}$} & \multicolumn{2}{c|}{Lines on $H_1=\{x_1=0\}$} & \multicolumn{2}{c|}{Lines on $H_2=\{x_2=0\}$}
\\
\hline
$\overline{B_1C_1}$ & $x_3=0$ & $\overline{A_1C_1}$ & $x_3=0$ & $\overline{A_1B_1}$ & $x_3=0$
\\
\hline
$\overline{B_1C_2}$ & $-x_2+x_3=0$ & $\overline{A_1C_2}$ & $-x_2+x_3=0$ & $\overline{A_1B_2}$ & $-x_1+x_3=0$
\\
\hline
$\overline{B_1C_3}$ & $-\nu x_2+x_3=0$ & $\overline{A_1C_3}$ & $-\nu x_2+x_3=0$ & $\overline{A_1B_3}$ & $-\mu x_1+x_3=0$
\\
\hline
$\overline{B_2C_1}$ & $-x_1+x_3=0$ & $\overline{A_2C_1}$ & $-x_0+x_3=0$ & $\overline{A_2B_1}$ & $-x_0+x_3=0$
\\
\hline
$\overline{B_2C_2}$ & $-x_1-x_2+x_3=0$ & $\overline{A_2C_2}$ & $-x_0-x_2+x_3=0$ & $\overline{A_2B_2}$ & $-x_0-x_1+x_3=0$
\\
\hline
$\overline{B_2C_3}$ & $-x_1-\nu x_2+x_3=0$ & $\overline{A_2C_3}$ & $-x_0-\nu x_2+x_3=0$ & $\overline{A_2B_3}$ & $-x_0-\mu x_1+x_3=0$
\\
\hline
$\overline{B_3C_1}$ & $-\mu x_1+x_3=0$ & $\overline{A_3C_1}$ & $-\lambda x_0+x_3=0$ & $\overline{A_3B_1}$ & $-\lambda x_0+x_3=0$
\\
\hline
$\overline{B_3C_2}$ & $-\mu x_1-x_2+x_3=0$ & $\overline{A_3C_2}$ & $-\lambda x_0-x_2+x_3=0$ & $\overline{A_3B_2}$ & $-\lambda x_0-x_1+x_3=0$
\\
\hline
$\overline{B_3C_3}$ & $-\mu x_1-\nu x_2+x_3=0$ & $\overline{A_3C_3}$ & $-\lambda x_0-\nu x_2+x_3=0$ & $\overline{A_3B_3}$ & $-\lambda x_0-\mu x_1+x_3=0$
\\
\hline
\end{tabular}
\end{table}

\begin{lemma}
\label{logcanonicityconcurrentlinesthroughA1}
Let $X\subseteq\mathbb{P}^3$ be a quadric cone and let $L_1,\ldots,L_n\subseteq X$ be distinct lines passing through the $A_1$ singularity. Let $0<c_1,\ldots,c_n\leq1$ be rational numbers and denote by $c$ their sum. Then the pair $\left(X,\sum_{i=1}^nc_iL_i\right)$ is log canonical if and only if $c\leq2$.
\end{lemma}

\begin{proof}
The blow up $f\colon\widehat{X}\rightarrow X$ at the $A_1$ singularity provides a log resolution of $\left(X,\sum_{i=1}^nc_iL_i\right)$; denote by $E$ the exceptional divisor. First notice that for all $i$,
\[
f^*L_i=\widetilde{L}_i+xE,
\]
where $\widetilde{L}_i$ denotes the strict transform of $L_i$ and $x$ is some rational number. By multiplying both sides by $E$ and using the projection formula, we can compute that $x=\frac{1}{2}$. Now let us compute the discrepancy for the pair $\left(X,\sum_{i=1}^nc_iL_i\right)$. For some $a\in\mathbb{Q}$ we have that
\[
K_{\widehat{X}}+\sum_{i=1}^nc_i\widetilde{L}_i=f^*\left(K_X+\sum_{i=1}^nc_iL_i\right)+aE.
\]
We have that $K_{\widehat{X}}=f^*K_X$, so
\begin{align*}
\sum_{i=1}^nc_i\widetilde{L}_i=f^*\left(\sum_{i=1}^nc_iL_i\right)+aE\implies\sum_{i=1}^nc_i\widetilde{L}_i=\sum_{i=1}^nc_i\left(\widetilde{L}_i+\frac{1}{2}E\right)+aE\implies0=\frac{c}{2}E+aE,
\end{align*}
which implies that $a=-\frac{c}{2}$. So the pair is log canonical if and only if $c\leq2$.
\end{proof}

\begin{lemma}
\label{logcanonicityconcurrentlines}
Let $L_1,\ldots,L_n$ be distinct lines passing through the origin of $\mathbb{A}^2$. Let $0<c_1,\ldots,c_n\leq1$ be rational numbers and denote by $c$ their sum. Then the pair $\left(\mathbb{A}^2,\sum_{i=1}^nc_iL_i\right)$ is log canonical if and only if $c\leq2$.
\end{lemma}

\begin{proof}
The blow up $f\colon Y\rightarrow\mathbb{A}^2$ at the origin provides a log resolution of $\left(\mathbb{A}^2,\sum_{i=1}^nc_iL_i\right)$; denote by $E$ the exceptional divisor. First notice that for all $i$,
\[
f^*L_i=\widetilde{L}_i+E,
\]
where $\widetilde{L}_i$ denotes the strict transform of $L_i$. Then for some $a\in\mathbb{Q}$ we have that
\[
K_{Y}+\sum_{i=1}^nc_i\widetilde{L}_i=f^*\left(K_{\mathbb{A}^2}+\sum_{i=1}^nc_iL_i\right)+aE.
\]
Note that $K_Y=f^*K_{\mathbb{A}^2}+E=E$, so
\begin{align*}
&E+\sum_{i=1}^nc_i\widetilde{L}_i=f^*\left(\sum_{i=1}^nc_iL_i\right)+aE\\
\implies&\sum_{i=1}^nc_i\widetilde{L}_i=\sum_{i=1}^nc_i\left(\widetilde{L}_i+E\right)+(a-1)E\implies0=cE+(a-1)E,
\end{align*}
which implies that $a=1-c$. So the pair is log canonical if and only if $c\leq2$.
\end{proof}

\begin{proposition}
\label{properflatfamilywithrelativedivisor}
The family $\left(\overline{\mathcal{S}},\left(\frac{1}{9}+\epsilon\right)\overline{\mathcal{B}}\right)\rightarrow\overline{\mathbf{N}}$ is proper, flat, and the divisor $\overline{\mathcal{B}}$ is flat over the base $\overline{\mathbf{N}}$.
\end{proposition}

\begin{proof}
$\overline{\mathcal{S}}\rightarrow\overline{\mathbf{N}}$ is clearly proper because it is a morphism of proper varieties, and it is flat by \cite[Lemma 10.12]{HKT09}.

We now show that $\overline{\mathcal{B}}$ is flat over $\overline{\mathbf{N}}$. First, recall that for a smooth cubic surface whose $27$ lines are normal crossing, the $27$ lines intersect at exactly $135$ points (see \cite[Table 4.1]{Hun96}). This fact and \cite[Theorem 3.1]{Day09} imply that the Hilbert polynomial of the union of these $27$ lines in $\mathbb{P}^3$ is $27m-108$. By the upper semi-continuity of the dimension of cohomology, if the Hilbert polynomials of the most degenerate limits of the $27$ lines parametrized by $\overline{\mathbf{N}}$ are also $27m-108$, then we would have that all the fibers of $\overline{\mathcal{B}}\rightarrow\overline{\mathbf{N}}$ have the same Hilbert polynomial. Hence, $\overline{\mathcal{B}}\rightarrow\overline{\mathbf{N}}$ is flat. Up to the choice of marking, there are exactly two maximal degenerations for the $27$ lines, which are parametrized by the zero-dimensional strata of $\overline{\mathbf{N}}$, which are of type $A_1^4$ and $(A_1^3,N)$.

The cubic surface parametrized by a $0$-stratum of type $(A_1^3,N)$ is $x_0x_1x_2=0$ in $\mathbb{P}^3$ and the limit of the $27$ lines is discussed in the proof of Proposition~\ref{narukifamilyisafamilyofstablepairs} (see also Table~\ref{ksbastabledegenerations27linesweight1/9+e}). We can compute the Hilbert polynomial of this configuration of lines directly with Macaulay2 using the following code:

\begin{flalign*}
\mathtt{i1:} 
\;\; & 
\mathtt{QQ[x_0,x_1,x_2, x_3]};  & &
\\
\mathtt{i2:}
\;\;  &
\mathtt{
I_A
=
intersect(ideal(x_0,x_3^4), ideal(x_1,x_3^4), 
ideal(x_2,x_3^4); 
}
\\ 
\mathtt{i3:}
\;\; &
\mathtt{
I_B
=
intersect(ideal(x_0,x_1+x_2-x_3), 
ideal(x_1,x_0 + x_2 - x_3), 
ideal(x_2,x_0 + x_1 -x_3); 
}
\\
\mathtt{i4:} 
\;\; &
\mathtt{
I_C
=
intersect( 
ideal(x_0,(x_1 - x_3)^2), 
ideal(x_1, (x_2 -x_3)^2),
ideal(x_2, (x_0-x_3)^2); 
}
\\
\mathtt{i5:}
\;\; &
\mathtt{
I_D
=
intersect(
ideal(x_0, (x_2 -x_3)^2),
ideal(x_1, (x_0-x_3)^2), 
ideal(x_2,(x_1 - x_3)^2); 
}
\\
\mathtt{i6:}
\;\; &
\mathtt{
I
=
intersect(I_A, I_B, I_C,I_D);
}
\\
\mathtt{i7:}
\;\; &
\mathtt{
hilbertPolynomial(I,Projective=>false)
}
\end{flalign*}

The cubic surface parametrized by a $0$-stratum of type $A_1^4$ is the so called Cayley cubic surface, which is the vanishing locus of the following equation:
\[
x_0x_1x_2+x_0x_1x_3+x_0x_2x_3+x_1x_2x_3=0.
\]
The above equation is obtained from the family $\overline{\mathcal{S}}\rightarrow\overline{\mathbf{N}}$ in Definition~\ref{def:FamilyCubic} by taking the limit for $\lambda,\mu,\nu\to0$ and $\rho\to1$ (see \cite[Proof of Lemma 11.4]{Nar82}). The corresponding limits of the $27$ lines (see the proof of Proposition~\ref{narukifamilyisafamilyofstablepairs}, part 1 (4)) can be computed explicitly from Table~\ref{tab:27lines}: the three lines of multiplicity one are given by the intersection of the Cayley cubic with the plane $x_0+x_1+x_2+x_3=0$, and the six lines of multiplicity four are $x_i=x_j=0$ for all $i,j\in\{0,\ldots,3\}$, $i\neq j$. Again, we compute the Hilbert polynomial of this configuration of lines using Macaulay2:

\begin{flalign*}
\mathtt{i1:} 
\;\; & 
\mathtt{QQ[x_0,x_1,x_2, x_3]};  & &
\\
\mathtt{i2:}
\;\;  &
\mathtt{
I_H
=
intersect(ideal(x_0^2),ideal(x_1^2),ideal(x_2^2),ideal(x_3^2));
}
\\ 
\mathtt{i3:}
\;\; &
\mathtt{
I_S
=
ideal(
x_0x_1x_2+x_0x_1x_3+x_0x_2x_3+x_1x_2x_3
);
}
\\
\mathtt{i4:} 
\;\; &
\mathtt{
I_D
=
I_S+ideal(x_0+x_1+x_2+x_3);
}
\\
\mathtt{i5:}
\;\; &
\mathtt{
I
=
intersect(I_S + I_H, I_D);
}
\\
\mathtt{i6:}
\;\; &
\mathtt{
hilbertPolynomial(I,Projective=>false)
}
\end{flalign*}

In both cases the above calculations yield $27m-108$. So our claim follows.
\end{proof}


\subsection{Proof of Theorem~\ref{thm:NarukiModular}}

\begin{theorem}
The Naruki compactification $\overline{\mathbf{N}}$ is isomorphic to the normalization of the KSBA compactification $\overline{\mathbf{Y}}_{\frac{1}{9}+\epsilon}$.
\end{theorem}

\begin{proof}
We start with the following observation. By Definition~\ref{definitionksbaweight1/9+epsilon} we know that $\mathbf{Y}_\times$ is a dense open subset of $\overline{\mathbf{Y}}_{\frac{1}{9}+\epsilon}$. However, the family over $\mathbf{Y}_\times$ in Definition~\ref{definitionksbaweight1/9+epsilon} can be extended to $\mathbf{Y}$ by the proof of Proposition~\ref{narukifamilyisafamilyofstablepairs}, step $1$. So $\mathbf{Y}\subseteq\overline{\mathbf{Y}}_{\frac{1}{9}+\epsilon}$ is a dense open subset.

By Proposition~\ref{narukifamilyisafamilyofstablepairs}, Proposition~\ref{properflatfamilywithrelativedivisor}, and from the fact that $\overline{\mathbf{Y}}_{\frac{1}{9}+\epsilon}$ is a coarse moduli space, there exists an induced morphism $f\colon\overline{\mathbf{N}}\rightarrow\overline{\mathbf{Y}}_{\frac{1}{9}+\epsilon}$. Notice that $f$ is an isomorphism on the common open subset $\mathbf{Y}$ parametrizing smooth cubic surfaces. The fact that $\mathbf{Y}$ is a dense open subset of $\overline{\mathbf{N}}$ follows from the definition of $\overline{\mathbf{N}}$. 

Let us prove that $f$ is a finite morphism. To do this, we only have to check that its fibers are finite. Let $C\subseteq\overline{\mathbf{N}}$ be a complete curve and assume by contradiction that $C$ is contracted by $f$ to a point. Then $C\subseteq D$, where $D$ is a boundary divisor of $\overline{\mathbf{N}}$. There are two cases to analyze.

\begin{enumerate}

\item $D$ is of type $N$. Hence, $D\cong(\mathbb{P}^1)^3$. Since $f|_D$ contracts a curve, $f$ contracts the divisor $D$. In particular, one of the three $\mathbb{P}^1$ components is contracted to a point, and we may assume without loss of generality it is the third one. We show that we can find $([\lambda:1],[\mu:1],[\nu:1]),([\lambda:1],[\mu:1],[\nu':1])\in\mathbb{P}^1\times\mathbb{P}^1\times\mathbb{P}^1$ parametrizing non-isomorphic stable pairs, producing a contradiction.

Start by considering any $\lambda,\mu,\nu\in\mathbb{C}^*$. If $S_0=\{x_0x_1x_2=0\}$, let $B_0$ be the $27$ lines on $S_0$ in Table~\ref{limitlinesonthreeplanes} corresponding to $\lambda,\mu,\nu$. In $H_0$, the lines $\overline{B_1C_2},\overline{B_2C_1},\overline{B_2C_2},\overline{B_2C_3}$ intersect the line $\overline{B_1C_1}$ at the following four points respectively:
\[
[0:1:0:0],~[0:0:1:0],~[0:-1:1:0],~[0:-\nu:1:0].
\]
Let $\beta$ be the cross-ratio of these four ordered points. Now choose $\nu'\in\mathbb{C}^*$ such that the cross-ratio $\beta'$ of the four points on $\overline{B_1C_1}$
\[
[0:1:0:0],~[0:0:1:0],~[0:-1:1:0],~[0:-\nu':1:0]
\]
is different from $\beta$. Let $B_0'$ be the $27$ lines on $S_0$ in Table~\ref{limitlinesonthreeplanes} corresponding to $\lambda,\mu,\nu'$. Then the pairs $\left(S_0,\left(\frac{1}{9}+\epsilon\right)B_0\right),\left(S_0,\left(\frac{1}{9}+\epsilon\right)B_0'\right)$ are not isomorphic.

\item $D$ is of type $A_1$. Denote by $U\subseteq\overline{\mathbf{N}}$ the complement of the $40$ type $N$ divisors in $\overline{\mathbf{N}}$ (recall that the morphism $\overline{\mathbf{N}}\rightarrow\overline{\mathbf{Y}}_{\GIT}$ is an isomorphism on $U$). From the previous part, we have that $C$ cannot be contained in a type $N$ divisor, so $C\cap U$ is nonempty. Consider the composition:
\[
g\colon\overline{\mathbf{N}}\rightarrow\overline{\mathbf{Y}}_{\GIT}\rightarrow\mathbb{P}(H^0(\mathbb{P}^3,\mathcal{O}(3)))/\!/_{\mathcal{O}(1)}\SL_4,
\]
which is a finite map on $U$. So the image of $C\cap U$ under $g$ is still one dimensional, and it parametrizes non isomorphic cubic surfaces because $g(C\cap U)$ is contained in the quotient by $\SL_4$ of the stable locus. It follows that $C\cap U$ parametrizes infinitely many distinct isomorphism classes of cubic surfaces, which contradicts the fact that $f(C)$ is a point.

\end{enumerate}

Since $f$ is a finite birational morphism, we have that $\overline{\mathbf{N}}$ is isomorphic to the normalization of $\overline{\mathbf{Y}}_{\frac{1}{9}+\epsilon}$ by Zariski's Main Theorem.
\end{proof}


\section{Extension criteria}

The following lemma allows us to compare the toroidal compactification of a ball quotient with another compactification which in our applications has geometric origin.
\begin{lemma}
\label{extensiontotoroidalwithnccompactificationandfreeactionandhyperplanearrangement}
Let $\Gamma\backslash\mathbb{B}$ be a ball quotient of dimension larger than or equal to $2$ and let $\mathcal{H}$ be a countable union of hyperplanes on $\mathbb{B}$ such that $\Gamma$ acts freely on $\mathbb{B}\setminus\mathcal{H}$. Let $\overline{M}$ be a smooth normal crossing compactification of a finite cover of $\Gamma\backslash(\mathbb{B}\setminus\mathcal{H})$ admitting a generically finite morphism $\phi\colon\overline{M}\rightarrow\overline{\Gamma\setminus\mathbb{B}}^{\bb}$.
Then there exists a morphism $\overline{M}\rightarrow\overline{\Gamma\backslash\mathbb{B}}^{\tor}$ inducing the following commutative diagram:
\begin{center}
\begin{tikzpicture}[>=angle 90]
\matrix(a)[matrix of math nodes,
row sep=2em, column sep=2em,
text height=2.5ex, text depth=0.25ex]
{\overline{M}&&\overline{\Gamma\backslash\mathbb{B}}^{\tor}\\
&\overline{\Gamma\backslash\mathbb{B}}^{\bb}.&\\};
\path[->] (a-1-1) edge node[]{}(a-1-3);
\path[->] (a-1-1) edge node[below left]{$\phi$}(a-2-2);
\path[->] (a-1-3) edge node[below right]{$\pi$}(a-2-2);
\end{tikzpicture}
\end{center}
\end{lemma}

\begin{proof}
Let $\Gamma_0 \trianglelefteq \Gamma$ be a neat subgroup. Consider the covering maps $\beta\colon
\overline{\Gamma_0\backslash\mathbb{B}}^{\bb}
\rightarrow\overline{\Gamma\backslash\mathbb{B}}^{\bb}$ and $\tau\colon\overline{\Gamma_0\backslash\mathbb{B}}^{\tor}\rightarrow\overline{\Gamma\backslash\mathbb{B}}^{\tor}$. Let $\overline{M}_{0}$ be the normalization of $\overline{M}$ in the function field of $\Gamma_0\backslash\mathbb{B}$, and denote by $\mu$ the morphism $\overline{M}_{0}\rightarrow\overline{M}$. Consider the rational map $g\colon\overline{M}\dashrightarrow\overline{\Gamma\backslash\mathbb{B}}^{\tor}$, which restricts to a finite cover of $\Gamma\backslash(\mathbb{B}\setminus\mathcal{H})$. We have that $g$ lifts to a rational map $g_0\colon\overline{M}_0\dashrightarrow\overline{\Gamma_0\backslash\mathbb{B}}^{\tor}$ giving a finite cover of $\Gamma_0\backslash(\mathbb{B}\setminus\mathcal{H})$. The morphisms we described fit in the following commutative diagram:

\begin{center}
\begin{tikzpicture}[>=angle 90]
\matrix(a)[matrix of math nodes,
row sep=2em, column sep=2em,
text height=2.5ex, text depth=0.25ex]
{\overline{M}_0&&\overline{\Gamma_0\backslash\mathbb{B}}^{\tor}\\
&\overline{\Gamma_0\backslash\mathbb{B}}^{\bb}&\\
\overline{M}&&\overline{\Gamma\backslash\mathbb{B}}^{\tor}\\
&\overline{\Gamma\backslash\mathbb{B}}^{\bb}.&\\};
\path[dashed,->] (a-1-1) edge node[above]{$g_0$}(a-1-3);
\path[->] (a-1-1) edge node[below left]{$\phi'$}(a-2-2);
\path[->] (a-1-3) edge node[below right]{$\pi'$}(a-2-2);
\path[dashed,->] (a-3-1) edge node[above left]{$g$}(a-3-3);
\path[->] (a-3-1) edge node[below left]{$\phi$}(a-4-2);
\path[->] (a-3-3) edge node[below right]{$\pi$}(a-4-2);
\path[->] (a-1-1) edge node[left]{$\mu$}(a-3-1);
\path[->] (a-1-3) edge node[right]{$\tau$}(a-3-3);
\path[->] (a-2-2) edge node[above right]{$\beta$}(a-4-2);
\end{tikzpicture}
\end{center}
We want to show that $g_0$ extends to $\overline{M}_0$, which implies that $g$ extends to $\overline{M}$ by Lemma~\ref{lemma:Extension}. This is done in two steps.

\textbf{Step 1.} Preliminarily, we show that $\overline{M}_0$ has only algebraic abelian quotient singularities. This is true because, by hypothesis, $\Gamma$ acts freely on $\mathbb{B}\setminus\mathcal{H}$ and $\overline{M}$ is smooth. Therefore, the finite map $\mu\colon\overline{M}_0\rightarrow\overline{M}$ branches along a divisor contained in $\phi^{-1}(\overline{\Gamma\backslash\mathcal{H}}^{\bb})$. Since $\overline{M}$ is simple normal crossing compactification of $\phi^{-1}(\Gamma\backslash(\mathbb{B}\setminus\mathcal{H}))$ by hypothesis, the branch divisor is simple normal crossing as well. So $\overline{M}_0$ has only algebraic abelian quotient singularities by Lemma~\ref{sing:Cover}.

\textbf{Step 2.} We show that $g_0$ extends to $\overline{M}_0$. First of all, we have that $g_0'=(\pi')^{-1}\circ\phi'$ extends $g_0$ to $(\phi\circ\mu)^{-1}(\Gamma\backslash\mathbb{B})$ ($\pi'$ is a blow up of the zero cusps of $\overline{\Gamma_0\backslash\mathbb{B}}^{\tor}$). Now, assume by contradiction that there exists a point $x\in\overline{M}_0$ in the indeterminacy locus of $g_0'$. Recall that if $W\subseteq X\times Y$ is the closure of the graph of a rational map $f\colon X \dashrightarrow Y$ and $p,q$ are the projections from $W$ to $X$ and $Y$ respectively, the \emph{total transform of $x\in X$} is defined as 
\[
\In(x):=q(p^{-1}(x))\subseteq Y.
\]
In our specific case, $\In(x)$ is contained in $\overline{\Gamma_0\setminus \mathbb{B}}^{\tor}\setminus\left( \Gamma_0\backslash \mathbb{B} \right)$, which is an union of abelian varieties $A_i$ because $\Gamma_0$ is neat (here is where we used that $\dim(\Gamma\backslash\mathbb{B})\geq2$). By step $1$ we have that $\overline{M}_0$ has abelian quotient singularities, which are log terminal. Therefore, the pair $(\overline{M}_0,0)$ is dlt (see \cite[page 43, (3)]{Kol13}), which implies by Lemma~\ref{lemmaRationalCurves} that $\In(x) \subseteq A_i$ is covered by rational curves. This is impossible because abelian varieties do not contain rational curves. So the map $g_0'$ extends to $\overline{M}_0$, implying that $g$ extends to $\overline{M}$.
\end{proof}

\begin{lemma}[{\cite[Theorem 2.23]{Kol07}}]
\label{sing:Cover}
Let $g:X \to Y$ be a finite and dominant morphism from a normal variety $X$ to a smooth variety $Y$, all defined over $\mathbb{C}$. Assume that there is a simple normal crossing divisor $D \subseteq Y$ such that $g$ is smooth over $Y \setminus D$.  Then $X$ has algebraic abelian quotient singularities.
\end{lemma}

\begin{lemma}[{\cite[Proposition 3.16]{CMGHL15}}]
\label{lemma:Extension}
Let $S$ be a scheme. Let $X$, $Y$ be integral locally Noetherian $S$-schemes. Assume
that $X$ is normal, and $Y$ is separated and locally of finite type over $S$. Consider the following commutative
diagram of morphisms and rational maps of $S$-schemes:
\begin{center}
\begin{tikzpicture}[>=angle 90]
\matrix(a)[matrix of math nodes,
row sep=2em, column sep=4em,
text height=1.5ex, text depth=0.25ex]
{X'&Y'\\
X&Y,\\};
\path[dashed,->] (a-1-1) edge node[above]{$\phi'$}(a-1-2);
\path[->] (a-1-1) edge node[left]{$f$}(a-2-1);
\path[dashed,->] (a-2-1) edge node[above]{$\phi$}(a-2-2);
\path[->] (a-1-2) edge node[right]{$g$}(a-2-2);
\end{tikzpicture}
\end{center}
and assume that $f$ is a composition of smooth surjective and finite surjective morphisms, and $g$ is a
finite surjective morphism. 
Then $\phi$ extends to a morphism if and only if $\phi'$
extends to a morphism.
\end{lemma}

\begin{lemma}[{\cite[Corollary 1.7]{HM09}}]
\label{lemmaRationalCurves}
Let $f: X \dashrightarrow Y$ be a rational morphism of normal proper varieties such that 
$(X, \Delta)$ a dlt pair for
some effective divisor $\Delta$. Then, for each closed point $x \in X$, the total transform $\In(x) \subseteq Y$ is covered by rational curves.
\end{lemma}

\begin{lemma}
\label{extensionOpen}
For $i=1,2$, let $\overline{M}_i$ be a projective normal compactification of $M_i$ such that $M_i\subseteq\overline{M}_i$ is an open subset and the boundary $D_i:=\overline{M}_i\setminus M_i$ is a disjoint union of the same number of irreducible divisors for both $i$. If there exists a birational map $\varphi\colon\overline{M}_1\dashrightarrow\overline{M}_2$ inducing an isomorphism $M_1\cong M_2$ and such that $\varphi|_{D_1}\colon D_1\dashrightarrow D_2$ is dominant, then $\varphi$ restricts to an isomorphism away from closed subsets of codimension at least two.
\end{lemma}

\begin{proof}
By hypothesis, $\dim(D_1)=\dim(D_2)$ and the number of irreducible components of $D_i$ is the same for $i=1,2$, so the rational map $\varphi|_{D_1}\colon D_1\dashrightarrow D_2$ is generically finite. After possibly removing closed subsets contained in $D_i$ obtaining open subsets $V_i\subseteq\overline{M}_i$, we have that $\varphi|_{V_1}\colon V_1\rightarrow V_2$ is a finite morphism (the complements of $V_i$ in $\overline{M}_i$ have codimension at least two). Notice that $\varphi|_{V_1}\colon V_1\rightarrow V_2$ is projective because it is the restriction to $V_1$ of the resolution of indeterminacies of $\varphi\colon\overline{M}_1\dashrightarrow\overline{M}_2$, and being projective, using Hartshorne's definition, is stable under base-change \cite[Tag 01WF]{Sta19}. $V_2$ is normal because it is an open subset of a normal variety, so we have that $\varphi|_{V_1}\colon V_1\rightarrow V_2$ is an isomorphism by Zariski's Main Theorem.
\end{proof}


\bibliographystyle{plain}

\end{document}